\documentclass[10pt]{amsart}
\usepackage{amsthm, amsfonts, amssymb, color}
\usepackage{mathrsfs}
\usepackage{amsmath}
\usepackage{amstext, amsxtra}
\usepackage{txfonts}
\usepackage[colorlinks, linkcolor=black, citecolor=blue, pagebackref, hypertexnames=false]{hyperref}

\allowdisplaybreaks
\newtheorem{theorem}{Theorem}[section]
\newtheorem{proposition}[theorem]{Proposition}
\newtheorem{corollary}[theorem]{Corollary}
\newtheorem{lemma}[theorem]{Lemma}
\newtheorem{definition}[theorem]{Definition}

\newtheorem{remark}[theorem]{Remark}

\numberwithin{equation}{section}

\title[Time Decay Estimates and Strichartz Estimates ]
{Decay Estimates and Strichartz Estimates of Fourth-order Schr\"{o}dinger Operator }

\author{Hongliang Feng, \ Avy Soffer \ and Xiaohua Yao}

\address {Hongliang Feng, School of Mathematics and Statistics, Central China
Normal University, Wuhan, 430079, P. R. China}
\email{fenghongliangccnu@163.com}
\address{Avy Soffer, School of Mathematics and Statistics, Central China
Normal University, Wuhan, 430079, P. R. China\\
On leave from Rutgers University}
\email{soffer@math.rutgers.edu}
\address{Xiaohua Yao, Department of Mathematics and  Hubei Province Key Laboratory of Mathematical Physics,
 Central China Normal University, Wuhan, 430079, P.R. China}
\email{yaoxiaohua@mail.ccnu.edu.cn}
\date{\today}

\begin{document}
\maketitle
\begin{abstract}
We study time decay estimates of the fourth-order Schr\"{o}dinger operator $H=(-\Delta)^{2}+V(x)$ in $\mathbb{R}^{d}$ for $d=3$ and $d\geq5$.  We analyze the low energy and high energy behaviour of resolvent $R(H; z)$, and then  derive the Jensen-Kato dispersion decay estimate and local decay estimate for $e^{-itH}P_{ac}$ under suitable spectrum assumptions of $H$. Based on Jensen-Kato type decay estimate and local decay estimate, we obtain the $L^1\rightarrow L^{\infty}$ estimate of $e^{-itH}P_{ac}$ in $3$-dimension by Ginibre argument, and also establish the endpoint global Strichartz estimates of $e^{-itH}P_{ac}$ for $d\geq5$. Furthermore, using the local decay estimate and the Georgescu-Larenas-Soffer conjugate operator method, we prove the Jensen-Kato type decay estimates for some functions of $H$.
\end{abstract}

\tableofcontents

\section{Introduction}
\setcounter{equation}{0}

In this paper we consider the time decay estimates of the operator
$$ H=H_{0}+V(x),\,\, H_{0}=(-\Delta)^{2}$$
in $L^{2}(\mathbb{R}^{d})$ for $d=3$ and $d\geq5$, where $V(x)$ is a real valued function as a multiplication operator. In the sequel, we assume that $V(x)=O(|x|^{-\beta})$ for large $|x|$ with some $\beta>0$ ( the specific $\beta$ will be given in conclusions below ).

It is well known that the fourth-order Schr\"odinger equation was introduced by Karpman \cite{Karpman1, Karpman2} and Karpman and Shagalov \cite{Karpman-Shagalov} to take into account the role of small fourth-order dispersion terms in the propagation of intense laser beams in a bulk medium with Kerr nonlinearity.  The nonlinear beam equation, or fourth-order wave equation has been involved in the study of plate and beams, see  e.g. Love \cite{Love}, in  the study of interaction of water waves, see Bretherton \cite{Bre}, and in the study of the motion of a suspension bridge, see Lazer and MacKenna \cite{LM} and MacKenna and Walter \cite{MW1, MW2}. Recently, these fourth-order equations were considered in mathematics by many authors. For example,  Levandosky and Strauss had considered the stability and instability of fourth-order solitary waves \cite{L1}, the time decay estimates for fourth-order wave equations \cite{L2} and  \cite{Le-Strauss}. Moreover, the well-posedness and scattering problems of  nonlinear fourth-order Schr\"odinger equation have been further studied by many authors now, see e.g. Miao, Xu and Zhao \cite{MXZ1, MXZ2}, Pausader \cite{Pau1, Pau2}, C. Hao, L. Hsiao and B. Wang \cite{HHW1, HHW2}, Ruzhansky, B. Wang and H. Zhang \cite{RWZ}, Segata \cite{Segata1, Segata2} and references therein.

In the studies of linear or nonlinear dispersive equations, one is faced with the need to quantitatively estimate the time decay of the solution in different kinds of norms. Indeed, many interesting estimates including {\it local decay estimates, Jensen-Kato type decay estimates, $L^{p}$-decay estimates and Strichartz estimates}, play central roles in these studies. Note that all the papers we mentioned above were concerned with fourth-order linear or nonlinear equations related to the operators $\Delta^{2}+\varepsilon\Delta$, $\varepsilon\in\{-1, 0, 1\}$.  The purpose of this paper is to establish such estimates for the fourth-order Schr\"odinger type operator (homogeneous case) $H=(-\Delta)^{2}+V$ with some decay potential. Our method also can treat the inhomogeneous cases $\Delta^{2}\pm\Delta$ and general $p(-\Delta)$ where $p$ is a polynomial.

Our work is partially motivated by Jensen and Kato's famous work \cite{JK}. They proved the time decay estimates of $e^{-it(-\Delta+V)}P_{ac}$ in the weighted $L^2$-norm. Precisely (with the assumption that zero is regular point of $-\Delta+V$),
  \begin{equation*}
   \|\langle x\rangle^{-\sigma}e^{-it(-\Delta+V)}P_{ac}\langle x\rangle^{-\sigma}\|_{L^{2}(\mathbb{R}^{3})\rightarrow L^{2}(\mathbb{R}^{3})}\leq c\langle t\rangle^{-3/2}.
\end{equation*}
Furthermore, Murata \cite{MM} had generalized Jensen and Kato's work to the operator $P(D)+V$, where $P(D)$ is an $m$-order elliptic differential operator with real constant coefficients, assuming that the all critical points of polynomial $P(\xi)$ are  non-degenerate, i.e.
$$\Big(\nabla P\Big)(\xi_0)=0, \ \ \ \ \ \  \det\Big(\partial_{i}\partial_{j}P(\xi)\Big)\big|_{\xi_0}\neq0.$$
However, the biharmonic  operator $(-\Delta)^2$  does not satisfy this assumption at $\xi=0$, thus  Murata's method does not apply for $H=(-\Delta)^2+V$.  Hence in this paper, we first establish Jensen-Kato type decay estimate and local decay estimate for $H$, which are very important, for example to asymptotic completeness of the perturbed linear fourth-order Schr\"odinger equations. Secondly, based on the Jensen-Kato type decay estimate and  local decay estimate, we prove $L^p$-type decay estimates and endpoint Strichartz estimates for $H$, which can be then applied to the well-posedness problems, scattering theory and soliton asymptotic stability problems of the nonlinear fourth-order Schr\"{o}dinger equation. Finally, we introduce the Georgescu-Larenas-Soffer conjugate operator method to derive the Jensen-Kato type estimate which starts only from the local decay estimate. Our methods differ from Murata, and apply to more general functions of the Laplacian,  including $(-\Delta)^m ( m\ge 2 )$ with the degenerate original point.

We notice that, for the biharmonic operator $(-\Delta)^2$, Ben-Artzi, Koch and Saut \cite{BKS} had proven the following sharp kernel estimate,
\begin{equation}\label{freeL1Linfityepsilon}
  |D^{\alpha}I_{0}(t,x)|\leq C|t|^{-(d+|\alpha|)/4}\ \Big(1+|t|^{-1/4}|x|\Big)^{(|\alpha|-d)/3},\,\, t\neq0,\,x\in\mathbb{R}^{d},
\end{equation}
where $I_{0}(t,x)$ is the kernel of $e^{-it\Delta^2}$, and similar pointwise (in time and space) estimates for $(-\Delta)^2\pm\Delta$. The above estimate implies the $L^{1}\rightarrow L^{\infty}$-estimate of $e^{-it\Delta^2}$, namely
\begin{equation}\label{freeL1Linfty}
  \|e^{-it\Delta^2}\|_{L^{1}(\mathbb{R}^{d})\rightarrow L^{\infty}(\mathbb{R}^{d})}\le C |t|^{-d/4}.
\end{equation}
Hence the endpoint Strichartz estimates for the free operator $(-\Delta)^2$ can be established, by using the $L^{1}\rightarrow L^{\infty}$-estimate \eqref{freeL1Linfty} and  Keel-Tao's arguments.  Besides, the Jensen-Kato type decay estimate, local decay estimate and  other $L^{p}$-decay estimates  of $e^{-itf(-\Delta)}$ for general operator $f(-\Delta)$ can be directly derived from the decay estimate similar to \eqref{freeL1Linfty}.

For the higher order Schr\"odinger operators $H_{f}:=f(-\Delta)+V$, it is much more difficult to establish  similar kernel estimate \eqref{freeL1Linfityepsilon} for $e^{-itH_{f}}$,  and  difficult to prove the $L^{1}\rightarrow L^{\infty}$-estimate similar to \eqref{freeL1Linfty}. In order to prove $L^{p}$-decay estimates and Strichartz estimates of $e^{-itH_{f}}$, we first prove Jensen-Kato type decay estimate and local decay estimate of $H_{f}$ to overcome difficulties due to the addition of potential $V$. A key point to obtain Jensen-Kato type decay estimate, is the asymptotic behaviour of the spectral density $E^{\prime}(\lambda)$ of $H_{f}$ near thresholds and infinity. In this paper, we focus on the case $H=(-\Delta)^2+V$. The only threshold of $H$ is zero. In the first part of the paper, we deduce asymptotic expansion in the weighted Sobolev spaces $\mathcal{H}_{\sigma}^{s}(\mathbb{R}^d)$ for resolvent $R(H; z)=(H-z)^{-1}$ and $E^{\prime}(\lambda)$ around zero for dimensions $d=3$ and $d\geq5$ assuming that {\it zero is a regular point for $H$}( see Definition \ref{resonace3} and \ref{resonace} below ).
Our strategy, is demonstrated by using the following free resolvent identity to get the asymptotic resolvent expansion for $(-\Delta)^{2}+V$:
\begin{equation}\label{freeresolventidentity}
  R(H_0; z)=(H_0-z)^{-1}=\frac{1}{2z^{1/2}}\big[(-\Delta-z^{1/2})^{-1}-(-\Delta+z^{1/2})^{-1}\big],\,z\in\mathbb{C}\setminus[0,\,+\infty).
\end{equation}
Similar formulas hold for general polynomials of $-\Delta$. See Remark \ref{polynomial}. Here and in other places, we denote the resolvent of $T$ by
$R(T; z)=(T-z)^{-1}$.

Since the leading term of the resolvent expansion depends on the dimension $d$, we deal with three cases separately: {\it  $d=3$,\,\, $d\geq5$ and odd,\,\  $d\geq6$ and even}. The following are three typical examples of our results, the expansions of $R(H; z)$ for $d=3,5,6$  as $|z|\rightarrow 0$ (with appropriate choice of weighted function $w$):
\begin{equation}\label{d=3}
 d=3,\,\, w(H-z)^{-1}w=C_0+z^{1/4}C_1+z^{2/4}C_2+\cdots
\end{equation}
\begin{equation}\label{d=5}
 d=5,\,\, R(H,z)=B_0+\frac{1-i}{2}z^{1/4}B_1+\frac{1+i}{2}z^{3/4}B_2+(-1)zB_3+\cdots
\end{equation}
\begin{equation}\label{d=6}
  d=6,\,\, R(H,z)=B^{0}_1+z^{1/2}B_{2}^{1}+z\ln z^{1/2}B_{3}^{1}+zB_{2}^{1,1}+\cdots
\end{equation}
where $z^{1/4}$ is in the first quadrant of complex plane.
The expansions are valid in the operator norm in $\mathcal{B}\big(\mathcal{H}^{-2}_{\sigma}(\mathbb{R}^d),\mathcal{H}^{2}_{\sigma^{\prime}}(\mathbb{R}^d)\big)$, where $\mathcal{H}^{s}_{\sigma}(\mathbb{R}^d)$ is the weighted Sobolev space with the associated norm
\begin{equation*}
  \|u\|_{\mathcal{H}^{s}_{\sigma}(\mathbb{R}^d)}=\|\langle x\rangle^{\sigma}\langle i\nabla\rangle^{s}u\|_{L^{2}(\mathbb{R}^d)}.
\end{equation*}
And $L^{2}_{\sigma}(\mathbb{R}^{d})$ denotes the space $\mathcal{H}^{s}_{\sigma}(\mathbb{R}^{d})$ when $s=0$, i.e. $L^{2}_{\sigma}(\mathbb{R}^{d})=\mathcal{H}^{0}_{\sigma}(\mathbb{R}^{d})$. Here $s,\sigma, \sigma^{\prime}\in\mathbb{R}$ and $\langle x\rangle=(1+|x|^2)^{1/2}$. In general,  the expansions to higher orders require larger $\beta$ and $\sigma, \sigma^{\prime}$.

In order to establish Jensen-Kato type decay estimate for $e^{-itH}$ and local smoothing, we also need to study the high energy decay properties of $R(H; z)$, see Subsection 2.3. In fact, the high energy decay estimate is easier than the low energy decay estimate. For Schr\"{o}dinger operator, for instance, in Kopylova and Komech \cite{KK} one can find the high energy decay of the free and perturbed resolvent in the weighted Sobolev norms in 3-dimension. For the constant coefficients differential operator $P(D)$ of order $m$ and of principal type, Agmon first established high energy decay estimate in the fundamental work \cite{Agmon}.  Moreover, Murata had also established high energy decay estimate for first order pseudo-differential operators \cite{MM1} and higher order elliptic operators \cite{MM2}. For the fourth-order Schr\"odinger operator $H$, our method is using the results of free resolvent $R(-\Delta; \zeta)$ and the resolvent identity \eqref{freeresolventidentity} to get the high energy decay estimates of $R((-\Delta)^2; z)$ directly, and then to get high energy derivative estimate of $R^{(k)}(H; z)$ for any $k\ge0$. Our decay rate of $R^{(k)}(H; z)$ is $-(3+3k)/4$, which is compatible with Agmon's result if $k=0$.

In this paper, we always assume that {\it thresholds are regular points of $H$} (so, in particular no bound states at threshold).
For the Schr\"{o}dinger operators, the fact about the absence of positive eigenvalue was first shown in Kato's work \cite{K} if the potential is continuous and decay $ O(|x|^{-\beta})$ at infinity for some $\beta>1$. Since then, the classical result has been extended to Schr\"odinger operators with rough integrable potentials by several authors ( see e. g. Jerison and Kenig \cite{JeKenig}, Kenig, Ruiz and Sogge \cite{KeRS}, Ionescu and Jerison \cite{IJ}, Koch and Tataru \cite{KoTa} and references therein ). Their basic strategy is proving the new Carleman estimate and unique continuation theorem, and then showing the absence of positive eigenvalues.
 %A series work of Kenig, Ruiz, Sogge and Jerison had generalized Kato's work and proved the Carleman estimate, uniformly Sobolev estimate and unique continuation for second-order constant coefficient differential operators.
The difficulty to follow these ideas for fourth-order Schr\"{o}dinger operator is to establish suitable Carleman type estimate for $(-\Delta)^2$ and suitable form of unique continuation theorem for $H$. At present, a general criterion about absence of positive eigenvalues in higher order cases is not yet available except that $V$ is a small potential, see e.g. \cite{SYY}.  Furthermore, we remark that Froese and Herbst's approach \cite{FH, FHHH} is more general than the works mentioned above, where they use the Mourre estimate of the Schr\"odinger operator and the positive preserving property of $e^{-t(-\Delta+V)}$. However,  for $H=(-\Delta)^2+V$, the positive preserving property of semigroup is an clear obstacle, even in the free case $e^{-t \Delta^2}$, see e.g. Reed and Simon \cite[Theorem XIII. 53]{RS2}.

In particular, it is expected that positive eigenvalues exist even for $C_{0}^{\infty}$-potentials. Some of our results are obtained with the assumption that $H$ has no positive embedded eigenvalues. But we must point out that we can remove the absence of positive embedded eigenvalues assumption by using Mourre theory \cite{Mourre, ABG, FH1, Georg-Gera-Moller} for $\bar{H}:=\bar{P}H\bar{P}$. $\bar{P}=1-P_{eign}$ where $P_{eign}$ denotes the orthogonal projection onto the span of eigenvector related to a positive eigenvalue. Due to the presence of the projection $\bar{P}$, the operator has purely continuous spectrum near the eigenvalue. Therefore, we will use Mourre thoery for energies in the continuous spectrum which contain a positive eigenvalue. See Section \ref{commutator method}. Furthermore, Ben-Artzi and Nemirovsky \cite{Ben-Nem} have established the limiting absorption principle near the threshold for general Schr\"odinger type operators $f(-\Delta)+V$ with short range potential. Here $f$ is a real-valued nonnegative continuous function with $f^{\prime}$ satisfies suitable estimates and H\"older continuity.

%For Schr\"{o}dinger operators, the time decay estimates of the corresponding non-stationary equations can be used to the scattering theory and is also important in the study of mapping properties of the propagator and the wave operators, see \cite{B, Schlag, JS} and reference therein.
%So does the resolvent expansion and the time decay of the fourth order Schr\"{o}dinger operator also are important to the corresponding non-stationary equations. Since we have analyzed the low energy asymptotic propertity and the high energy decay estimates, we obtain that $\langle x\rangle^{-\sigma}$ as a multiplication is $H$-smooth. So that $V$ is $H$-smooth if the potential $V$ is repulsive. The theory of smooth operators was developed by Kato. There are a number of equivalent forms of $H$-smooth, see section 7 chapter XIII of Simon's book \cite{RS2}. One important application of the smoothness theory is the Putnam-Kato Theorem. They prove that Hamiltonians with repulsive potentials have purely absolutely continuous spectrum. Further, we prove the Kato-Jensen decay estimates for $e^{-itH}$ by the spectral theorem. And the time decay rate in 3-dimensions is $-5/4$, in $d$-dimensions($d\geq5$) is $-d/4$ under our spectral assumptions of $H$.

Now we state one of our main results: {\it Jensen-Kato type decay estimate} ( see Section \ref{section3} for the local decay estimate of $e^{-it(\Delta^2+V)}$ ).
In the following discussion, all the constants $C$ are allowed to depend on the dimension $d$, and to vary from line to line.

\begin{theorem}\label{timedecay}
Let $H=(-\Delta)^2+V$ with $V(x)=O(|x|^{-\beta})$ for $|x|$ large and for some $\beta>1$ as detailed below. Assume $V$ is a compact operator from $\mathcal{H}^{2}_{0}$ to $\mathcal{H}^{-2}_{\beta}$. Under the assumption that $H$ has no positive embedded eigenvalues and 0 is a regular point for $H$, then the following conclusions hold:

(\text{i}) If $d=3$ and  $\beta>11+3/2$, then for  any $\sigma>2+1/2$ we have
\begin{equation}\label{JK3}
  \|e^{-itH}P_{ac} u\|_{L^{2}_{-\sigma}(\mathbb{R}^{3})}\leq C\langle t\rangle^{-5/4}\|u\|_{L^{2}_{\sigma}(\mathbb{R}^{3})},\,\,  t\in\mathbb{R};
\end{equation}

(\text{ii}) If $d\geq5$, $d$ odd  and  $\beta>d$, then for any $\sigma>d/2$ we have
\begin{equation}\label{JKd}
  \|e^{-itH}P_{ac}u\|_{L^{2}_{-\sigma}(\mathbb{R}^{d})}\leq C \langle t\rangle^{-d/4}\|u\|_{L^{2}_{\sigma}(\mathbb{R}^{d})}, \,\,  t\in\mathbb{R};
\end{equation}

(\text{iii}) If $d\geq6$, $d$ even  and $\beta>d+4$, then the above estimate \eqref{JKd} holds again for any $\sigma>d/2+2$.\\
Here $L^2_{\sigma}(\mathbb{R}^{d})$ is the weighted Sobolev space, and $P_{ac}$ denotes the projection onto the absolutely continuous spectrum space of $H$. The constants $C$ depend on the dimension $d$ only.
\end{theorem}

In the second part of this paper, we apply the $L^{p}$-estimate for the free case and Jensen-Kato type decay estimate above to derive the $L^p$-type estimate ( Ginibre argument ) and Strichartz estimates for $H$. The unpublished argument of Ginibre for Schr\"{o}dinger operator in three or higher dimensions allows passing from the local decay to global decay, in the form of  $L^{1}\cap L^{2}\rightarrow L^{2}+L^{\infty}$.  For Schr\"{o}dinger operator, such result is
\begin{equation*}
  \|e^{-it(-\Delta+V)}P_{ac}u\|_{L^{2}+L^{\infty}(\mathbb{R}^d)}\leq C(d)\langle  t\rangle^{-d/2}\|u\|_{L^{1}\cap L^{2}(\mathbb{R}^d)}.
\end{equation*}
Here $P_{ac}$ is the projection onto the continuous spectrum space of $-\Delta+V$, see e.g. W. Schlag \cite{Schlag} and references therein. The first optimal $L^p$-time decay estimate is due to Journ\'{e}, Soffer and Sogge's work \cite{JS}. They have established the $L^{1}\rightarrow L^{\infty}$-decay estimate using new cancellation lemma, see e.g. \cite[Lemma 2.2]{JS}. Then it was remarked by Ginibre(unpublished) that the above weaker estimate can be derived by simplified argument which requires Jensen-Kato type estimates.

For the biharmonic operator $(-\Delta)^2$, we know that the $L^{1}\rightarrow L^{\infty}$-decay estimate \eqref{freeL1Linfty} holds from Ben-Artzi, Koch and Saut's work \cite{BKS}. For the perturbed fourth-order Schr\"odinger operator $H=(-\Delta)^2+V$, there are few results about the $L^{1}\rightarrow L^{\infty}$ time decay estimate of $e^{-itH}P_{ac}$ in any dimension until now. The method in Journ\'{e}, Soffer and Sogge \cite{JS} can not be simply applied to $H=(-\Delta)^2+V$, as some similar cancellation lemmas involving $e^{-it\Delta^2} V e^{it\Delta^2}$ are much more complicated than Laplacian $-\Delta$. In this paper, based on the free decay estimate \eqref{freeL1Linfty}, we can use Ginibre argument for $e^{-itH}P_{ac}$ with decay $\langle t\rangle^{-d/4}$ for $d\geq5$, and also obtain the $L^{1}\rightarrow L^{\infty}$ time decay estimate of $e^{-itH}P_{ac}$ with decay $|t|^{-1/2}$ for $d=3$, though it is not optimal.
\begin{theorem}\label{Ginibred=3}
Let $d=3$ and  $H$ satisfy the same conditions as given in Theorem \ref{timedecay}. For $V(x)\in L^{\infty}$, then we have
\begin{align}\label{1inftyd=3}
\|e^{-itH}P_{ac}\|_{L^{1}(\mathbb{R}^3)\rightarrow L^{\infty}(\mathbb{R}^3)}\leq
\ \begin{cases}C \ |t|^{-3/4},\,\, \ \ for \,\, 0<|t|<1,\\ C\ |t|^{-1/2},\,\, \ \ for \,\,  |t|\ge1,\end{cases}
\end{align}
where $P_{ac}$ denotes the projection onto the absolutely continuous spectrum space of $H$.
\end{theorem}
\vskip0.3cm

\begin{theorem}\label{Ginibred5}
Let  $d\geq5$ and   $H$ satisfy the same conditions as given in Theorem \ref{timedecay}. For $V(x)\in L^{\infty}$, then we have
\begin{equation}\label{Ginid}
\|e^{-itH}P_{ac}u\|_{L^{2}+L^{\infty}(\mathbb{R}^d)}\leq C \ \langle t\rangle^{-d/4}\ \|u\|_{L^{2}\cap L^{1}(\mathbb{R}^d)}, \ t\in\mathbb{R},
\end{equation}
where $P_{ac}$ denotes the projection onto the absolutely continuous spectrum space of $H$.
\end{theorem}
The studies of space-time integrability properties of the solutions for Schr\"odinger equations and the corresponding inhomogeneous equation, have been pursued by many authors in the last thirty years. In particular, the Strichartz estimates, which have become fundamental and amazing tools for the studies of PDEs including the well-posedness and scattering theory, see e.g. \cite{Berezin-Shubin, KT, TT1, TT2, YK}. For the nonlinear fourth-order Schr\"odinger equation without potential, see \cite{MXZ1, MXZ2} for some global well-posedness and scattering results in both focusing and defocusing cases.
 Here we  will establish Strichartz estimates for the following fourth-order Schr\"odinger equation with potential $V$ and source term $h(t)$:
\begin{equation}\label{fourthorder equ}
\left\{ \begin{gathered}
   i\partial_t \Psi = (\Delta^2+V)\Psi+h(t), \hfill \\
   \Psi(0,\cdot)=\Psi_0 \in L^{2}(\mathbb{R}^d) . \hfill \\
\end{gathered}  \right.
\end{equation}
In the case of $d=3$, we get local Strichartz type estimate by interpolation. For $d\geq5$, based on Jensen-Kato type decay estimate \eqref{JKd} and local decay estimate \eqref{local decay}, we prove the following global endpoint Strichartz estimate. Recall that {\it the admissible pair $(q,r)$ for the fourth-order Schr\"odinger equation satisfies}
\begin{equation}\label{admissblepair}
  \frac{4}{q}+\frac{d}{r}=\frac{d}{2}, \,\, 2\leq q\leq\infty,\  d\ge5.
\end{equation}
Especially, $r=2d/(d-4)$ when $q=2$.
\begin{theorem}\label{SchtriE}
Consider the equation \eqref{fourthorder equ}. Let $H$ satisfy the same conditions as given in Theorem \ref{Ginibred5}. Then for any admissible pairs $(q,r)$ and $(\tilde{q},\tilde{r})$,  we have the homogeneous Strichartz estimate
\begin{equation}\label{homo}
  \|e^{-itH}P_{ac}\Psi_0\|_{L_{t}^{q}L_{x}^{r}(\mathbb{R}\times\mathbb{R}^{d})}\leq C(d)\, \|\Psi_0\|_{L^{2}(\mathbb{R}^{d})},
\end{equation}
and the dual homogeneous Strichartz estimate
\begin{equation}\label{dualhomo}
  \big\|\int_{\mathbb{R}}e^{isH}P_{ac}h(s,\cdot)ds\big\|_{L^{2}_{x}(\mathbb{R}^{d})}\leq C(d)\, \|h\|_{L_{t}^{\tilde{q}^{\prime}}L_{x}^{\tilde{r}^{\prime}}(\mathbb{R}\times\mathbb{R}^{d})}.
\end{equation}
Furthermore, the solution $\Psi(t,x)$ satisfies that
\begin{equation}\label{retarded}
  \|P_{ac}\Psi(t,x)\|_{L_{t}^{q}L_{x}^{r}(\mathbb{R}\times\mathbb{R}^{d})}\leq C(d)\, \|\Psi_0\|_{L^{2}}+\|h\|_{L_{t}^{\tilde{q}'}L_{x}^{\tilde{r}'}(\mathbb{R}\times\mathbb{R}^{d})},
\end{equation}
where $P_{ac}$ is the projection onto the absolutely continuous spectrum of $H$.
\end{theorem}

%For the beam equation, as an application of the resolvent asymptotic expansion, we establish Jensen-Kato time decay estimate and Ginibre argument for $e^{-it\sqrt{H+m^2}}$. Since $H_0+V$ has bounded states if $m^2$ is large enough then $H+m^2>0$, thus $\sqrt{H+m^2}$ is well defined by the spectral theory.  The dispersive behaviour of beam eqaution rely on some dispersion estimates. For the free case $e^{-it\sqrt{H_0+1}}$, W. Chen, C. Miao and X. Yao \cite{CMY} had considered a quite general operator $\phi(D)$, and given the pointwise dispersion estimates. They use the kernel to get the pointwise estimate, and their result implies the $L^p\rightarrow L^{p^{\prime}}$-estimates for $e^{-it\sqrt{H_0+1}}$. For the perturbed case, we prove the time decay rate of $e^{-it\sqrt{H+m^2}}P_{ac}$ in the weighted Sobolev norm are $-5/2(d=3)$ and $-d/2(d\geq5)$.

At the end of this paper, we use the Georgescu-Larenas-Soffer conjugate operator method to get Jensen-Kato type decay estimate for $e^{-itH}$ and $e^{-it\sqrt{H+m^2}}$. For the free half-wave operator $e^{-it\sqrt{H_{0}+1}}$, W. Chen, C. Miao and X. Yao \cite{CMY} had proven the $L^p\rightarrow L^{q}$-estimates using the kernel of $e^{-it\sqrt{H_{0}+1}}$. The conjugate operator method here we used reveals that the local decay estimate implies the Jensen-Kato type decay estimate. The idea of this method is to construct the Larenas-Soffer conjugate operator $\tilde{A}=A+B$, where $A=-\frac{i}{2}(x\cdot\nabla+\nabla\cdot x)$ and $B$ is a bounded operator. The commutator of $H$ and $A$ equals $q(H)+K$, where $q(H)$ is a function of $H$ and $K$ is a good operator in some sense. The new conjugate operator $\tilde{A}$  keeps the same good properties of $A$, and kills the tail $K$ when $\tilde{A}$ does commute with $H$. The difficulty of using this approach \cite{GLS} is to prove the $C^{k}$-condition,  i.e. $H\in C^{k}(A)$ and $\sqrt{H+m^2}\in C^{k}(A)$. We remark that this method relies on the local decay estimate to show the existence of operator $B$ and then allows to get stronger decay estimate in an easier way than the Jensen-Kato expansion method, see \cite{ABG, GG, GLS, LS}.

\textbf{Notations.} In what follows, we write $A\lesssim B$ to signify that there exists a constant $C$ such that $A\leq CB$. And $odd~ d\geq5$ means $d\geq5$ and $d$ is odd. similarly, for $even~d\geq6$.
%&&&&&&&&&&&&&&&&&&&&&&&&&&&&&&&&&&&&&&&&&&&&&&&&&&&&&&&&&&&&&&&&&&&&&&&&&&&&&mark
%&&&&&&&&&&&&&&&&&&&&&&&&&&&&&&&&&&&&&&&&&&&&&&&&&&&&&&&&&&&&&&&&&&&&&&&&&&&&&&

\section{The resolvent $R(H; z)$ of $H=(-\Delta)^2+V$}
\setcounter{equation}{0}
By the spectral theorem, we know
$$e^{-itH}P_{ac}=\int_{0}^{\infty}e^{-it\lambda}E^{\prime}(\lambda)d\lambda.$$
In order to get Jensen-Kato type decay estimate, we need to analyze the property of $E^{\prime}(\lambda)$ and higher order derivatives  $E^{(k)}(\lambda)$ for $\lambda$ small and large. This is related to the asymptotic properties of the resolvent $R(H; z)$ for $z$ small and large since $\pi E^{\prime}(\lambda)=\textrm{Im}~R(H; \lambda+i0)$.

In this section, we aim to obtain the low energy asymptotic expansion and the high energy decay estimate of $R(H; z)$. For $z$ large, we prove the high energy decay of $R(H; z)$ in $\mathcal{B}\big(L^{2}_{\sigma}(\mathbb{R}^{d}),\,\, L^{2}_{-\sigma}(\mathbb{R}^{d})\big)$ directly. For $z$ near zero, we use the second resolvent formula
\begin{equation}\label{Borndecomposition}
R(H; z)=(1+R(H_0; z)V)^{-1}R(H_0; z)
\end{equation}
to derive the asymptotic expansion in the weighted Sobolev space. At the end, we show that for $\lambda>0$, the limit $\lim_{\epsilon\downarrow0}R(H; \lambda\pm i0)$ exists in $\mathcal{B}\big(L^{2}_{\sigma}(\mathbb{R}^{d}),\,\, L^{2}_{-\sigma}(\mathbb{R}^{d})\big)$ to obtain the asymptotic property of $E^{\prime}(\lambda)$ for $\lambda$ small and decay of $E^{\prime}(\lambda)$ for $\lambda$ large.

Recall that, for the Schr\"{o}dinger operator $-\Delta+V$, Kato and Jensen derived resolvent expansion of $(-\Delta+V-\zeta)^{-1}$ around zero in dimensions $d\ge3$ by using the second resolvent formula
\begin{equation}\label{Born}
  R(-\Delta+V; \zeta)=(1+R(-\Delta; \zeta)V)^{-1}R(-\Delta; \zeta).
\end{equation}
For the free resolvent $R(-\Delta; \zeta)$, they use the kernel $k(x,y;\zeta)$ of $(-\Delta-\zeta)^{-1}$ to derive the asymptotic expansion, where
\begin{equation}\label{lakernel}
  k(x,y;\zeta)=\frac{i}{4}\Big(\frac{\zeta^{1/2}}{2\pi|x-y|}\Big)^{d/2-1}H^{(1)}_{d/2-1}(\zeta^{1/2}|x-y|),\,\, \rm{Im}\, \zeta^{1/2}\geq0,
\end{equation}
 and $H^{(1)}_{d/2-1}$ is the first Hankel function. Thus, the key point is to get the expansion of $(1+R(-\Delta; \zeta)V)^{-1}$ near zero, see \cite{JK, J, J1}. For the cases $d\geq3$, since the free resolvent $R(-\Delta; \zeta)$ has no singularity at zero, the expansion of resolvent $(-\Delta+V-\zeta)^{-1}$ has no terms with negative power of $z$ nor $\ln z$ separately. While for $d=1$ there exists $\zeta^{-\frac{1}{2}}G_{-1}$ and for $d=2$ there exists $\ln\zeta G_{0,-1}$, the classical second resolvent formula in \cite{J1} couldn't be useful in dimensions $d=1,2$. So Jensen and Nenciu \cite{JN} developed a unified approach to deal with these cases.
%They rewrite the Born splitting in the symmetrized form \eqref{equ:57} then through some identities get \eqref{equ:58} and then through a iteration process to get the expansion of $M(\mu)^{-1}$.
%There are many works about the time asymptotic expansion of the Schr\"odinger operator, see Murata's work \cite{MM} and references therein.

For the fourth-order Schr\"odinger operator $H=(-\Delta)^2+V$, we know that the singularity of $R((-\Delta)^2; z)$ at $z=0$ in $d$-dimensions ( $d\geq3$ ) is the same as $R(-\Delta; \zeta)$ in $(d-2)$-dimensions. Therefore, in the case $d=3$, we can apply the unified approach of Jensen and Nenciu \cite{JN} for Schr\"odinger operators of 1 and 2-dimensions. For $d\geq5$, we will follow the original one \cite{J1}.

\subsection{Asymptotic expansion of free resolvent $R(H_0; z)$ near $z=0$}

For the asymptotic expansion of free resolvent, the direct way is expanding its kernel. Actually, one can calculate the kernel of $R(H_0; z)$ by dividing the integrand into two parts and using the Cauchy integral directly. For instance, the explicit kernel $K(x,y; z)$ of $R(H_0; z)$ in $d=3$ is as follows:
\begin{equation*}
   K(x,y; z)=\frac{1}{8\pi}\frac{e^{iz^{1/4}|x-y|}-e^{-z^{1/4}|x-y|}}{z^{1/2}|x-y|},\,\,\textrm{Re}~z^{1/4}>0,\, \textrm{Im}~z^{1/4}>0.
\end{equation*}
For other dimensions, one can also using the same dividing trick to give the kernel of $R(H_0; z)$ by Hankel function with $z\in\mathbb{C}\setminus [0, +\infty)$. Here, we make use of the asymptotic expansion of free resolvent $R(-\Delta; \zeta)$ of Schr\"odinger operator, and apply the resolvent splitting \eqref{freeresolventidentity} to get the asymptotic expansion of $R(H_0; z)$ with $z$ around zero.

Recall the asymptotic expansion of $R(-\Delta; \zeta)$ around zero, see e.g. \cite{JK, J, J1}, we have:
\begin{lemma}
For $\zeta\in \mathbb{C}\setminus[0, +\infty)$ and $~\textrm{Im}~\zeta^{1/2}>0$, we have the following formal expansion of the resolvent of free Laplacian $R(-\Delta; \zeta)$ as $\zeta\rightarrow 0$:
\begin{equation}\label{lap odd}
 R_{0}(-\Delta; \zeta)=\sum_{j=0}^{\infty}(i\zeta^{1/2})^{j}G_{j}^{odd}, \ \ \ odd\,\,d\geq3;
\end{equation}
\begin{equation}\label{lap even}
 R_{0}(-\Delta; \zeta)=\sum_{j=0}^{\infty}\sum_{k=0}^{1}\zeta^{j}(\ln \zeta)^{k}G_{j}^{k,even}, \ \ \ even\,\,d\geq6.
\end{equation}

For $j=0,1,2,\cdots$, $G_{j}^{odd}$ are operators given by the following integral kernels
\begin{equation}\label{Gjodd}
  G_{j}^{odd}(x,y)=\frac{(-1)^{(d-3)/2}}{2(2\pi)^{(d-1)/2}}d_{j}|x-y|^{j-(d-2)},
\end{equation}
with $d_{j}=\sum_{k=0,k\geq (d-3)/2-j}^{(d-3)/2}\frac{((d-3)/2+k)!}{k!((d-3)/2-k)!}\frac{(-2)^{-k}}{(k+j-(d-3)/2)!}$. Especially, $d_j=1$ when $d=3$.

For $j=0,1,2,\cdots,d/2-2$, $G_{j}^{k,even}$ are operators given by the following integral kernels
\begin{equation}\label{Gjeven0xiao}
  G_{j}^{0,even}(x,y)=\pi^{-d/2}\frac{(d/2-j-2)!}{j!}4^{-j-1}|x-y|^{2j+2-d},
\end{equation}
\begin{equation}\label{Gjeven1xiao}
  G_{j}^{1,even}(x,y)=0.
\end{equation}

For $j\geq d/2-1$, $G_{j}^{k,even}$ are operators given by the following integral kernels
\begin{equation}\label{Gjeven0da}
\begin{split}
  G_{j}^{0,even}(x,y)&=(4\pi)^{-d/2}[\vartheta(j+1)+\vartheta(j+2-d/2)]\frac{(-1/4)^{j+1-d/2}}{j!(d/2-1+j)!}|x-y|^{2j+2-d}\\
                     &-2(4\pi)^{-d/2}(-1/4)^{j+1-d/2}\frac{\ln(|x-y|/2)}{j!(d/2-1+j)!}|x-y|^{2j+2-d}\\
                     &+\frac{i}{4}(4\pi)^{-d/2+1}\frac{1}{j!(d/2-1+j)!}(-1/4)^{j+1-d/2}|x-y|^{2j+2-d}.
\end{split}
\end{equation}
\begin{equation}\label{Gjeven1da}
  G_{j}^{1,even}(x,y)=-(4\pi)^{-m/2}\frac{(-1/4)^{j+1-d/2}}{j!(d/2-1+j)!}|x-y|^{2j+2-d}.
\end{equation}
Here $\vartheta(k)$ is given by $\vartheta(1)=-1$, $\vartheta(k)=\sum_{j=1}^{k-1}1/j-\pounds$, and $\pounds$ is the Euler's constant.
\end{lemma}

\begin{remark}\label{rem2}
$d_j=0$ for $j$ odd and $0<j<d-2$, see \cite[Lemma 3.3]{J}. Further, all the $G_j^{odd},\, G_{j}^{k, even}\in\mathcal{B}(\mathcal{H}^{-2}_{\sigma},\mathcal{H}^{2}_{\sigma^{\prime}})$ with $\sigma, \sigma^{\prime}$ depend on $j$. The proof details and more properties of $G_j^{odd}$ and  $G_{j}^{k, even}$, please see \cite{JK, J}. One difference is that \cite{JK,J} proved $G_j^{odd},\, G_{j}^{k, even}\in\mathcal{B}(\mathcal{H}^{0}_{\sigma},\mathcal{H}^{0}_{\sigma^{\prime}})$ and then using the identity $$(1-\Delta)R(-\Delta; \zeta)=1+(1+\zeta)R(-\Delta; \zeta)$$ to improve into $\mathcal{B}(\mathcal{H}^{-1}_{\sigma},\mathcal{H}^{1}_{\sigma^{\prime}})$. Here for fourth-order Schr\"odinger operator, we can  improve  into $\mathcal{B}(\mathcal{H}^{-2}_{\sigma},\mathcal{H}^{2}_{\sigma^{\prime}})$ by the same way, since we have $$(1+\Delta^2)R(H_0; z)=1+(1+z)R(H_0; z).$$
\end{remark}

Based on the expansion of $R(-\Delta; \zeta)$ and the resolvent identity \eqref{freeresolventidentity}, we obtain the formal expansions of $R(H_0; z)$ directly. For simplifying the notation, we let $z=\mu^{4}$ and choose $\mu$ in the first quadrant of the complex plane i.e. $\textrm{Re}~\mu>0$ and $\textrm{Im}~\mu>0$. {\it Note that if $\mu$ in the first quadrant,  then $z\in\mathbb{C}\setminus[0, +\infty)$.}
\begin{lemma}
For $\mu$ in the first quadrant of the complex plane, we have the formal expansions of the resolvent of free fourth-order Schr\"odinger operator $R(H_0; z)$:
\begin{equation}\label{sqlap odd}
 R(H_0; \mu^4)=\sum_{j=0}^{\infty}\frac{i^j-(-1)^j}{2}\mu^{j-2}G_{j}^{odd},\,\, odd\,\,d\geq3;
\end{equation}
\begin{equation}\label{sqlap even}
 R(H_0; \mu^4)=\sum_{j=0}^{\infty}\sum_{k=0}^{1}\frac{1}{2}(\ln \mu^2)^k \mu^{2j-2}\tilde{G}_{j}^{k,even},\,\,even\,\,d\geq6.
\end{equation}
Here for all $j\in \mathbb{N}$, $\tilde{G}_{j}^{0,even}=\frac{1-(-1)^j}{2}G_{j}^{0,even}+\frac{(-1)^j i\pi}{2}G_{j}^{1,even}$ and $\tilde{G}_{j}^{1,even}=\frac{1-(-1)^j}{2}G_{j}^{1,even}$.
\end{lemma}

Now we give a strict meaning for the above formal expansions. The formal series \eqref{sqlap odd} and \eqref{sqlap even} are an asymptotic expansions for $\mu\rightarrow0$ in the following sense.
\begin{proposition}\label{freeasymp}
For $z\in\mathbb{C}\setminus [0, +\infty)$, we have in $\mathcal{B}\big(\mathcal{H}^{-2}_{\sigma}(\mathbb{R}^{d}),\mathcal{H}^{2}_{\sigma^{\prime}}(\mathbb{R}^{d})\big)$ the following asymptotic expansions as $z\rightarrow 0$.

(i) For $d\geq3$ and $d$ odd,
\begin{equation}\label{R_0odd}
 R(H_0; z)=\sum_{j=0}^{N}\frac{i^j-(-1)^j}{2}z^{(j-2)/4}G_{j}^{odd}+o(z^{(N-2)/4}),
\end{equation}
with $\sigma$ and $\sigma^{\prime}$ satisfy:

1) for $0\leq N\leq (d-3)/2$: $\sigma, \sigma^{\prime}>1/2$ and $\sigma+\sigma^{\prime}>(d+1)/2$;

2) for $(d-3)/2<N\leq d-3$: $\sigma, \sigma^{\prime}>N+2-d/2$ and $\sigma+\sigma^{\prime}>N+2$;

3) for $N\geq d-2$: $\sigma, \sigma^{\prime}>N+2-d/2$.

(ii) For $d\geq6$ and $d$ even,
\begin{equation}\label{R_0even}
 R(H_0; z)=\sum_{j=0}^{N}\sum_{k=0}^{1}\frac{1}{2}(\ln z^{1/2})^k z^{(j-1)/2}\tilde{G}_{j}^{k,even}+o(z^{\frac{N-1}{2}}\ln z^{1/2}),
\end{equation}
with $\sigma$ and $\sigma^{\prime}$ satisfy:

1) for $0\leq N\leq (d-3)/4$: $\sigma, \sigma^{\prime}>1/2$ and $\sigma+\sigma^{\prime}>(d+1)/2$;

2) for $(d-3)/4<N<d/2-1$: $\sigma, \sigma^{\prime}>2N+2-d/2$ and $\sigma+\sigma^{\prime}>2N+2$;

3) for $N\geq d/2-1$: $\sigma, \sigma^{\prime}>2N+2-d/2$.

Furthermore, the expansion can be differentiated in $z$ any number of times. More precisely, the r-th derivative of the finite series in \eqref{R_0odd} ( Res. \eqref{R_0even} ) up to $j=N$,  is equal to $(d/dz)^{r}R(H_0; z)$ up to an error $o(z^{(N-2)/4-r})$ ( Res. $o( z^{\frac{N-1}{2}-r}\ln z^{1/2} )$ ) in the norm of $\mathcal{B}\big(\mathcal{H}^{-2}_{\sigma}(\mathbb{R}^{d}),\,\, \mathcal{H}^{2}_{\sigma^{\prime}}(\mathbb{R}^{d})\big)$ with $\sigma,\, \sigma^{\prime}$ satisfy the relationships with $N$ as given in the above Lemma \ref{freeasymp}.
\end{proposition}
\begin{proof}
The proof details please see \cite[Lemma 2.3]{JK} and \cite[Lemma 3.5, Lemma 3.9]{J}, since we derive the expansion around zero of $R(H_0;z)$ by the expansion of $R(-\Delta; \zeta)$.
\end{proof}

Notice that there are factors as $i^j-(-1)^j$ and $1-(-1)^j$ in the expansions of $R(H_0; z)$, so many terms can be cancelled. For $d=3$, the first two terms are $\frac{1}{4\pi}\mu^{-1}+(-G_{2}^{d=3})$ and $G_{2}^{d=3}(x,y)=1/(4\pi)|x-y|^{4-3}$, so $z=0$ is a singularity point of $R(H_0; z)$. For $d\geq5$, the lowest power of $\mu$ is positive since there are many terms equal zero and the zero term depends on the dimension $d$. For $odd \,\, d\geq5$, the first term is $$-G_{2}^{odd}(x,y)=\frac{(-1)^{(d-1)/2}}{2(2\pi)^{(d-1)/2}}d_2|x-y|^{4-d},$$ and for $even \,\, d\geq6$, the first term is $$G_{1}^{0,even}(x,y)=\frac{(d/2-3)!}{16\pi^{d/2}}|x-y|^{4-d}.$$ They are both the convolution kernel of the Riesz potential $(-\Delta)^{-2}$ in $\mathbb{R}^d$ respectively, see Stein \cite{Stein}. Further, \cite[Lemma 2.3]{J} implies that $(-\Delta)^{-2}\in \mathcal{B}(\mathcal{H}^{-2}_{\sigma},\mathcal{H}^{2}_{\sigma^{\prime}})$ with $d\geq9$ and $\sigma+\sigma^{\prime}\geq4$.

\subsection{Asymptotic expansion of $R(H; z)$ near $z=0$}

We deal with the expansion of $R(H; z)$ in $d=3$ and $d\geq5$ separately,  since zero is a singular point of $R(H_0; z)$ in 3-dimensions while $d\geq5$ not. For these two cases, we both use the second resolvent formula but different form. Before getting the expansion, we should analyze the zero threshold point of $H$.

In order to avoid stating some results separately for $d$ even and $d$ odd, we use the following notation:
\begin{equation*}
  G_0(x,y)=\begin{cases}-G_{2}^{odd}(x,y)=\frac{(-1)^{(d-1)/2}}{2(2\pi)^{(d-1)/2}}d_2|x-y|^{4-d},\,\,odd\,\,d\geq3;\\
                   G_{1}^{0,even}(x,y)=\frac{(d/2-3)!}{16\pi^{d/2}}|x-y|^{4-d},\,\,even\,\,d\geq6.
      \end{cases}
\end{equation*}

For the 3-dimensional case, we start from the symmetrized second resolvent formula
\begin{equation}\label{equ:57}
  (H-\mu^{4})^{-1}=(H_{0}-\mu^{4})^{-1}-(H_{0}-\mu^{4})^{-1}vM(\mu)^{-1}v(H_{0}-\mu^{4})^{-1},
\end{equation}
where
\begin{align}\label{equ:60}
  v(x)=|V(x)|^{\frac{1}{2}},\quad U(x)=\begin{cases}1,\, & V(x)\geq0,\\ -1,\, & V(x)<0.\end{cases}, \,\,M(\mu)=U+v(H_{0}-\mu^{4})^{-1}v.
\end{align}
Let $w(x)=U(x)v(x)$, then from the identity
\begin{equation}\label{RH}
  \Big(1-w(H-\mu^{4})^{-1}v\Big)\Big(1+w(H_{0}-\mu^{4})^{-1}v\Big)=1
\end{equation}
we have
\begin{equation}\label{equ:58}
  w(H-\mu^{4})^{-1}w=U-M(\mu)^{-1}.
\end{equation}
Now, the aim is to get the asymptotic expansion of $M(\mu)^{-1}$. Since
$$ M(\mu)=\frac{(1+i)\alpha}{8\pi}P\mu^{-1}+\sum_{j=0}^{p-1}M_{j}\mu^{j}+\mu^{p}\mathfrak{R}_{0}(\mu,|x-y|),$$
where $P=\alpha^{-1}\langle v, \cdot\rangle$ and $\alpha=\|v\|^{2}$, see Appendix A1.

The following definition is motivated by Jensen and Nenciu \cite{JN}, cf. the case of $S=0$ in their Theorem 5.2. See Schlag \cite[Definition 7]{Schlag-CMP}.
\begin{definition}\label{resonace3}
Let $Q=1-P$. We say that zero is a regular point of the spectrum of $H=(-\Delta)^{2}+V$ provided $Q(U+vG_{0}v)Q$ is invertible on $QL^{2}(\mathbb{R}^{3})$. In this  case $[Q(U+vG_{0}v)Q]^{-1}$ as an operator on $QL^{2}(\mathbb{R}^{3})$.
\end{definition}
Notice that, if we take $|v|$ as the weight function instead of the classic one $\langle x\rangle^{\sigma}$, then we
clearly, define the unusual weighted spaces $L^{2}_{v}(\mathbb{R}^{3})$ by
\begin{equation*}
  \|\psi\|_{L^{2}_{v}(\mathbb{R}^{3})}:=\|v(x)\psi(x)\|_{L^{2}(\mathbb{R}^{3})}<\infty.
\end{equation*}
 In order to get the expansion of $(H-\mu^{4})^{-1}$ in the unusual weighted space $L^{2}_{v}(\mathbb{R}^{3})$, it suffices to obtain the expansion of $M(\mu)^{-1}$ in $L^{2}(\mathbb{R}^{3})$.
Note that $M(\mu)$ has known expansion in powers of $\mu$ up to an order depending upon decay rate of $V$ at infinity, hence the problem is to prove that $M(\mu)^{-1}$ also has expansion in powers of $\mu$ up to some order and to compute the coefficients. Applying the unified approach, under our assumptions we have:
\begin{theorem}\label{resolventexpansion3}
 Let $\langle x\rangle^{\kappa} V(x)\in L^{2}(\mathbb{R}^3)$ for $\kappa$ large enough and $p$ be the largest integer satisfying $\kappa>2p+5$. Assume that 0 is a regular point for $H$. Then for $\mu$ in the first quadrant of the complex plane, there exists $\mu_0>0$ such that for $|\mu|\leq \mu_0$, $w(H-\mu^4)^{-1}w$ has the expansion in $\mathcal{B}\big(L^{2}(\mathbb{R}^3), \,L^{2}(\mathbb{R}^3)\big)$,
\begin{equation}\label{equ:63}
  wR(H; \mu^4)w=U-Qm_{0}^{-1}Q+\sum_{j=1}^{p-1}M^{\prime}_{j}\mu^{j}+\mu^{p}\mathfrak{R}(\mu).
\end{equation}
Here $\mathfrak{R}(\mu)$ is uniformly bounded and the coefficients $M^{\prime}_{j}$ can be computed explicitly.
\end{theorem}
For the proof, we refer the readers to see Appendix A1 of this paper. Here we need to point out that we first get the expansion of $U-M(\mu)^{-1}$ in $\mathcal{B}\big(L^{2}(\mathbb{R}^3), \,L^{2}(\mathbb{R}^3)\big)$, and then by identity \eqref{equ:58} get the asymptotic property of $(H-\mu^{4})^{-1}$ in the usual weighted Sobolev space with $v\langle x\rangle^{\sigma}\in L^{\infty}(\mathbb{R}^{3})$ using the following trick :
\begin{equation*}
  w\langle x\rangle^{\sigma}\langle x\rangle^{-\sigma}(H-\mu^{4})^{-1}\langle x\rangle^{-\sigma}\langle x\rangle^{\sigma}w=U-M(\mu)^{-1}.
\end{equation*}
Moreover, if potential $V$ satisfies $\langle x\rangle^{\beta}V(x)\in L^{\infty}(\mathbb{R}^{3})$ with $\beta>\kappa+3/2$, which implies the condition that $\langle x\rangle^{\kappa} V(x)\in L^{2}(\mathbb{R}^3)$.

For the cases of $d\geq5$, we use the Born expansion to finite order of the second resolvent formula \eqref{Borndecomposition} to get the expansion of $R(H; z)$. It's enough to get the expansion of $(1+R(H_0; z)V)^{-1}$.
Proposition \ref{freeasymp} and Remark \ref{rem2} imply that
\begin{equation*}
  1+R(H_0; z)V=1-G^{odd}_{2}V+\sum_{j=3}^{N}\frac{i^{j}-(-1)^{j}}{2}z^{(j-2)/4}G^{odd}_{j}V+o(z^{(N-2)/4}),\ \ \ \text{odd}\,\,d\geq5;
\end{equation*}
\begin{equation*}
 1+R(H_0; z)V=1+G_{1}^{0,even}V+\sum_{j=2}^{N}\sum_{k=0}^{1}\frac{1}{2}(\ln z^{1/2})^k z^{(j-1)/2}\tilde{G}_{j}^{k,even}V+o(z^{(N-1)/2} \ln z^{1/2}),\ \ \text{even}\,\, d\geq6.
\end{equation*}

Let $V(x)=O(|x|^{-\beta})$ for large $|x|$ with some $\beta>3$ and $V$ is compact from $\mathcal{H}^{2}_{0}(\mathbb{R}^{d})$ to $\mathcal{H}^{-2}_{\beta}(\mathbb{R}^{d})$. Note that as a multiplier, $V$ is also compact from $\mathcal{H}^{2}_{s}(\mathbb{R}^{d})$ to $\mathcal{H}^{-2}_{s+\beta}(\mathbb{R}^{d})$ for any $s\in\mathbb{R}$. Thus for $d\geq5$, $G_{0}V$ is compact for $4-d/2<s<\beta+d/2-4$ since $G_{0}\in \mathcal{B}\Big(L^{2}_{s}(\mathbb{R}^{d}); L^{2}_{-s}(\mathbb{R}^{d})\Big)$ with $s\geq4-d/2$. We define
\begin{equation*}
  \mathcal{N}=\Big\{~ f\in L^{2}_{-s}(\mathbb{R}^{d})~\big|~ (1+G_{0}V)f=0 ~\Big\},\,\,
  \mathcal{M}=\Big\{~ f\in L^{2}_{s}(\mathbb{R}^{d})~\big|~ (1+VG_{0})f=0 ~\Big\}.
\end{equation*}
A priori $\mathcal{N}$ and $\mathcal{M}$ may depend on $s$, but $\mathcal{N}$ and $\mathcal{M}$ are obviously monotone in $s$ in opposite directions, and also dual to each other, so $\dim\mathcal{N}=\dim\mathcal{M}<\infty$ shows that $\mathcal{N}$ and $\mathcal{M}$ are independent of $s\in (4-d/2,\,\, \beta+d/2-4)$.
\begin{definition}\label{resonace}
We say that a resonance occurs at zero of $H$, provided there is a distributional solution $u$ of the equation $(\Delta^2+V)u=0$, where for every $s>4-d/2$ one has $u\in L^{2}_{-s}(\mathbb{R}^{d})\setminus L^{2}(\mathbb{R}^{d})$.
\end{definition}
\begin{remark}
Following the same argument as in part of Jensen \cite{J}, one can prove that there is no zero energy resonance of $H=(-\Delta)^2+V$ for $d\geq9$. Actually, in dimensions $d\geq9$, \cite[Lemma 2.4]{J} shows that $G_{0}$ is nice in the sense that $G_{0}$ is bounded from $L^{2}(\mathbb{R}^{d})$ to $L^{2}_{4}(\mathbb{R}^{d})$. By the identity $R(H; z)(1+VR(H_{0}; z))=R(H_{0}; z)$, taking $z=0$ as with $d\geq9$, then $R(H; 0)(1+VG_{0})=G_{0}=(-\Delta)^{-2}$. As a consequence there is no zero resonance for $H$ in dimensions $d\geq9$ since $G_{0}$ is defined on $L^{2}(\mathbb{R}^{d})$.
\end{remark}

{\it The assumption that zero is neither an eigenvalue nor a resonance of $H$ imply that the function $(1+R(H_0; 0)V)u=0$ only admits zero solution}, here $R(H_0; 0)=(-\Delta)^{-2}$. Thus we can expand $(1+R(H_0; z)V)^{-1}$ ( Neumann series ) for small $z$ and then get the expansion of $R(H; z)$ near zero.

\begin{proposition}\label{1+RoV}
Assume that 0 is a regular point for $H$. We have in $\mathcal{B}\big(\mathcal{H}^{2}_{-\sigma}(\mathbb{R}^{d}), \mathcal{H}^{2}_{-\sigma}(\mathbb{R}^{d})\big)$ the expansion as $z\in \mathbb{C}\setminus \mathbb{R}$ and $z\rightarrow 0$.

(i) For $d\geq5$ and $d$ odd,
\begin{equation}\label{asymp1}
  (1+R(H_0; z)V)^{-1}=\sum_{j=0}^{N}\frac{(-1)^j-i^j}{2}z^{\frac{j-2}{4}}C_{j}+o(z^{\frac{N-2}{4}})
\end{equation}
with $\beta$ and $\sigma$ are assumed to satisfy:

1) for $0\leq N\leq (d-3)/2$: \,\, $\beta>(d+1)/2$ and $1/2<\sigma<\beta-1/2$;

2) for $(d-3)/2<N\leq d-3$: \,\, $\beta>N+2$, $N+2-d/2<\sigma<\beta-(N+2-d/2)$;

3) for $N\geq d-2$: \,\, $\beta>2N+4-d$ and $N+2-d/2<\sigma<\beta-(N+2-d/2)$.\par
All the $C_{j}$ can be calculated through the Newmann series. Especially, the first term of expansion \eqref{asymp1} is $(1-G_2^{odd}V)^{-1}$.

(ii) For $d\geq6$ and $d$ even,
\begin{equation}\label{asymp2}
  (1+R(H_0; z)V)^{-1}=\sum_{j=0}^{N}\sum_{k=0}^{\varsigma(j)}z^{\frac{j-1}{2}}(\ln z^{1/2})^{k}C_j^{k}+o\big(z^{\frac{j-1}{2}}(\ln z^{1/2})^{\varsigma(N)}\big)
\end{equation}
with $\beta$ and $\sigma$ are assumed to satisfy:

1) for $0\leq N\leq (d-3)/4$: \,\, $\beta>(d+1)/2$ and $1/2<\sigma<\beta-1/2$;

2) for $(d-3)/4<N< d/2-1$: \,\, $\beta>2N+2$, $2N+2-d/2<\sigma<\beta-(2N+2-d/2)$;

3) for $N\geq d/2-1$: \,\, $\beta>4N+4-d$ and $2N+2-d/2<\sigma<\beta-(2N+2-d/2)$.

Here $\varsigma(j)\in \{0,\, 1\}$ and $C_{j}^{k}$ can be computed by the Neumann series. Especially, the first term of expansion \eqref{asymp2} is $(1+G_{1}^{0,even}V)^{-1}$.
\end{proposition}
\vskip0.3cm
\begin{theorem}\label{resolventexpansion>5}
Under the same assumptions of Proposition \ref{1+RoV}, for $z\in \mathbb{C}\setminus \mathbb{R}$ and $z\rightarrow0$ we have:

(i) For $d\geq5$ and $d$ odd,
\begin{equation}\label{odd}
  R(H,z)=B_{0}+\sum_{j=3}^{N}\frac{i^j-(-1)^j}{2}z^{(j-2)/4}B_{j}+o(|z|^{(N-2)/4}).
\end{equation}
Here $B_0=(1-G_{2}^{odd}V)^{-1}(-G_{2}^{odd})$ and $B_{j}=0$ for $3\leq j<d-2$ and $d$ odd.

(ii) For $d\geq6$ and $d$ even,
\begin{equation}\label{even}
  R(H,z)=B^{0}_{1}+\sum_{j=1}^{N}\sum_{k=0}^{\rho(j)}(z^{1/2})^{j-1}(\ln z^{1/2})^{k}B^{k}_{j}+o\big(|z|^{(N-1)/2}(\ln z^{1/2})^{\rho(N)} \big),
\end{equation}
and $B^{0}_{1}=(1+G_{1}^{0, even}V)^{-1}G_{1}^{0, even}$.

All the other $B_j, B_{j}^{k}$ and $\rho(N)\in\{0,\, 1\}$ can be calculated explicit by the product of expansion series of $R(H_0; z)$ and $(1+R(H_0; z)V)^{-1}$.
\end{theorem}
\vskip0.3cm
\begin{remark}\label{r1}
For the 4-dimensions case, Jensen and Nenciu's unified approach can be applied to get the expansion of resolvent of the fourth-order Schr\"odinger operator near zero but it's very complicated even for Schr\"odinger operator. Notice that we deduce the expansion under the assumption that 0 is a regular point of $H$. If zero is a simple eigenvalue or resonance of $H$, one can also get the expansion by the second resolvent formula by the same argument as one for Schr\"odinger operator \cite{JK, J}.  However, it would be hard and interesting to further analyze the full structure of the zero eigenspace of the fourth-order Schr\"odinger operator.
\end{remark}
\vskip0.3cm
\begin{proposition}\label{daoshu}
The asymptotic expansions for $R(H; z)$ above can be differentiated in $z$ any times, in the sense that
\begin{equation}\label{daoshu0}
(\frac{d}{dz})^{r}\Big[wR(H; z)w-(U-Qm_{0}^{-1}Q+\sum_{j=1}^{p-1}M^{\prime}_{j}\mu^{j})\Big]=o(|z|^{p/4-r}), \ \ \ d=3;
\end{equation}
\begin{equation}\label{daoshu1}
(\frac{d}{dz})^{r}\Big[R(H; z)-(B_{0}+\sum_{j=3}^{N}\frac{i^j-(-1)^j}{2}z^{(j-2)/4}B_{j})\Big]=o(|z|^{(N-2)/4-r}), \ \ \ \text{odd}\,\,d\geq5;
\end{equation}
\begin{equation}\label{daoshu2}
(\frac{d}{dz})^{r}\Big[R(H; z)-(B^{0}_{1}+\sum_{j=1}^{N}\sum_{k=0}^{\rho(j)}z^{\frac{j-1}{2}}(\ln z^{1/2})^{k}B^{k}_{j})\Big]=o\big(|z|^{\frac{N-1}{2}-r}(\ln z^{1/2})^{\rho(N)} \big), \ \ \ \text{even}\,\,d\geq6.
\end{equation}
\end{proposition}
\begin{proof}
To see this, for $d\geq5$ note that $R(H; z)=(1+R(H_0; z)V)^{-1}R(H_0; z)$, in which $R(H_0; z)$ has differentiable asymptotic series ( Proposition \ref{freeasymp} ). Since the product of two differentiable asymptotic series is differentiable, it suffices to show that the asymptotic series for $(1+R(H_0; z)V)^{-1}$ is differentiable. This is seen from
\begin{equation*}
(d/dz)(1+R(H_0; z)V)^{-1}=-(1+R(H_0; z)V)^{-1}R^{\prime}(H_0; z)V(1+R(H_0; z)V)^{-1},
\end{equation*}
due to the result just mentioned about the product of two asymptotic series. For $d=3$, note that $w(H-\mu^{4})^{-1}w=U-M(\mu)^{-1}$ and
\begin{equation*}
(d/dz)[wR(H; z)w]=M^{-1}(\mu)vR^{\prime}(H_0; z)vM^{-1}(\mu).
\end{equation*}
Higher order derivatives can be done similarly. Note that we do the derivative in the topology of $\mathcal{B}\big(L^{2}_{\sigma}(\mathbb{R}^{d}),\,\, L^{2}_{-\sigma}(\mathbb{R}^{d})\big)$. In order to ensure that $(d/dz)^{r}R(H; z)\in\mathcal{B}\big(L^{2}_{\sigma}(\mathbb{R}^{d}),\,\, L^{2}_{-\sigma}(\mathbb{R}^{d})\big)$, $\sigma$ should be larger than the case $r=0$. Precisely, we need $\sigma>r+1/2$ which will be shown in the  Theorem \ref{continuity} below.
\end{proof}

\subsection{High energy decay estimates of $R(H; z)$}

The following result is the fundamental Agmon-Kato estimate on decay of the Schr\"{o}dinger resolvent operator for complex $z$ goes to infinity in the weighted Sobolev norms. It plays a crucial role in  time-decay estimates of the solution to Schr\"{o}dinger equation.
\begin{lemma}\label{freelap}
(\cite[Theorem 16.1]{KK}) For $\zeta\in\mathbb{C}\setminus[0,+\infty)$, $k=0,1,2,3,\cdots$, any $\sigma>k+\frac{1}{2}$ and any $a>0$, the bound
\begin{equation}\label{22}
  \|R^{(k)}(-\Delta; \zeta)\|_{\mathcal{H}_{\sigma}^{s}(\mathbb{R}^{3})\rightarrow\mathcal{H}_{-\sigma}^{s}(\mathbb{R}^{3})}\leq C(\sigma, a)|\zeta|^{-(1+k)/2},\,\,|\zeta|\geq a,
\end{equation}
holds for $s\in\mathbb{R}$.
\end{lemma}

We note that the proof of Theorem 16.1 in \cite{KK} does not depends on the dimension $d$, so that for free Schr\"odinger operator in $d$-dimensions, the estimate \eqref{22} also holds. The following is the similar conclusion for $H_0=(-\Delta)^2$. \vskip0.3cm
\begin{proposition}\label{freehighenergydecay}
For $z\in\mathbb{C}\setminus[0,+\infty)$, $k=0,1,2,3,\cdots$,  any $\sigma>k+\frac{1}{2}$ and any $a>0$, the bound
\begin{equation}\label{23}
  \|R^{(k)}(H_0; z)\|_{\mathcal{H}_{\sigma}^{s}(\mathbb{R}^{d})\rightarrow\mathcal{H}_{-\sigma}^{s}(\mathbb{R}^{d})}\leq C(\sigma,  a)|z|^{-(3+3k)/4},\,\,|z|\geq a
\end{equation}
holds for $s\in\mathbb{R}$.
\end{proposition}
\begin{proof}
First we prove decay \eqref{23} for $k=0$. We aim to prove that for $\sigma>\frac{1}{2}$,
\begin{equation}\label{24}
  \|R(H_0; z)\|_{\mathcal{H}_{\sigma}^{s}(\mathbb{R}^{d})\rightarrow\mathcal{H}_{-\sigma}^{s}(\mathbb{R}^{d})}\leq C(\sigma,   a)|z|^{-3/4},\,\,|z|>a,\,z\in\mathbb{C}\setminus[0,+\infty).
\end{equation}
By Lemma \ref{freelap}, we have
\begin{align*}
  \|R(H_0; z)\|_{\mathcal{H}_{\sigma}^{s}\rightarrow\mathcal{H}_{-\sigma}^{s}}
  & = \frac{1}{2|\mu^{2}|}\|R(-\Delta; \mu^{2})-R(-\Delta; (i\mu)^{2})\|_{\mathcal{H}_{\sigma}^{s}
         \rightarrow\mathcal{H}_{-\sigma}^{s}}\\
  & \leq \frac{1}{2|\mu^{2}|}\left(\|R(-\Delta; \mu^{2})\|_{\mathcal{H}_{\sigma}^{s}\rightarrow\mathcal{H}_{-\sigma}^{s}}
     +\|R(-\Delta; (i\mu)^{2})\|_{\mathcal{H}_{\sigma}^{s}\rightarrow\mathcal{H}_{-\sigma}^{s}}\right)\\
  & \leq C(\sigma, a)|z|^{-3/4}.
\end{align*}
Now we check decay estimate \eqref{23} for $k\neq0$. For $R(H_0; z)$ we have the recurrent relations
\begin{equation}\label{28}
  z R^{(k)}(H_0; z)=-k R^{(k-1)}(H_0; z)+\frac{1}{4}[x\cdot\nabla,\, R^{(k-1)}(H_0; z)].
\end{equation}
By an induction process we get the estimate \eqref{23}.
\end{proof}

Next, we prove the high energy decay estimate of perturbed resolvent $R^{(k)}(H; z)$. The proof relies on estimate \eqref{23} for the free resolvent $R(H_0; z)$ and on some useful identities for $R(H_0; z)$ and $R(H; z)$. Note that $V$ is $H_{0}$-relative bounded under the compactness assumption on $V$. Precisely, $V$ is a compact operator from $\mathcal{H}_{0}^{2}$ to $\mathcal{H}_{\beta}^{-2}$. Thus there exists a finite constant $V_0\in \mathbb{R}$, such that for any $\lambda\in \mathbb{R}\setminus [V_0, +\infty)$, $H-\lambda=\Delta^{2}+V-\lambda>0$, then for the resolvent set $\rho(H)$ of $H$, we have $\mathbb{C}\setminus [V_0, +\infty)\subseteq\rho(H)$.
\begin{lemma}\label{Fredholm}
Let $V(x)=O(|x|^{-\beta})$ for large $|x|$ with some $\beta>1$. Assume that $V$ is a compact operator from $\mathcal{H}_{0}^{2}$ to $\mathcal{H}_{\beta}^{-2}$. Then for $z\in \mathbb{C}\setminus[0, +\infty)$, $R(H_0; z)V,\,\, VR(H_0; z)\in\mathcal{B}(L^{2}(\mathbb{R}^{d}))$ are compact. And $1+R(H_0; z)V$, $1+V R(H_0; z)$ are invertible for $z\in\mathbb{C}\setminus[V_0, +\infty)$.
\end{lemma}
\begin{proof}
The proof relies on the Hermitian symmetry and the Fredholm theorem. The compactness of $R(H_0; z)V$ and $VR(H_0; z)$ follows from the boundedness of $R(H_0; z)$ and the compactness assumption on $V$. Then we use the Fredholm theorem to get the invertibility of $1+R(H_0; z)V$. The invertibility of $1+VR(H_0; z)$ follows by the duality. The proof details please see the Appendix A2.
\end{proof}
\begin{proposition}\label{Perturbed Resolvent}
Let $k=0,1,2,3,\cdots$, and assume that the potential $V(x)$ satisfies $V(x)=O(|x|^{-\beta})$ for large $|x|$ with $\beta>k+1$. Assume that $V$ is a compact operator from $\mathcal{H}_{0}^{2}$ to $\mathcal{H}_{\beta}^{-2}$. Then for large $z\in\mathbb{C}\setminus[V_0, +\infty)$ and any $\sigma>k+1/2$, the bound
\begin{equation}\label{29}
  \|R^{(k)}(H; z)\|_{\mathcal{H}_{\sigma}^{s}(\mathbb{R}^{d})\rightarrow\mathcal{H}_{-\sigma}^{s}(\mathbb{R}^{d})}
  \leq C(\sigma, k)\big(|z|^{-(3+3k)/4}\big)
\end{equation}
holds for any $s\in\mathbb{R}$.
\end{proposition}
\begin{proof}
By the second resolvent formula, we have
\begin{equation}\label{30}
  R(H; z)=[1+R(H_0; z)V]^{-1}R(H_0; z),\,\,
  R(H; z)=[1+VR(H_0; z)]^{-1}R(H_0; z).
\end{equation}
The identities \eqref{30} imply
\begin{equation}\label{31}
  [1+R(H_0; z)V]R(H; z)=R(H_0; z),\,\,
  R(H; z)[1+VR(H_0; z)]=R(H_0; z).
\end{equation}
Differentiating \eqref{31} $k$ times, we obtain
\begin{equation}\label{38}
\begin{split}
R^{(k)}(H; z)& =R(H_0; z)^{(k)}-\sum_{k_{1}+k_{2}=k-1}\frac{(k-1)!}{k_{1}!k_{2}!}R^{(k_{1})}(H; z)VR^{(k_{2}+1)}(H_0; z)\\
                 & -\sum_{k_{1}+k_{2}=k-1}\frac{(k-1)!}{k_{1}!k_{2}!}R^{(k_{1}+1)}(H_0; z)VR^{(k_{2})}(H; z)\\
                 & +\sum_{k_{1}+k_{2}+k_{3}=k-1}\frac{(k-1)!}{k_{1}!k_{2}!k_{3}!}R^{(k_{1})}(H; z)VR^{(k_{2}+1)}(H_0; z)
                 VR^{(k_{3})}(H; z).
\end{split}
\end{equation}

Note that, if $\sigma_1<\sigma_2$, then estimate \eqref{29} holds for $\sigma_1$ implies \eqref{29} holds for $\sigma_2$. This follows by $\|f_1\|_{L^{2}_{\sigma_1}}\leq\|f_1\|_{L^{2}_{\sigma_2}}$ and $\|f_2\|_{L^{2}_{-\sigma_2}}\leq\|f_2\|_{L^{2}_{-\sigma_1}}$.

For $k=0$, it is enough to prove estimate \eqref{29} for $\sigma\in(\frac{1}{2},\,\, \frac{\beta}{2}]$. By the second resolvent formula, \eqref{29} holds if the norm of the inverse operator $[1+R(H_0; z)V]^{-1}:\mathcal{H}^{s}_{-\sigma}\rightarrow\mathcal{H}^{s}_{-\sigma}$ is uniform bounded in $z$ for large $z\in\mathbb{C}\setminus[0, +\infty)$, since $\|R(H_0; z)\|_{\mathcal{H}^{s}_{\sigma}\rightarrow\mathcal{H}^{s}_{-\sigma}}\leq C(|z|^{-3/4})$ by decay estimate \eqref{23} with $k=0$. Now, we aim to show $[1+R(H_0; z)V]^{-1}:\mathcal{H}^{s}_{-\sigma}\rightarrow\mathcal{H}^{s}_{-\sigma}$ is uniformly bounded in $z$ for large $z\in\mathbb{C}\setminus[0, +\infty)$. It is equivalent to prove that, for large $z\in\mathbb{C}\setminus[0, +\infty)$,
\begin{equation}\label{inverse1}
  \|g\|_{\mathcal{H}^{s}_{-\sigma}}\leq C\|(1+R(H_0; z)V)g\|_{\mathcal{H}^{s}_{-\sigma}}.
\end{equation}
In fact, by the triangle inequality we have
\begin{equation}\label{inverse2}
\big|\|g\|_{\mathcal{H}^{s}_{-\sigma}}-\|R(H_0; z)Vg\|_{\mathcal{H}^{s}_{-\sigma}}\big|\leq \|(1+R(H_0; z)V)g\|_{\mathcal{H}^{s}_{-\sigma}}\leq \|g\|_{\mathcal{H}^{s}_{-\sigma}}+\|R(H_0; z)Vg\|_{\mathcal{H}^{s}_{-\sigma}}.
\end{equation}
Then for $\|R(H_0; z)Vg\|_{\mathcal{H}^{s}_{-\sigma}}$, by the decay estimate \eqref{23} we have
\begin{equation*}
  \|R(H_0; z)Vg\|_{\mathcal{H}^{s}_{-\sigma}}\leq C|z|^{-3/4}\|Vg\|_{\mathcal{H}^{s}_{\sigma}}\leq C|z|^{-3/4}\|g\|_{\mathcal{H}^{s}_{-\sigma}},
\end{equation*}
since $\sigma\in(1/2, \beta/2]$ and the compactness of $V$. Thus for $z$ large enough, we have
$$\|R(H_0; z)Vg\|_{\mathcal{H}^{s}_{-\sigma}}\leq \frac{1}{4}\|g\|_{\mathcal{H}^{s}_{-\sigma}},$$
hence \eqref{inverse1} holds by \eqref{inverse2}. Notice that the constant $C$ in \eqref{inverse1} do not depend on $z$.

We prove the estimates \eqref{29} for $k\geq1$ by induction. Namely, assume \eqref{29} holds for $R^{(j)}(H; z)$ with $j=0,1,2,\cdots,k-1$. Consider the second summand on the right hand side of \eqref{38}. Choosing $\sigma^{\prime}\in(k_{1}+\frac{1}{2},\,\, \beta-\frac{3}{2}-k_{2})$ ( it is possible since $\beta>k+1$ ), we obtain
\begin{equation*}
  \begin{split}
  &\quad \|R^{(k_{1})}(H; z)VR^{(k_{2}+1)}(H_0; z)\psi\|_{\mathcal{H}_{-\sigma}^{s}}\\
  & \leq C(|z|^{-(3+3k_{1})/4})\|VR^{(k_{2}+1)}(H_0; z)\psi\|_{\mathcal{H}_{\sigma^{\prime}}^{s}}\\
  & \leq C(|z|^{-(3+3k_{1})/4})\|R^{(k_{2}+1)}(H_0; z)\psi\|_{\mathcal{H}_{\sigma^{\prime}-\beta}^{s}}\\
  & \leq C(|z|^{-(6+3k)/4})\|\psi\|_{\mathcal{H}^{s}_{\sigma}}
  \end{split}
\end{equation*}
since $\beta-\sigma^{\prime}>k_{2}+1+\frac{1}{2}$.

The third summand can be estimated similarly by choosing
$\sigma^{\prime}\in(k_{1}+\frac{3}{2},\,\, \beta-\frac{3}{2}-k_{2})$.

Finally, consider the last summand. Taking $\sigma^{\prime}\in(k_{1}+k_{3}+\frac{1}{2},\,\, \beta-\frac{3}{2}-k_{2})$, we get
\begin{equation*}
 \begin{split}
&\quad \|R^{(k_{1})}(H; z)VR^{(k_{2}+1)}(H_0; z)VR^{(k_{3})}(H; z)\psi\|_{\mathcal{H}^{s}_{-\sigma}}\\
&\leq C(|z|^{-(3+3k_{1})/4})\|VR^{(k_{2}+1)}(H_0; z)VR^{(k_{3})}(H; z)\psi\|_{\mathcal{H}^{s}_{\sigma^{\prime}}}\\
&\leq C(|z|^{-(3+3k_{1})/4})\|R^{(k_{2}+1)}(H_0; z)VR^{(k_{3})}(H; z)\psi\|_{\mathcal{H}^{s}_{\sigma^{\prime}-\beta}}\\
&\leq C(|z|^{-(6+3k_{1}+3k_{2})/4})\|VR^{(k_{3})}(H ;z)\psi\|_{\mathcal{H}^{s}_{\beta-\sigma^{\prime}}}\\
&\leq C(|z|^{-(6+3k_{1}+3k_{2})/4})\|R^{(k_{3})}(H; z)\psi\|_{\mathcal{H}^{s}_{-\sigma^{\prime}}}\\
&\leq C(|z|^{-(6+3k)/4})\|\psi\|_{\mathcal{H}^{s}_{\sigma}}
 \end{split}
\end{equation*}
since $\sigma^{\prime}>k_{1}+\frac{1}{2}$,\, $\beta-\sigma^{\prime}>k_{2}+1+\frac{1}{2}$ and $\sigma^{\prime}>k_{3}+\frac{1}{2}$.
\end{proof}

\subsection{Limiting absorption principle of $R(H; z)$}\label{limitabsorp}
The limiting absorption principle was known in the diffraction theory for wave and Maxwell equations. It means the existence and continuity of the resolvent in the continuous spectrum. The continuity of the resolvent for Schr\"odinger operator in the weighted Sobolev norms was established by Agmon \cite{Agmon}. H\"ormander also considered such problem for the general real coefficient self-adjoint operator $P(D)$. See \cite[Charpter XIV]{H2}.  Here we prove the continuity of $R(H; z)$ up to the positive real line in order to get the behaviour of $E^{\prime}(\lambda)$ for $\lambda>0$.

Denote by $\mathbb{C}^{+}=\{\rm{Im}\,z>0\}$ the open upper half complex plane, and by $\mathbb{C}^{-}=\{\rm{Im}\,z<0\}$ the open lower half complex plane. Define $\Xi$ be the disjoint union of $\overline{\mathbb{C}^{+}}$ and $\overline{\mathbb{C}^{-}}$ with the identified points $z\leq0$. For the resolvent of free Schr\"odinger operator $R(-\Delta; \zeta)$, summarizing the limiting absorption principle results of \cite{JK, J, J1}, or see Ginibre and Moulin \cite{GinMou} and Kuroda \cite{Kuroda}, we have:
\begin{lemma}
( \cite[Theorem 8.1]{JK} ) Let $k=0,1,2,\cdots$. If $\sigma>k+1/2$, then $R^{(k)}(-\Delta; \zeta)\in \mathcal{B}\big(L^{2}_{\sigma}(\mathbb{R}^{d}),$ $ L^{2}_{-\sigma}(\mathbb{R}^{d})\big)$ is continuous in $\zeta \in\Xi\setminus\{0\}$. Further, the boundary value
\begin{equation*}
  R^{(k)}(-\Delta; \lambda\pm i0)=\lim_{\epsilon\downarrow0}R^{(k)}(-\Delta; \lambda\pm i\epsilon)\in \mathcal{B}\big(L^{2}_{\sigma}(\mathbb{R}^{d}),\,\, L^{2}_{-\sigma}(\mathbb{R}^{d})\big)
\end{equation*}
exists for any $\lambda\in(0, +\infty)$. The decay estimate \eqref{22} can be extended from $\zeta\in\mathbb{C}\setminus[0, +\infty)$ to $\zeta\in\Xi\setminus\{0\}$.
\end{lemma}
Notice that the function $1/\mu^{2}$ is analytic for $\mu\in \mathbb{C}\setminus \{0\}$, and the fact that the difference of two analytic functions is also analytic. Here we need to point out that $z\in \mathbb{C}\setminus[0, +\infty)$ implies that $\mu^{2}, (i\mu)^{2}\in \mathbb{C}\setminus[0, +\infty)$.  By the resolvent identity \eqref{freeresolventidentity}, for the resolvent of free fourth order Schr\"odinger operator $R(H_0; z)$,  we have:
\begin{corollary}\label{freehighabsorp}
Let $k=0,1,2,\cdots$. For $\sigma>k+1/2$, then $R^{(k)}(H_0; z)\in \mathcal{B}(L^{2}_{\sigma}(\mathbb{R}^{d}),\,\, L^{2}_{-\sigma}(\mathbb{R}^{d}))$
is continuous in $z\in \Xi\setminus \{0\}$. Further, the boundary value
\begin{equation*}
  R^{(k)}(H_0; \lambda\pm i0)=\lim_{\epsilon\downarrow0}R^{(k)}(H_0; \lambda\pm i\epsilon)\in \mathcal{B}\big(L^{2}_{\sigma}(\mathbb{R}^{d}),\,\, L^{2}_{-\sigma}(\mathbb{R}^{d})\big)
\end{equation*}
exists for any $\lambda\in(0,\, +\infty)$, and  the bound
\begin{equation}\label{freehigh}
   \|R^{(k)}(H_0; z)\|_{L_{\sigma}^{2}(\mathbb{R}^{d})\rightarrow L_{-\sigma}^{2}(\mathbb{R}^{d})}=O(|z|^{-(3+3k)/4})
\end{equation}
holds as $z\rightarrow \infty$ in $\Xi\setminus \{0\}$.
\end{corollary}

\begin{remark}
For general operator $f(-\Delta)$, M. Ben-Artzi and J. Nemirovsky \cite[Theorem 2A]{Ben-Nem} have established the limiting absorption results for the free resolvent of $f(-\Delta)$ (while without results about higher order derivative of resolvent) under the following assumption on $f$. $f(\theta)$ is a real-valued positive continuously differentiable function for $\theta\in (0, +\infty)$. Its derivative $f^{\prime}(\theta)$ is positive and locally H\"older continuous. For instance, $f(-\Delta)=\sqrt{1-\Delta}$.
\end{remark}
It is well known that, for $H_{0}=(-\Delta)^{2}$, $\sigma_{ac}(H_{0})=\sigma_{c}(H_{0})=[0,+\infty)$. For $H=H_{0}+V$, $\sigma_{c}(H)=[0, +\infty)$ for short range potential. However, positive embedded eigenvalues may occur for $H$ even with $C_{0}$-potential. Denote $\Sigma$ to be the discrete sequence of imbedded eigenvalues of $H$.
Next, we will prove that the boundary value $R(H; \lambda\pm i0)$ exists on $\lambda\in\sigma_{c}(H)\setminus(\Sigma\cup\{0\})$. If $V_0\geq0$, then the segment $[V_0,\, 0]=\{0\}$.

\begin{lemma}\label{inverseabsorp}
Let $V(x)=O(|x|^{-\beta})$ for large $|x|$ with some $\beta>1$. Assume that $V$ is a compact operator from $\mathcal{H}_{0}^{2}$ to $\mathcal{H}_{\beta}^{-2}$. Then for $\sigma>1/2$ and $\lambda>0$, $R(H_0; \lambda\pm i0)V: L^{2}_{-\sigma}(\mathbb{R}^{d})\rightarrow L^{2}_{-\sigma}(\mathbb{R}^{d})$ and $VR(H_0; \lambda\pm i0): L^{2}_{\sigma}(\mathbb{R}^{d})\rightarrow L^{2}_{\sigma}(\mathbb{R}^{d})$ are compact.
\end{lemma}

\begin{proof}
Indeed, the compactness of $VR(H_0; \lambda\pm i0)$ follows from Corollary \ref{freehighabsorp} and the relative compactness assumption on V. The compactness of $R(H_0; \lambda\pm i0)V$ follows by the duality argument.
%Indeed, the compactness of $R(H_0; \lambda\pm i0)V$ follows by the Sobolev compact embedding $\mathcal{H}^{4}_{-\sigma^{\prime}}(\mathbb{R}^{d})\rightarrow L^{2}_{-\sigma}(\mathbb{R}^{d})$ with $\sigma^{\prime}\in (1/2, \min\{\sigma,\,\,\beta-\sigma\})$, see \cite[Theorem 2.5]{KK}. Multiplication operator $V: L^{2}_{-\sigma}(\mathbb{R}^{d})\rightarrow L^{2}_{\sigma^{\prime}}(\mathbb{R}^{d})$ is continuous since $\sigma+\sigma^{\prime}<\beta$. For the case $k=0$ in Corollary \ref{freehighabsorp}, using the identity $(1+\Delta^2)R(H_0; z)=1+(1+z)R(H_0; z)$, we can improve $L^{2}_{-\sigma}(\mathbb{R}^{d})$ to $\mathcal{H}^{4}_{-\sigma}(\mathbb{R}^{d})$. Thus  $R(H_0; \lambda\pm i0): L^{2}_{\sigma^{\prime}}(\mathbb{R}^{d})\rightarrow \mathcal{H}^{4}_{-\sigma^{\prime}}(\mathbb{R}^{d})$ is also continuous by Corollary \ref{freehighabsorp}.
\end{proof}
Ben-Artzi and Nemirovsky \cite[Theorem 4A]{Ben-Nem} established the limiting absorption principle for $R(H_{f}; \lambda\pm i0)$ with $H_{f}=f(-\Delta)+V$. The following result is the special case of Ben-Artzi and Nemirovsky \cite[Theorem 4A]{Ben-Nem} for $f(-\Delta)=(-\Delta)^{2}$.
\begin{lemma}\label{limitabsorp large}
{(\cite[Theorem 4A]{Ben-Nem})} Let $V(x)=O(|x|^{-\beta})$ for large $|x|$ with some $\beta>1$. Assume that $V$ is a compact operator from $\mathcal{H}_{0}^{2}$ to $\mathcal{H}_{\beta}^{-2}$.  Then for $\sigma>1/2$,  $R(H; z)\in \mathcal{B}\big( L^{2}_{\sigma}(\mathbb{R}^{d}),\,\,  L^{2}_{-\sigma}(\mathbb{R}^{d})\big)$ is continuous in $z\in \Xi\setminus(\Sigma\cup\{0\})$. Further, the boundary value
\begin{equation*}
  R(H; \lambda\pm i0)=\lim_{\epsilon\downarrow0}R(H; \lambda\pm i\epsilon)\in \mathcal{B}\big(L^{2}_{\sigma}(\mathbb{R}^{d}), L^{2}_{-\sigma}(\mathbb{R}^{d})\big)
\end{equation*}
exists for $\lambda\in\sigma_{c}(H)\setminus(\Sigma\cup\{0\})$.
\end{lemma}

%Here, we give the proof of the case that $\Sigma$ is empty. The analyticity of $R(H; z)$ follows by \eqref{38}. By Corollary \ref{freehighabsorp} and second resolvent formula if
%\begin{equation*}
%  \big(1+R(H_0; \lambda\pm i\epsilon)V\big)^{-1}\rightarrow \big(1+R(H_0; \lambda\pm i0)V\big)^{-1},\,\,\, \epsilon\downarrow0,
%\end{equation*}
%in the norm of $\mathcal{B}(L^{2}_{\sigma}, L^{2}_{-\sigma})$. The convergence holds if and only if both limit operators $1+R(H_0; \lambda\pm i0)V: L^{2}_{-\sigma}\rightarrow L^{2}_{-\sigma}$ are invertible for $\lambda>0$. According to the Fredholm Theorem and Lemma \ref{inverseabsorp}, it suffices to prove that for $\psi\in L^{2}_{-\sigma}$ with $\sigma\in (1/2,\,\, \beta-1/2)$ the equations
%\begin{equation}\label{99}
%  [1+R(H_0; \lambda\pm i0)V]\psi=0
%\end{equation}
%admit only zero solution. Notice that \eqref{99} implies that $\psi=R(H_0; \lambda\pm i0)f$ with $f=-V\psi$. The assumption on $V$ with $\beta>1$ implies that $f\in L^{2}_{\sigma^{\prime}}$ where $\sigma^{\prime}=\beta-1/2$. So that $\psi$ is $\lambda$-incoming or outgoing. Then by \cite[Theorem 14.5.2]{H2}, for any $s\in\mathbb{R}$ we have
%$\int\langle x\rangle^{s}|\psi|^2dx<\infty,$
%which implies that $\psi\in L^{2}$. While $\psi$ satisfies that
%\begin{equation}\label{909}
%  (H-\lambda)\psi=(H_0-\lambda)(1+R(H_0; \lambda+ i0)V)\psi=0.
%\end{equation}
%Hence $\psi$ is an eigenfunction related to the positive eigenvalue $\lambda$ which is a contradiction to the assumption that the absence of positive embedded eigenvalues of $H$, thus $\psi=0$.

\begin{remark}\label{remove absence}
In case $\lambda>0$ is an embedded eigenvalue of $H$.  We can replace $H$ by $\bar{H}=\bar{P}H\bar{P}$ where $\bar{P}=1-P_{eign}$ and $P_{eign}$-the projection on the eigenspace(in $L^{2}$) related to an eigenvalue. The limiting absorption principle can be deduced using positive commutator estimates, see e.g. \cite{ABG}. One uses that, the Mourre estimate holds at positive energies for $\bar{P}H\bar{P}$ localized around the energy $\lambda$. To conclude that Mourre theory applies, one also requires that $[A, \bar{P}], [A, [A, \bar{P}]]$ are relatively $H_0$-compact. Recall that, due to the presence of the projection $\bar{P}$, the operator $\bar{H}$ has purely continuous spectrum near the eigenvalue $\lambda$. See \cite{Golenia, Mourre-W} and references therein.
\end{remark}

\begin{theorem}\label{continuity}
Let $k=0,1,2,3,\cdots$, and  $V(x)=O(|x|^{-\beta})$ for large $|x|$ with $\beta>k+1$. Assume $V$ is a compact operator from $\mathcal{H}_{0}^{2}$ to $\mathcal{H}_{\beta}^{-2}$. Then for any $\sigma>k+1/2$, $R^{(k)}(H; z)\in \mathcal{B}\big(L_{\sigma}^{2}(\mathbb{R}^{d}), \,\, L_{-\sigma}^{2}(\mathbb{R}^{d}) \big)$ is continuous for $z\in\Xi\setminus(\Sigma\cup\{0\})$. Further, the decay estimate \eqref{29} can be extended from $z\in\Xi\setminus[V_0, +\infty)$ to $z\in\Xi\setminus(\Sigma\cup\{0\})$, i.e. the bound
\begin{equation}\label{extendhighenergy}
  \|R^{(k)}(H; z)\|_{L_{\sigma}^{2}(\mathbb{R}^{d})\rightarrow L_{-\sigma}^{2}(\mathbb{R}^{d})}
  = O(|z|^{-(3+3k)/4})
\end{equation}
holds as $z\rightarrow \infty$ in $\Xi\setminus(\Sigma\cup\{0\})$.
\end{theorem}
\begin{proof}
Note that Lemma \ref{limitabsorp large} implies the case $k=0$ holds. Next we aim to show the case $k=1$. Other cases hold by the induction process.

By the iterated identity \eqref{38}, let $k=1$ we have
\begin{equation}\label{onederivative}
 \begin{split}
  R^{\prime}(H; z)&=R^{\prime}(H_0; z)-R(H; z)VR^{\prime}(H_0; z)\\
                  &\quad -R^{\prime}(H_0; z)VR(H; z)+R(H; z)VR^{\prime}(H_0; z)VR(H; z).
 \end{split}
\end{equation}
Thus by Theorem \ref{freehighabsorp} and Theorem \ref{limitabsorp large}, for $\lambda\in \mathbb{R}\setminus(\Sigma\cup\{0\})$, we know the limit
\begin{equation*}
  R^{\prime}(H; \lambda\pm i0)=\lim_{\epsilon\downarrow0}R^{\prime}(H; \lambda\pm i\epsilon)
\end{equation*}
exists in $\mathcal{B}\big(L^{2}_{\sigma}(\mathbb{R}^{d}),\,\, L^{2}_{-\sigma}(\mathbb{R}^{d})\big)$ with $\sigma>3/2$.

Now, for the case $k=1$, it is remaining to show
\begin{equation*}
\|R^{\prime}(H; \lambda\pm i0)\|_{L_{\sigma}^{2}(\mathbb{R}^{d})\rightarrow L_{-\sigma}^{2}(\mathbb{R}^{d})}
  = O(|\lambda|^{-(3+3)/4}), \,\,\text{as}\,\, \lambda\rightarrow +\infty.
\end{equation*}
Indeed, by the Born splitting of $R(H;z)$, we have
\begin{equation}
 \begin{split}
  R^{\prime}(H; z)&=[1+R(H_0; z)V]^{-1}R^{\prime}(H_0; z)\\
                  &\, -[1+R(H_0; z)V]^{-1}R^{\prime}(H_0; z)[1+R(H_0; z)V]^{-1}R(H_0; z).
 \end{split}
\end{equation}
Thus by Theorem \ref{freehighabsorp} and the proof of Theorem \ref{limitabsorp large}, we know that
\begin{equation}
\begin{split}
  R^{\prime}(H; \lambda\pm i0)
  &=[1+R(H_0; \lambda\pm i0)V]^{-1}R^{\prime}(H_0; \lambda\pm i0)\\
  &\, -[1+R(H_0; \lambda\pm i0)V]^{-1}R^{\prime}(H_{0};z)[1+R(H_0; \lambda\pm i0)V]^{-1}R(H_0; \lambda\pm i0).
\end{split}
\end{equation}
Following the same argument as in Proposition \ref{Perturbed Resolvent}, then $[1+R(H_0; \lambda\pm i0)V]^{-1}$ is uniform bounded in the norm of $\mathcal{B}\big(L^{2}_{-\sigma}(\mathbb{R}^{d}),\,\,  L^{2}_{-\sigma}(\mathbb{R}^{d})\big)$ with $\sigma>1/2$ for $\lambda>0$ large enough. To summarize, we have shown that estimate \eqref{extendhighenergy} holds in the case $k=1$.
\end{proof}

\noindent{\bf Behaviours of the spectral density $E^{\prime}(\lambda)$.}\quad
Since $\pi E^{\prime}(\lambda)= \textrm{Im}~ R(H; \lambda+i0)$, under the assumption that $H$ has no positive embedded eigenvalues and $0$ is a regular point, then for any $k=0,1,2,\cdots$,  by Theorem \ref{continuity} we know $E^{(k+1)}(\lambda)\in \mathcal{B}\big(L^{2}_{\sigma}(\mathbb{R}^{d}),\,\, L^{2}_{-\sigma}(\mathbb{R}^{d})\big)$ is continuous in $\lambda\in (0,\, +\infty)$ for suitable large $\sigma$. Furthermore, in order to obtain the Jensen-Kato type decay estimate ( see Theorem \ref{timedecay} ), it is necessary to know the endpoint behaviour of $E^{\prime}(\lambda)$ and high order derivatives $E^{(k+1)}(\lambda)$. Indeed, for small $\lambda$, the asymptotic expansion for $E^{\prime}(\lambda)=\frac{1}{\pi}\rm{Im} R(\lambda+i0)$ can be deduced immediately from those asymptotic expansion for $R(H; z)$ near zero  in the preceding section. For large $\lambda$, we make use of the high energy decay estimate and  continuity of $R^{(k)}(H; z)$.

\begin{proposition}\label{Elambda}
Let $H=(-\Delta)^2+V$ with $V(x)=O(|x|^{-\beta})$ for large $|x|$ with some $\beta>1$. Assume $V$ is a compact operator from $\mathcal{H}_{0}^{2}$ to $\mathcal{H}_{\beta}^{-2}$ and  $0$ is a regular point for $H$. Then the following conclusions hold in $\mathcal{B}\big(L^{2}_{\sigma}(\mathbb{R}^{d}),\, L^{2}_{-\sigma}(\mathbb{R}^{d})\big)$ as $\lambda\rightarrow0$:

(\text{i}) If $d=3$ and  $\beta>11+3/2$, then for any $\sigma>0$ we have
\begin{equation}\label{JK3density}
  E^{\prime}(\lambda)=\lambda^{1/4}+\lambda^{1/2}+o(\lambda^{1/2}),
\end{equation}
and the expansion can be differentiated 2 times.

(\text{ii}) If odd $d\geq5$  and  $\beta>d$, then for any $\sigma>d/2$ we have
\begin{equation}\label{JKddensity}
  E^{\prime}(\lambda)= \lambda^{d/4-1}+o(\lambda^{d/4-1}),
\end{equation}
and the expansion can be differentiated $d-3$ times.

(\text{iii}) If even $d\geq6$ and $\beta>d+4$, then for any $\sigma>d/2+2$ we have
\begin{equation}\label{JKdevendensity}
 E^{\prime}(\lambda)= \lambda^{d/4-1}+o(\lambda^{d/4}\ln \lambda^{\frac{1}{2}}),
\end{equation}
and the expansion can be differentiated $d/2-1$ times.

Here, the above asymptotic expansion of $E^{\prime}(\lambda)$ can be differentiated in $\lambda$ in the same sense as in Proposition \ref{daoshu}.
\end{proposition}

%{\color{red}\begin{remark}
%Here, we can remove the assumption of absence of positive embedded eigenvalues by replacing $H$ by $\bar{P}H\bar{P}$. See Remak \ref{remove absence}.
%\end{remark}}

\begin{proof}
Since $\pi E^{\prime}(\lambda)= \textrm{Im}~ R(H; \lambda+i0)$, then the continuity of $R^{(k)}(H; z)$ implies that $E^{\prime}(\lambda)$ is differentiable. Recall that there are many terms cancelled in the expansion of $R(H_0; z)$. In order to get an error term, we need to choose $p=3$ in expansion \eqref{equ:63}, thus the result \eqref{JK3density} holds. Similarly, \eqref{JKddensity} follows from Theorem \ref{resolventexpansion>5} with choosing $N=d-2$ and \eqref{JKdevendensity} follows from Theorem \ref{resolventexpansion>5} with choosing $N=d/2$ to get an error term. Note that for the even case, we choose $N=d/2$ not $d/2-1$, the purpose is that to get the error $o(\lambda^{d/4}\ln \lambda^{\frac{1}{2}})$.  The chosen relationship between $\beta$ and $\sigma$ can be seen from the expansion results of Proposition \ref{1+RoV} and Theorem \ref{resolventexpansion3}.

The differentiability is established by the same sense as what we did for the differentiability of $R(H; z)$ in Proposition \ref{daoshu}.
\end{proof}
\begin{proposition}\label{spectraldesitylarge}
For $\lambda\rightarrow +\infty$, let $k=0,1,2,3,\cdots$.  Let $V(x)=O(|x|^{-\beta})$ for large $|x|$ with some $\beta>1$. Assume $V$ is a compact operator from $\mathcal{H}_{0}^{2}$ to $\mathcal{H}_{\beta}^{-2}$. Then for any $\sigma>k+1/2$, we have in the norm of $\mathcal{B}\big(L^{2}_{\sigma}(\mathbb{R}^{d}),\,\, L^{2}_{-\sigma}(\mathbb{R}^{d})\big)$,
\begin{equation}\label{highderivative}
  E^{(k+1)}(\lambda)= O(\lambda^{-3(k+1)/4}), \,\,\lambda\rightarrow +\infty.
\end{equation}
\end{proposition}
\begin{proof}
For large $\lambda>0$, this proposition is a consequence of Theorem \ref{continuity} actually.
\end{proof}

\section{Local decay estimate and Jensen-Kato type decay estimate for $e^{-itH}$}\label{section3}
\setcounter{equation}{0}

In this section, we will prove the Jensen-Kato type decay estimate which depends on the asymptotic properties of $E^{\prime}(\lambda)$ for $\lambda$ small and large. Before doing this we give the following local decay estimate which is equivalent that $\langle x\rangle^{-\sigma}$ is {\it $H$-smooth}.  $H$-smooth theory has many deep connections with scattering theory and spectral analysis, especially for Schr\"odinger operator with repulsive potential. As for the more backgrounds of $H$-smooth theory, we refer the readers to see Reed and Simon \cite[P. 344, XIII.7]{RS2}. Notice that the local decay estimate shows the time-space integrability of the propagator $e^{-itH}$, which is also the key point of the Georgescu-Larenas-Soffer conjugate operator method.

\begin{theorem}\label{local decay thm}
Let  $H=(-\Delta)^2+V$ and $V(x)=O(|x|^{-\beta})$ for large $|x|$ with $\beta>11+{3\over 2}$ for $d=3$, and $\beta>(d+1)/2$ for $d\geq5$.  Assume $V$ is a compact operator from $\mathcal{H}_{0}^{2}$ to $\mathcal{H}_{\beta}^{-2}$. Then under the assumption that $H$ has no positive embedded eigenvalues and 0 is a regular point,  for any $\sigma>1/2$ and $\phi\in L^{2}(\mathbb{R}^{d})$, we have
\begin{equation}\label{local decay}
  \int_{\mathbb{R}}\|\langle x\rangle^{-\sigma}e^{-itH}P_{ac}\phi\|^{2}_{L^{2}(\mathbb{R}^{d})}dt\leq   C \|\phi\|^{2}_{L^{2}(\mathbb{R}^{d})},
\end{equation}
where $P_{ac}$ is the projection onto the absolutely continuous spectrum of $H$.
\end{theorem}
\begin{proof}
For $z\in \mathbb{C}\setminus [V_0,\, +\infty)$, by Proposition \ref{Perturbed Resolvent} we have as $z\rightarrow \infty$,
\begin{equation*}
  \|\langle x\rangle^{-\sigma}(H-z)^{-1}\langle x\rangle^{-\sigma}\|_{L^{2}(\mathbb{R}^{d})\rightarrow L^{2}(\mathbb{R}^{d})}=O(|z|^{-3/4}).
\end{equation*}
For $z\in \mathbb{C}\setminus [V_0,\, +\infty)$, by the low energy asymptotic of $R(H; z)$, we have as $z\rightarrow 0$
\begin{equation*}
  \|\langle x\rangle^{-\sigma}(H-z)^{-1}\langle x\rangle^{-\sigma}\|_{L^{2}(\mathbb{R}^{3})\rightarrow L^{2}(\mathbb{R}^{3})}=O(1),\,\, d=3;
\end{equation*}
\begin{equation*}
  \|\langle x\rangle^{-\sigma}(H-z)^{-1}\langle x\rangle^{-\sigma}\|_{L^{2}(\mathbb{R}^{d})\rightarrow L^{2}(\mathbb{R}^{d})}=O(1),\,\,  d\geq5.
\end{equation*}
Then the conclusion follows by Corollary of \cite[P.148]{RS2}.
\end{proof}

\begin{remark}
For a suitable class of function $f(\Delta)$ of $\Delta$, Ben-Artzi and Nemirovsky \cite{Ben-Nem} have proved a stronger estimate for $H_{f}:=f(-\Delta)+V$. However, they assume that $f^{\prime}(0)>0$. Furthermore, in case that $\lambda>0$ is a positive embedded eigenvalue of $H$, the Mourre theory \cite{Mourre} is applied to $\bar{H}=\bar{P}H\bar{P}$ to establish the local decay estimate around $\lambda$.
\end{remark}

Now we begin to prove the Jensen-Kato type decay estimate ( see Theorem \ref{timedecay} ).
%\vskip0.5cm
\begin{proof}[\bf Proof of Theorem \ref{timedecay} ( Jensen-Kato type decay estimate )]

The specific values $\beta$ and $\sigma$ depend on the needed expansion terms of the resolvent $R(H; z)$ around zero as shown in Proposition \ref{1+RoV} and  Theorem \ref{resolventexpansion3}.
Here we give the proof of 3-dimensional case. For $d\geq5$ it follows by the similar argument.

Now, it suffice to prove that for any $u, \tilde{u}\in P_{ac}L^{2}(\mathbb{R}^{3})\cap L^{2}_{\sigma}(\mathbb{R}^{3})$,
\begin{equation*}
  |\langle \tilde{u},\,\, e^{-itH}u\rangle|\leq C\langle t\rangle^{-5/4}\|\tilde{u}\|_{L^{2}_{\sigma}(\mathbb{R}^{3})}\|u\|_{L^{2}_{\sigma}(\mathbb{R}^{3})}.
\end{equation*}
For $u, \tilde{u}\in P_{ac}L^{2}(\mathbb{R}^{3})$, by the spectral theorem we have
\begin{equation*}
  \langle \tilde{u},\,\, e^{-itH}u\rangle=\int_{0}^{\infty}e^{-it\lambda}\langle \tilde{u},\,\,E^{\prime}(\lambda)u\rangle d\lambda
                                         =\int_{0}^{\infty}e^{-it\lambda}g(\lambda) d\lambda,
\end{equation*}
where $g(\lambda)=\langle \tilde{u},\,\,E^{\prime}(\lambda)u\rangle$. Notice that $g(\lambda)$ is smooth since $E^{(k+1)}(\lambda)$ is continuous in $\lambda\in(0,\, +\infty)$ for $k=0,1,2,\cdots$. Further, for any $\mathbb{N}\ni k\ge 0$, \, $g(\lambda)$ satisfies :
\begin{equation}\label{g}
  |g^{(k)}(\lambda)|=\Big|\Big\langle \langle x\rangle^{\sigma} \tilde{u},\,\,\langle x\rangle^{-\sigma}E^{(k+1)}(\lambda)\langle x\rangle^{-\sigma}\langle x\rangle^{\sigma}u\Big\rangle\Big|\leq \|E^{(k+1)}(\lambda)\|_{\mathcal{B}(L^{2}_{\sigma},\,L^{2}_{-\sigma})}\|\tilde{u}\|_{L^{2}_{\sigma}}\|u\|_{L^{2}_{\sigma}}.
\end{equation}

Let $\chi_{l}(\lambda)$ be a smooth cutoff function, i.e.
\begin{equation*}
  \chi_{l}\in C^{\infty}_{0}(\mathbb{R}),\,\,
\chi_{l}(\lambda)=\begin{cases}1,\, & |\lambda|\leq\frac{1}{2}; \\
                               0,\, & |\lambda|\geq1.
                  \end{cases}
\end{equation*}
Then $g(\lambda)=\chi_{l}(\lambda)g(\lambda)+(1-\chi_{l}(\lambda))g(\lambda)=g_{l}(\lambda)+g_{h}(\lambda)$. Thus
\begin{equation}\label{decom}
  \langle \tilde{u},\,\, e^{-itH}u\rangle=\int_{0}^{\infty}e^{-it\lambda}g_{l}(\lambda) d\lambda+\int_{0}^{\infty}e^{-it\lambda}g_{h}(\lambda) d\lambda.
\end{equation}

For the second integral of \eqref{decom}, it is the Fourier transform of $g_{h}(\lambda)$ clearly. Note that Proposition \ref{spectraldesitylarge} and  estimate \eqref{g} imply that for any positive integer $k$, $g_{h}^{(k)}(\lambda)\in L^{1}\big((0,\infty)\big)$. Then the Riemann-Lebesgue's lemma tells that:  for large $t$ we have,
\begin{equation*}
 \begin{split}
       & \big|\int_{0}^{\infty}e^{-it\lambda}g_{h}(\lambda) d\lambda \big|\leq |t|^{-k}\|g_{h}^{(k)}(\lambda)\|_{L^{1}}\\
  \leq & \,\, |t|^{-k}\int_{1/2}^{\infty}\|E^{(k+1)}(\lambda)\|_{\mathcal{B}(L^{2}_{\sigma},\,L^{2}_{-\sigma})}d\lambda\,\, \|\tilde{u}\|_{L^{2}_{\sigma}}\|u\|_{L^{2}_{\sigma}}\\
  \leq & \,\, C|t|^{-k}\|\tilde{u}\|_{L^{2}_{\sigma}}\|u\|_{L^{2}_{\sigma}}.
 \end{split}
\end{equation*}

For the first integral of \eqref{decom}, integration by parts, we obtain
\begin{equation*}
  \int_{0}^{\infty}e^{-it\lambda}g_{l}(\lambda) d\lambda=\int_{0}^{\infty}\frac{e^{-it\lambda}}{it}g^{\prime}_{l}(\lambda) d\lambda.
\end{equation*}
So that it remains to prove that
\begin{equation*}
  \int_{0}^{\infty}e^{-it\lambda}g^{\prime}_{l}(\lambda) d\lambda=O(t^{-1/4}),\,\, t\rightarrow \infty.
\end{equation*}
In fact, we have for large $t$
\begin{equation*}
\begin{split}
& \quad \big| \int_{0}^{\infty}e^{-it\lambda}g^{\prime}_{l}(\lambda) d\lambda \big|
 =\frac{1}{2} \big|\int_{0}^{\infty}e^{-it\lambda}\big(g^{\prime}_{l}(\lambda+\pi/t)-g^{\prime}_{l}(\lambda)\big) d\lambda\big|\\
&\leq \int_{0}^{\pi/t}\big|\big(g^{\prime}_{l}(\lambda+\pi/t)-g^{\prime}_{l}(\lambda)\big)\big| d\lambda+\int_{\pi/t}^{\infty}\big|\big(g^{\prime}_{l}(\lambda+\pi/t)-g^{\prime}_{l}(\lambda)\big)\big| d\lambda\\
&\leq 2\int_{0}^{\pi/t}\big|g^{\prime}_{l}(\lambda)\big| d\lambda+\int_{\pi/t}^{\infty}d\lambda\int_{\lambda}^{\lambda+\pi/t}|g^{\prime\prime}_{l}(\tilde{\lambda})|d\tilde{\lambda}=O(|t|^{-1/4})
\end{split}
\end{equation*}
by the asymptotic properties of $E^{\prime}(\lambda)$ as in Proposition \ref{Elambda}.
\end{proof}
\begin{corollary}
Let $\lambda_{j}$ be the negative eigenvalues of $H$ and $P_{j}$ be the associated eigen-projection. Then under the same spectral assumption as in Theorem \ref{timedecay}, we have in $\mathcal{B}\big(L^{2}_{\sigma}(\mathbb{R}^d), L^{2}_{-\sigma}(\mathbb{R}^d)\big)$ with $\sigma$ sufficient large:
\begin{equation*}
  e^{-itH}-\sum_{j=1}^{\sharp}e^{-it\lambda_{j}}-P_{0}=t^{-5/4}C_1+t^{-3/2}C_2+\cdots,\,\,d=3.
\end{equation*}
\begin{equation*}
  e^{-itH}-\sum_{j=1}^{\sharp}e^{-it\lambda_{j}}-P_{0}=t^{-d/4}B+\cdots,\,\,d\geq5.
\end{equation*}
Here $\sharp=\sharp\, \{~ \text{negative eigenvalues of}\,\, H  ~\}$. And $C_1, C_2$, $B$ can be calculate precisely from the expansion of $R(H_0; z)$ and $(1+R(H_0; z)V)^{-1}$.
\end{corollary}

\begin{remark}\label{polynomial}
Although Murata had considered a general class of elliptic operator $P(D)+V$, but for a degenerate operator $P(D)$ ( i.e. $P(\xi)$ has degenerate critical points ), like $(-\Delta)^{m}$ with $2\leq m\in\mathbb{Z}^{+}$, his approach does not work. In fact, our method could be applied to some higher order operators $p(-\Delta)$ where $p(x)$ is a real polynomial on $\mathbb{R}$,  such as $(-\Delta)^{2}\pm(-\Delta)$ and $(-\Delta)^{m}$. First of all, we need to get the asymptotic expansion of the resolvent $(p(-\Delta)-z)^{-1}$  at thresholds ( i.e. the  critical values of $p(|\xi|^2)$ ). Recall that the set of critical values of a function $f(\xi)$ is defined by,
\begin{equation*}
  \Lambda=\big\{\,f(\xi_{0})\,\mid \,|\nabla f(\xi)|_{\xi=\xi_0}=0,\,\, \xi_0\in\mathbb{R}^{n}\big\}.
\end{equation*}
It is well-known that the number $\sharp\Lambda$ of $\Lambda$ is finite for any elliptic polynomial $P(\xi)$ on $\mathbb{R}^{n}$, see e.g. Agmon \cite{Agmon}.

For the special case $p_{\pm}(-\Delta)=(-\Delta)^{2}\pm(-\Delta)$, the critical-values set $\Lambda_{p_{\pm}}$ of $p_{\pm}(-\Delta)$,  $\Lambda_{p_{-}}=\{-1/4, 0\}$ and $\Lambda_{p_{+}}=\{0\}$. Note that we can express the resolvent $(p_{\pm}(-\Delta)-z)^{-1}$ as
\begin{equation*}
 \begin{split}
  \big(p_{\pm}(-\Delta)-z\big)^{-1}=&\Big((-\Delta\pm1/2)^2-(1/4+z)\Big)^{-1}\\
                                   =&\frac{1}{2\omega}\Big[\big(-\Delta\pm\frac{1}{2}-\omega\big)^{-1}                  -\big(-\Delta\pm\frac{1}{2}+\omega\big)^{-1}\Big],
 \end{split}
\end{equation*}
where $\omega^2=1/4+z$.
%From the above resolvent splits, we have the formally identity
%\begin{align*}
%   \lim_{z\rightarrow-1/4}(p_{-}(-\Delta)-z)^{-1}=&\lim_{\epsilon\downarrow0}\frac{-1}{2(\epsilon+i\epsilon)}\Big[\big(-\Delta-(\frac{1}{2}+\epsilon+i\epsilon)\big)^{-1}                  -\big(-\Delta-(\frac{1}{2}-\epsilon-i\epsilon)\big)^{-1}\Big];\\
%   \lim_{z\rightarrow0}(p_{-}(-\Delta)-z)^{-1}=&-\lim_{\epsilon\downarrow0}\Big[\big(-\Delta-(\epsilon+i\epsilon)\big)^{-1}                  -\big(-\Delta-(1-\epsilon-i\epsilon)\big)^{-1}\Big];\\
%   \lim_{z\rightarrow0}(p_{+}(-\Delta)-z)^{-1}=&\lim_{\epsilon\downarrow0}\Big[\big(-\Delta-(\epsilon+i\epsilon)\big)^{-1}                  -\big(-\Delta+(1-\epsilon-i\epsilon)\big)^{-1}\Big].
%\end{align*}
%\begin{equation*}
% \begin{split}
%  (p_{\pm}(-\Delta)-z)^{-1}=&\frac{1}{2i Im(z+1/4)^{1/2}}\Big[\big(-\Delta\pm\frac{1}{2}-((\frac{1}{4}+a)^2+b^2)^{1/2}e^{i\theta}\big)^{-1}\\
%                          &-\big(-\Delta\pm\frac{1}{2}-((\frac{1}{4}+a)^2+b^2)^{1/2}e^{-i\theta}\big)^{-1}\Big].
% \end{split}
%\end{equation*}
Then for example, by the formula above, we can  deduce the asymptotic behaviour of $(p_{-}(-\Delta)-z)^{-1}$ around the degenerate threshold $z=-1/4$,  by using the asymptotic behaviour of $(-\Delta-w)^{-1}$ around $w=1/2$. Indeed, the Taylor expansion of the kernel of $(-\Delta-w)^{-1}$ at the point $w_{0}=1/2$ tells us the desired asymptotic behaviour around the threshold  $z=-1/4$ in suitable weighted Sobole space.

For $p(-\Delta)=(-\Delta)^{m}$, the critical-value is only $z=0$. The resolvent $\big((-\Delta)^m-z\big)^{-1}$ can be expressed as
\begin{equation*}
\big((-\Delta)^m-z\big)^{-1}=\frac{1}{mz}\sum_{k=0}^{m-1}z_k(-\Delta-z_k)^{-1}
\end{equation*}
where $z_k=|z|^{\frac{1}{m}}e^{i\frac{2k\pi}{m}}(k=0,1,2,\cdots,m-1)$ are the $k$-th roots of $z$, see e.g. \cite{HYZ}. Similar arguments also can be concluded for $(-\Delta)^{m}$.
%For general cases, the resolvent $\big(p(-\Delta)-z\big)^{-1}$ can be expressed as
%\begin{equation}\label{Lapn}
% \big(p(-\Delta)-z\big)^{-1}=\sum_{j}c_{j}(z)(-\Delta-z_j)^{-k_{j}},\,\,1\leq k_{j}\leq m.
%\end{equation}
%Note that for $z=p(\lambda)+i\epsilon \,\, (\lambda>0)$, then $z_{j}=z_{j}(\epsilon)$ satisfies that $\lim_{\epsilon\rightarrow0}z_{j}(\epsilon)=\lambda$.
According to these expressions and under some suitable assumptions, one would obtain Jensen-Kato type decay estimate of $p(-\Delta)+V$ by using the same strategy as the fourth-order operator $(-\Delta)^2 +V$.
\end{remark}

\section{$L^p$-type decay estimates------Ginibre argument}
\setcounter{equation}{0}
\subsection{$L^p$-boundedness of projection $P_{ac}$}
In this section, we apply the iterated Duhamel formula  to prove the $L^{1}\rightarrow L^{\infty}$ in 3-dimension and $L^{1}\cap L^{2}\rightarrow L^{2}+L^{\infty}$ in  $d\geq5$ for $e^{-itH}P_{ac}$.
Let us first recall that Duhamel formula
\begin{equation}\label{duhamel}
e^{-itH}P_{ac}=e^{-itH_{0}}P_{ac}+i\int_{0}^{t}e^{-i(t-s)H_{0}}VP_{ac}e^{-isH}ds,\ \  H=(-\Delta)^2 +V.
\end{equation}
From the above formula, we need to give the $L^p$-boundedness of $P_{ac}$. For this end, let us summarize the spectrum of $H$. Firstly, we have assumed the absence of embedding positive eigenvalues of $H$, and zero is not an eigenvalue nor a resonance. So we have that $\sigma_{c}(H)=\sigma_{ac}(H)=[0,\,\,+\infty)$. Secondly, Birman and Solomyak's results \cite{BS} implies that there are only finite many discrete negative eigenvalues of $H$.
%Recall that $\sharp=\sharp\,\,\{\text{negative eigenvalue of}\,\, H\}$, and the quantitatively estimate of $\sharp$, see \eqref{negative eigenvalue d} and \eqref{negative eigenvalue 3}.

Denote the number of eigenvalues of $H$ lying to the left of $\gamma$ ( counted according to their multiplicities) by $N(\gamma; H)$. The estimate of $N(\gamma; H)$ depends on the potential function. For the Schr\"odinger operator in $\mathbb{R}^3$, Birman and Schwinger bound in \cite{RS2} is the earliest estimate of $N(0; -\Delta+V)$ by
\begin{equation*}
   N(0; -\Delta+V)\leq (\frac{1}{4\pi})^2\int_{\mathbb{R}^3}\frac{|V(x)||V(y)|}{|x-y|}dxdy.
\end{equation*}
Later, Birman and Solomyak \cite{BS} discussed the operator $A_{l}(\alpha V)=(-\Delta)^{l}-\alpha V,\, \,l,\,\alpha>0$. They had given the estimate of $N(\gamma; A_{l}(\alpha V))$ and discussed the asymptotic property of $N(\gamma; A_{l}(\alpha V))$ as $\alpha$ goes to infinite. For the fourth-order Schr\"odinger operator $H=(-\Delta)^2+V$ in $\mathbb{R}^d$, It is known that
\begin{equation}\label{negative eigenvalue d}
 N(0; H)\leq C(d)\int_{\mathbb{R}^{d}} V_{-}^{d/4}dx, \,\,d\geq5.
\end{equation}
For $d=3$, according to Theorem 5.1 and Remarks in \cite{BS}, we have for any $\gamma$ positive,
\begin{equation}\label{negative eigenvalue 3}
  N(-\gamma; H)\leq C\int_{\mathbb{R}^{3}} V_{-}dx, \,\,d=3.
\end{equation}
Here $V_{-}$ denotes the negative part of $V$. So that $H$ has finitely many negative eigenvalues if the potential function decays fast enough.

On the decay of the eigenfunction, Deng, Ding and Yao \cite{DDY} have established the pointwise kernel estimates for $e^{-t(P(D)+V)}$ with $V$ belongs to the Kato potential class. One can prove that the eigenfunctions decay polynomialy using the heat kernel estimate. Under some suitable assumtions on $V$ and the fact that $P_{ac}=I-P_{disc}$, we have $P_{ac}$ is $L^{p}\rightarrow L^{p}$ bounded for $1\leq p\leq \infty$. Indeed, $P_{disc}=\sum_{j=0}^{\sharp}\langle\cdot,\, e_{j}\rangle e_{j}$ and the decay of eigenfunction $e_{j}$ implies that $P_{ac}$ is $L^{p}\rightarrow L^{p}$ bounded.   Furthermore, we can prove that $\langle x\rangle^{-s}P_{ac}\langle x\rangle^{s}$ is also $L^{p}\rightarrow L^{p}$ bounded for any $s$ positive by the same discussion.

\begin{proposition}\label{eigenfunctiondecay}
Let $V(x)\in L^{\infty}(\mathbb{R}^{d})$, and denote $e_{j}$ to be the eigenfunction of $H$ corresponding to eigenvalue $\lambda_{j}$. Then $e_{j}\in L_{\sigma}^{p}(\mathbb{R}^d)$  for $1\leq p\leq\infty$ and any $\sigma\in\mathbb{R}^{+}$.
\end{proposition}

\begin{proof}
We making use of the heat kernel $K(t, x, y)$ of $e^{-tH}$ to prove the decay of eigenfunction. By the work of Deng, Ding and Yao \cite[Theorem 1.1]{DDY}, we know that $K(t,x,y)$ satisfies
\begin{equation}\label{heatkernelestimate}
  |K(t,x,y)|\leq Ct^{-d/4}\exp\big\{-c|x-y|^{4/3}t^{-1/3}+\theta t\big\},\,\, t>0,
\end{equation}
with some positive constants $C, c, \theta$.

Since $e^{-H}e_{j}=e^{-\lambda_{j}}e_{j}$, then
\begin{equation*}
  \begin{split}
   &\quad \|\langle x\rangle^{\sigma} e_{j}\|_{L^{p}(\mathbb{R}^d)}
       \lesssim \|\langle x\rangle^{\sigma}(e^{-|\cdot|^{4/3}}*e_{j})\|_{L^{p}(\mathbb{R}^d)}=\Big\|\int_{\mathbb{R}^d}\langle x\rangle^{\sigma}e^{-|x-y|^{4/3}}e_{j}(y)dy\Big\|_{L^{p}(\mathbb{R}^d)}\\
      & \lesssim \int_{\mathbb{R}^d}\|\langle x\rangle^{\sigma}e^{-|x-y|^{4/3}}\|_{L^{p}(\mathbb{R}^d)}|e_{j}(y)|dy
      =\int_{\mathbb{R}^d}\|\langle x\rangle^{\sigma}e^{-|x-y|^{4/3}}\langle y\rangle^{\sigma}\|_{L^{p}(\mathbb{R}^d)}|\langle y\rangle^{-\sigma}e_{j}(y)|dy\\
      & \lesssim \int_{\mathbb{R}^d}\|\langle x\rangle^{2\sigma}e^{-|x-y|^{4/3}}(1+|x-y|)^{\sigma}\|_{L^{p}(\mathbb{R}^d)}|\langle y\rangle^{-\sigma}e_{j}(y)|dy\\
      & \lesssim \int_{\mathbb{R}^d}\|\langle x\rangle^{2\sigma}(1+|x-y|)^{-\Pi}\|_{L^{p}(\mathbb{R}^d)}|\langle y\rangle^{-\sigma}e_{j}(y)|dy.
  \end{split}
\end{equation*}
Here we choose $\Pi$ large enough such as $\Pi>2\sigma+d+1$. The last integral is finite by a simple discussion of the distance between $x$ and $y$.
\end{proof}

\begin{theorem}\label{projectionbounded}
Let $V(x)\in L^{\infty}(\mathbb{R}^{d})$ and such that integrals in \eqref{negative eigenvalue 3} and \eqref{negative eigenvalue d} are convergent. Assume that zero is a regular point of $H$ and with only finite many embedded eigenvalues. Then for $\sigma\in \mathbb{R}$ and $1\leq p\leq\infty$, we have
\begin{equation}
  \|\langle x\rangle^{-\sigma}P_{ac}\langle x\rangle^{\sigma}f\|_{L^{p}(\mathbb{R}^d)}\leq c\|f\|_{L^{p}(\mathbb{R}^d)},\,\, f\in L^{p}(\mathbb{R}^d).
\end{equation}
\end{theorem}
\begin{proof}
Since $P_{ac}=I-P_{disc}$, if the above inequality holds for $P_{disc}$ then \eqref{projectionbounded} holds.
In fact $P_{disc}=\sum_{j=0}^{\sharp}\langle\cdot,\,\,e_{j}\rangle e_{j}$, and then
\begin{equation*}
  \begin{split}
 &\quad \|\langle \cdot\rangle^{-\sigma}P_{ac}\langle \cdot\rangle^{\sigma}f\|_{L^{p}(\mathbb{R}^d)}
       =\big\|\sum_{j=0}^{\sharp}\langle\langle x\rangle^{\sigma} f,\,\,e_{j}\rangle \langle y\rangle^{-\sigma}e_{j}\big\|_{L^{p}(\mathbb{R}^d)} \\
       & \lesssim \sum_{j=0}^{\sharp} \big|\langle f,\,\,\langle x\rangle^{\sigma} e_{j}\rangle\big|\|\langle y\rangle^{-\sigma}e_{j}\|_{L^{p}(\mathbb{R}^d)}
       \lesssim \sum_{j=0}^{\sharp} \|f\|_{L^{p}(\mathbb{R}^d)}\|\langle x\rangle^{\sigma} e_{j}\|_{L^{p^{\prime}}(\mathbb{R}^d)}\|\langle y\rangle^{-\sigma}e_{j}\|_{L^{p}(\mathbb{R}^d)}.
  \end{split}
\end{equation*}
Thus Proposition \ref{eigenfunctiondecay} implies that the last sum is finite.
\end{proof}

\subsection{$L^1\rightarrow L^{\infty}$ decay estimate for $d=3$}
Now, we give the proof of Theorem \ref{Ginibred=3}. Our strategy is applying the iterated Duhamel formula
\begin{equation}\label{TDuhamel}
  \begin{split}
   e^{-itH}P_{ac}&=e^{-itH_{0}}P_{ac}+i\int_{0}^{t}e^{-i(t-s)H_{0}}VP_{ac}e^{-isH_{0}}ds\\
                &-\int_{0}^{t}\int_{0}^{s}e^{-i(t-s)H_{0}}Ve^{-i(s-\tau)H}P_{ac}Ve^{-i\tau H_{0}}d\tau ds\\
                &:=I+II+III,
  \end{split}
\end{equation}
and then estimate each term of \eqref{TDuhamel}. Notice that \eqref{freeL1Linfty} and the $L^p$-boundedness of $P_{ac}$ implies the $L^{1}\rightarrow L^{\infty}$ estimate for the free term $e^{-itH_{0}}P_{ac}$.
\begin{equation}\label{H0P_ac}
\|I\|:= \|e^{-itH_{0}}P_{ac}f\|_{L^{\infty}(\mathbb{R}^d)}\lesssim|t|^{-d/4}\|P_{ac}f\|_{L^{1}(\mathbb{R}^d)}\lesssim |t|^{-d/4}\|f\|_{L^{1}(\mathbb{R}^d)}.
\end{equation}
For the second term $II$ of \eqref{TDuhamel}, we have
\begin{equation*}
  \begin{split}
  &\quad \|II\|:= \Big\|\int_{0}^{t}e^{-i(t-s)H_0}VP_{ac}e^{-isH_0}uds\Big\|_{L^{\infty}(\mathbb{R}^3)}\\
  &      \lesssim\int_{0}^{t}(t-s)^{-3/4}\|VP_{ac}e^{-isH_0}u\|_{L^{1}(\mathbb{R}^3)}ds\\
  &\lesssim\int_{0}^{t}(t-s)^{-3/4}\|V\|_{L^{1}(\mathbb{R}^3)}\|P_{ac}\|_{L^{\infty}\rightarrow L^{\infty}(\mathbb{R}^3)}\|e^{-isH_0}u\|_{L^{\infty}(\mathbb{R}^3)}ds\\
  &\lesssim\int_{0}^{t}(t-s)^{-3/4}s^{-3/4}ds\|u\|_{L^{1}(\mathbb{R}^3)}\lesssim |t|^{-1/2}\|u\|_{L^{1}(\mathbb{R}^3)}.
  \end{split}
\end{equation*}
For the third term $III$ of \eqref{TDuhamel}, we have
\begin{equation*}
  \begin{split}
  &\quad \|III\|:= \Big\|\int_{0}^{t}\int_{0}^{s}e^{-i(t-s)H_0}Ve^{-i(s-\tau)H}P_{ac}Ve^{-i\tau H_0}u d\tau ds\Big\|_{L^{\infty}(\mathbb{R}^3)}\\
  &\lesssim\int_{0}^{t}\int_{0}^{s}(t-s)^{-3/4}\|V\langle x\rangle^{\sigma}\langle x\rangle^{-\sigma}e^{-i(s-\tau)H}P_{ac}\langle   x\rangle^{-\sigma}\langle x\rangle^{\sigma}Ve^{-i\tau H_0} u\|_{L^{1}(\mathbb{R}^3)}d\tau ds\\
  &\lesssim\int_{0}^{t}\int_{0}^{s}(t-s)^{-3/4}\|V\langle x\rangle^{\sigma}\|_{L^{2}(\mathbb{R}^3)}\|\langle   x\rangle^{-\sigma}e^{-i(s-\tau)H}P_{ac}\langle x\rangle^{-\sigma}\langle x\rangle^{\sigma}V
  e^{-i\tau H_0} u\|_{L^{2}(\mathbb{R}^3)}d\tau ds\\
  &\lesssim\int_{0}^{t}\int_{0}^{s}(t-s)^{-3/4}\|\langle   x\rangle^{-\sigma}e^{-i(s-\tau)H}P_{ac}\langle x\rangle^{-\sigma}\|_{L^{2}\rightarrow L^{2}(\mathbb{R}^3)}\|e^{-i\tau H_0} u\|_{L^{\infty}(\mathbb{R}^3)}d\tau ds\\
  &\lesssim\int_{0}^{t}\int_{0}^{s}(t-s)^{-3/4}\langle s-\tau \rangle^{-5/4}\tau^{-3/4}d\tau ds\|u\|_{L^{1}(\mathbb{R}^3)}
  \lesssim\langle t\rangle^{-1/2}\|u\|_{L^{1}(\mathbb{R}^3)}.
  \end{split}
\end{equation*}
Therefore, we can combine the steps above to conclude the proof of Theorem \ref{Ginibred=3}.
 \vskip0.5cm
\subsection{$L^{1}\cap L^{2}\rightarrow L^{2}+L^{\infty}$ decay estimate for $d\geq5$}
It is hard to obtain the $L^{1}\rightarrow L^{\infty}$ estimate of the propagator $e^{-itH}$ for operator $H$ with potential  in the dimension $d\ge 5$. Note that  the $L^{1}\rightarrow L^{\infty}$-estimate time decay rate for $e^{-it\Delta^2}$ is $-d/4$, and the decay rate of Jensen-Kato type estimate also is $-d/4$. For $e^{-itH}P_{ac}$ with $d\ge 5$, if we follow the same argument as $d=3$, then the last two integrals of \eqref{TDuhamel} are not convergent. Hence we will establish another type $L^p$ decay estimate, i.e. the $L^{1}\cap L^{2}\rightarrow L^{2}+L^{\infty}$ estimate,  which is weaker than the $L^{1}\rightarrow L^{\infty}$ estimate. Under such norm $\|\cdot\|_{L^{1}\cap L^{2}\rightarrow L^{2}+L^{\infty}}$, it can cancel the singularity at zero of the two integrals of \eqref{TDuhamel}.
\begin{definition}
For any measurable function $f$, if $f=f_1+f_2$ with $f_1\in L^{2}(\mathbb{R}^{d}),\,\, f_2\in L^{\infty}(\mathbb{R}^{d})$ and satisfies
\begin{equation*}
  \inf\left\{\,\, \|f_1\|_{L^{2}(\mathbb{R}^{d})}+\|f_2\|_{L^{\infty}(\mathbb{R}^{d})}\,\, \right\}<\infty.
\end{equation*}
Here the infimum takes from all the splitting of $f$. Then we denote $f\in L^{2}+L^{\infty}(\mathbb{R}^{d})$, and $L^{2}+L^{\infty}(\mathbb{R}^{d})$ is a Banach space with the norm
$$\|f\|_{L^{2}+L^{\infty}(\mathbb{R}^{d})}= \inf\left\{\,\, \|f_1\|_{L^{2}(\mathbb{R}^{d})}+\|f_2\|_{L^{\infty}(\mathbb{R}^{d})}\,\, \right\}.$$
\end{definition}
Note that for $f\in L^{2}\cap L^{\infty}(\mathbb{R}^{d})$, since $f$ can be divided as $f=f+0=0+f$, then $\|f\|_{L^{2}+L^{\infty}(\mathbb{R}^{d})}\leq \|f\|_{L^{2}(\mathbb{R}^{d})}$ and $\|f\|_{L^{2}+L^{\infty}(\mathbb{R}^{d})}\leq \|f\|_{L^{\infty}(\mathbb{R}^{d})}$.

We start from the following lemma since we will face such kind of integral in the proof.
\begin{lemma}
For any $a>0$ and $b>0$, we have
\begin{align}\label{jjbds}
\int_{0}^{t}\frac{ds}{\langle t-s\rangle^{a}\langle s\rangle^{b}}\leq
\begin{cases}
    C\langle t\rangle^{-a-b+1}, & 0<a,\, b<1,\\
    C\langle t\rangle^{-\min\{a,\,\, b\}}, & \text{otherwise}.
\end{cases}
\end{align}
\end{lemma}
\begin{proof}
Since $c_{1}(1+|x|)\leq\langle x\rangle\leq c_2(1+|x|)$, $c_1$ and $c_2$ are some positive constants, then
\begin{equation*}
 \begin{split}
  &\int_{0}^{t}\frac{ds}{\langle t-s\rangle^{a}\langle s\rangle^{b}}\simeq
     \int_{0}^{t}\frac{ds}{(1+(t-s))^{a}(1+s)^{b}}\\
  =&\int_{0}^{t/2}\frac{ds}{(1+(t-s))^{a}(1+s)^{b}}+\int_{t/2}^{t}\frac{ds}{(1+(t-s))^{a}(1+s)^{b}}\\
  =&\int_{0}^{t/2}\frac{ds}{(1+(t-s))^{a}(1+s)^{b}}+\int_{0}^{t/2}\frac{d\tau}{(1+\tau)^{a}(1+(t-\tau))^{b}}\\
  \leq &(1+\frac{t}{2})^{-a}\int_{0}^{t/2}\frac{ds}{(1+s)^{b}}+(1+\frac{t}{2})^{-b}\int_{0}^{t/2}\frac{d\tau}{(1+\tau)^{a}}\\
  =&\frac{1}{1-b}(1+\frac{t}{2})^{-a}[(1+\frac{t}{2})^{-b+1}-1]+\frac{1}{1-a}(1+\frac{t}{2})^{-b}[(1+\frac{t}{2})^{-a+1}-1]
 \end{split}
\end{equation*}
Now, the inequality \eqref{jjbds} holds by a simple discussion of the relationship between $a, b$ and $1$. The constant $C$ only depends on $a$ and $b$.
\end{proof}
%\begin{theorem}\label{Ginibred5}
%Let  $d\geq5$ and  $\langle x\rangle^{\sigma}V\in L^{2}\cap L^{\infty}(\mathbb{R}^d)$ with $\sigma>\frac{d}{2}+1$. Then under the same spectrum assumption of $H$ as Theorem \ref{timedecay},  we have
%\begin{equation}\label{Ginid}
%\|e^{-itH}P_{ac}u\|_{L^{2}+L^{\infty}(\mathbb{R}^d)}\lesssim\langle t\rangle^{-d/4}\|u\|_{L^{2}\cap L^{1}(\mathbb{R}^d)}.
%\end{equation}
%\end{theorem}
Now, we give the proof of $L^{2}\cap L^{1}(\mathbb{R}^d)\rightarrow L^{2}+L^{\infty}(\mathbb{R}^d)$ decay estimate for $d\geq5$.
\begin{proof}[\bf Proof of Theorem \ref{Ginibred5}]
Similarly,  applying the iterated Duhamel formula  \eqref{TDuhamel} again.  By the definition of $L^{2}+L^{\infty}$, we have
\begin{equation*}
  \|e^{-itH_{0}}P_{ac}u\|_{L^{2}+L^{\infty}(\mathbb{R}^d)}\leq \min\left\{\, \|e^{-itH_{0}}P_{ac}u\|_{L^{2}(\mathbb{R}^d)},\,\, \|e^{-itH_{0}}P_{ac}u\|_{L^{\infty}(\mathbb{R}^d)}  \,\right\}.
\end{equation*}
Then for $0<|t|\leq1$, we have
$\|e^{-itH_{0}}P_{ac}u\|_{L^{2}+L^{\infty}(\mathbb{R}^d)}\leq \|u\|_{L^{2}(\mathbb{R}^{d})}$.
And for $|t|>1$, by the estimate \eqref{freeL1Linfty}, we have
$\|e^{-itH_{0}}P_{ac}u\|_{L^{2}+L^{\infty}(\mathbb{R}^d)}\leq |t|^{-d/4}\|u\|_{L^{1}(\mathbb{R}^{d})}$ .
Thus for the first free term $I$ we have
\begin{equation}
  \|e^{-itH_{0}}P_{ac}u\|_{L^{2}+L^{\infty}(\mathbb{R}^d)}\lesssim\langle t\rangle^{-d/4}\|u\|_{L^{1}\cap L^{2}(\mathbb{R}^{d})}.
\end{equation}
For the second term $II$ of \eqref{TDuhamel}, we have
\begin{equation*}
  \begin{split}
  &\quad \int_{0}^{t}\|e^{-i(t-s)H_0}VP_{ac}e^{-isH_0}u\|_{L^{2}+L^{\infty}(\mathbb{R}^d)}ds\\
  &\lesssim\int_{0}^{t}\langle t-s\rangle^{-d/4}\|VP_{ac}e^{-isH_0}u\|_{L^{1}\cap L^{2}(\mathbb{R}^d)}ds\\
  &\lesssim\int_{0}^{t}\langle t-s\rangle^{-d/4}\|V\langle x\rangle^{\sigma}\|_{L^{\infty}\cap L^{2}(\mathbb{R}^d)} \|\langle x\rangle^{-\sigma}P_{ac}e^{-isH_0}u\|_{L^{2}(\mathbb{R}^d)}ds\\
  &\lesssim\int_{1}^{t}\langle t-s\rangle^{-d/4}\|\langle x\rangle^{-\sigma}\|_{L^{2}}\|P_{ac}e^{-isH_0}u\|_{L^{\infty}(\mathbb{R}^d)}ds\\
  &\quad +\int_{0}^{1}\langle t-s\rangle^{-d/4}\|\langle x\rangle^{-\sigma}\|_{L^{\infty}}\|P_{ac}e^{-isH_0}u\|_{L^{2}(\mathbb{R}^d)}ds\\
  &\lesssim\int_{1}^{t}\langle t-s\rangle^{-d/4} s^{-d/4}ds\|u\|_{L^{1}\cap L^{2}(\mathbb{R}^d)}+\int_{0}^{1}\langle t-s\rangle^{-d/4}2^{d/4}\langle s\rangle^{-d/4} ds\|u\|_{L^{1}\cap L^{2}(\mathbb{R}^d)}\\
  &\lesssim \langle t\rangle^{-d/4}\|u\|_{L^{2}\cap L^{1}(\mathbb{R}^d)}.
  \end{split}
\end{equation*}
For the third term $III$ of \eqref{TDuhamel}, we have
\begin{equation*}
  \begin{split}
  &\quad \int_{0}^{t}\int_{0}^{s}\Big\|e^{-i(t-s)H_0}Ve^{-i(s-\tau)H}P_{ac}Ve^{-i\tau H_0}u \Big\|_{L^{2}+L^{\infty}(\mathbb{R}^d)}d\tau ds\\
  &\lesssim\int_{0}^{t}\int_{0}^{s}\langle t-s\rangle^{-d/4}\|Ve^{-i(s-\tau)H}P_{ac}Ve^{-i\tau H_0}   u\|_{L^{1}\cap L^{2}(\mathbb{R}^d)}d\tau ds\\
  &\lesssim\int_{0}^{t}\int_{0}^{s}\langle t-s\rangle^{-d/4}\|V\langle x\rangle^{\sigma}\|_{L^{\infty}\cap L^{2}(\mathbb{R}^d)}\|\langle x\rangle^{-\sigma}e^{-i(s-\tau)H}P_{ac}Ve^{-i\tau H_0} u\|_{L^{2}(\mathbb{R}^d)}d\tau ds\\
  &\lesssim\int_{0}^{t}\int_{0}^{s}\langle t-s\rangle^{-d/4}\|\langle x\rangle^{-\sigma}e^{-i(s-\tau)H}P_{ac}\langle x\rangle^{-\sigma}\|_{L^{2}\rightarrow L^{2}(\mathbb{R}^{d})}\|\langle x\rangle^{\sigma}Ve^{-i\tau H_0}u\|_{L^{2}(\mathbb{R}^d)}d\tau ds\\
  &\lesssim\int_{0}^{t}\int_{0}^{1}\langle t-s\rangle^{-d/4}\langle s-\tau\rangle^{-d/4}\|\langle x\rangle^{\sigma}V\|_{L^{\infty}(\mathbb{R}^d)}\|e^{-i\tau H_0} u\|_{L^{2}(\mathbb{R}^d)}d\tau ds\\
  &\quad +\int_{0}^{t}\int_{1}^{s}\langle t-s\rangle^{-d/4}\langle s-\tau\rangle^{-d/4}\|\langle x\rangle^{\sigma}V\|_{L^{2}(\mathbb{R}^d)}\|e^{-i\tau H_0} u\|_{L^{\infty}(\mathbb{R}^d)}d\tau ds\\
  &\lesssim\int_{0}^{t}\int_{0}^{1}\langle t-s\rangle^{-d/4}\langle s-\tau\rangle^{-d/4} \|u\|_{L^{2}(\mathbb{R}^d)}d\tau ds\\
  &\quad +\int_{0}^{t}\int_{1}^{s}\langle t-s\rangle^{-d/4}\langle s-\tau\rangle^{-d/4} \tau^{-d/4} \|u\|_{L^{1}(\mathbb{R}^d)}d\tau ds\\
  &\lesssim\int_{0}^{t}\int_{0}^{1}\langle t-s\rangle^{-d/4}\langle s-\tau\rangle^{-d/4}2^{d/4}\langle \tau\rangle^{-d/4} \|u\|_{L^{1}\cap L^{2}(\mathbb{R}^d)}d\tau ds\\
  &\quad +\int_{0}^{t}\int_{1}^{s}\langle t-s\rangle^{-d/4}\langle s-\tau\rangle^{-d/4} \tau^{-d/4} \|u\|_{L^{1}\cap L^{2}(\mathbb{R}^d)}d\tau ds\\
  &\lesssim\langle t\rangle^{-d/4}\|u\|_{L^{1}\cap L^{2}(\mathbb{R}^d)}.
  \end{split}
\end{equation*}
Thus again, we can combine the steps above to conclude the proof of Theorem \ref{Ginibred5}.
\end{proof}

\section{Endpoint Strichartz estimates for $d\geq5$}
\setcounter{equation}{0}
According to Keel-Tao's method \cite{KT}, Strichartz estimate can be obtained from the $L^{1}\rightarrow L^{\infty}$ decay estimate of $e^{-itH}$. For free operator $H_{0}$, we can  get the $L^{1}\rightarrow L^{\infty}$ decay estimate by Fourier transform. But for the perturbed operator $H=H_0+V$  it is much harder. Here, we  obtain the $L^{1}\rightarrow L^{\infty}$ estimate with $t$ small and large separately in 3-dimensions,  and then give the Strichartz estimate local in time. For $d\geq5$, we apply the local decay estimate to derive the endpoint Strichartz estimate. Note that the endpoint pair are $q=2$ and $r^{\prime}=2d/(d-4)$ for $d\geq5$, the same argument of $d\ge 5$ does not work for 3-dimensional case. In this section, we first give the result of $d=3$ and then give the proof of Theorem \ref{SchtriE} for $d\geq5$.
\begin{proposition}
Let $d=3$ and $H$ satisfies the same conditions as in Theorem \ref{Ginibred=3}, and $1\leq q<\infty, 2\leq r\leq\infty$ satisfies $4/q+6/r>3$. Then for finite positive number $T$, we have
\begin{equation}\label{strichartzd=3}
  \|e^{-itH}P_{ac}u\|_{L_{t}^{q}L^{r}_{x}([0,T]\times\mathbb{R}^{3})}\leq C T^{\Im}\|u\|_{L^{r^{\prime}}(\mathbb{R}^{3})},
\end{equation}
where $\Im=\frac{1}{q}+\frac{3}{2r}-\frac{3}{4}$.
\end{proposition}
\begin{proof}
First, by the iterated Duhamel formula \eqref{TDuhamel}, scaling on time partially, we have
\begin{equation*}
  \begin{split}
   e^{-itH}P_{ac}&=e^{-i(t/T)(TH_{0})}P_{ac}+i\int_{0}^{t}e^{-i((t-s)/T)(TH_{0})}VP_{ac}e^{-i(s/T)(TH_{0})}ds\\
                &-\int_{0}^{t}\int_{0}^{s}e^{-i(t-s)H_{0}}Ve^{-i(s-\tau)H}P_{ac}Ve^{-i\tau H_{0}}d\tau ds\\
                &:=I^{\prime}+II^{\prime}+III.
  \end{split}
\end{equation*}
Following the same argument, we have
\begin{equation*}
  \|I^{\prime}\|\lesssim |t/T|^{-3/4}, ~~\|II^{\prime}\|\lesssim |t/T|^{-1/2}, ~~\|III\|\lesssim \langle t\rangle^{-1/2}.
\end{equation*}
Thus we have $$\|e^{-itH}P_{ac}\|_{L^{1}(\mathbb{R}^{3})\rightarrow L^{\infty}(\mathbb{R}^{3})}\lesssim |t|^{-3/4}, ~0<|t|<T.$$ Finally, we can obtain the estimate by taking interpolation of the $L^{1}\rightarrow L^{\infty}$ estimate and $L^{2}\rightarrow L^{2}$ estimate of $e^{-itH}P_{ac}$.
\end{proof}

Now we will show how to apply local decay estimate to derive Strichartz estimate with $d\geq5$. We start from the free case. For the free operator $H_0=(-\Delta)^2$, by making use of the $L^{1}(\mathbb{R}^d)\rightarrow L^{\infty}(\mathbb{R}^d)$ decay estimate \eqref{freeL1Linfty} and Keel-Tao's method \cite{KT},  we have the following estimate.
\begin{lemma}\label{free Strichartz}
For the free fourth-order Schr\"odinger equation in $\mathbb{R}^{d}$ with $d\geq5$,  we have
\begin{equation}\label{freestri}
  \|e^{-itH_0}u\|_{L_{t}^{q}L_{x}^{r}(\mathbb{R}\times\mathbb{R}^d)}\lesssim\|u\|_{L^{2}(\mathbb{R}^d)},
\end{equation}
\begin{equation}\label{freeretarded}
  \Big\|\int_{s<t}e^{i(t-s)H_0}f(s)ds\Big\|_{L_{t}^{q}L_{x}^{r}(\mathbb{R}\times\mathbb{R}^d)}
  \lesssim\|f\|_{L_{t}^{\tilde{q}^{\prime}}L_{x}^{\tilde{r}^{\prime}}(\mathbb{R}\times\mathbb{R}^d)}.
\end{equation}
where $(q, r), \,(\tilde{q}^{\prime},\,\tilde{r}^{\prime})$ satisfy \eqref{admissblepair}.
\end{lemma}

\vskip0.5cm
\noindent {\bf Proof of Theorem \ref{SchtriE} ( i.e. the global endpoint Strichartz estimates of $e^{-itH}$ )}.
\begin{proof} We divide the proof into the following several steps.

Step 1:\,\,We aim to show the homogeneous Strichartz estimate \eqref{homo} and the dual homogeneous Strichartz estimate \eqref{dualhomo}.

Let us consider the following equation
\begin{equation}\label{freefourthorder equ}
\left\{ \begin{gathered}
   i\partial_t \psi = H_0\psi+V\psi, \hfill \\
   \psi(0,\cdot)=\psi_0 \in L^{2}(\mathbb{R}^d) . \hfill \\
\end{gathered}  \right.
\end{equation}
For the homogeneous Strichartz estimate \eqref{homo}, using Duhamel formula, it is enough to show
\begin{equation*}
 \Big\|\int_{0}^{t} e^{-i(t-s)H_{0}}Ve^{-isH}P_{ac}\psi ds\Big\|_{L_{t}^{q}L_{x}^{r}(\mathbb{R}\times\mathbb{R}^d)}\lesssim\|\psi\|_{L^{2}(\mathbb{R}^d)}.
\end{equation*}
In fact, by \eqref{freeretarded} and the local decay estimate \eqref{local decay}, we have
\begin{equation*}
 \begin{split}
  &\quad\Big\|\int_{0}^{t} e^{-i(t-s)H_{0}}Ve^{-isH}P_{ac}\psi ds\Big\|_{L_{t}^{q}L_{x}^{r}(\mathbb{R}\times\mathbb{R}^d)}
  \lesssim\|Ve^{-itH}P_{ac}\psi \|_{L_{t}^{2}L_{x}^{\frac{2d}{d+4}}(\mathbb{R}\times\mathbb{R}^d)}\\
  &\lesssim\|V\langle x\rangle^{\sigma}\|_{L^{d/2}(\mathbb{R}^d)}\|\langle   x\rangle^{-\sigma}e^{-itH}P_{ac}\psi\|_{L_{t}^{2}L_{x}^{2}(\mathbb{R}\times\mathbb{R}^d)}\lesssim\|\psi\|_{L^{2}(\mathbb{R}^d)}.
 \end{split}
\end{equation*}
Further, the dual homogeneous Strichartz estimate \eqref{dualhomo} follows by the $T^*T$-method.

Step 2:\,\, We aim to show the retarded Strichartz estimate \eqref{retarded}. The solution $\Psi(t,x)$ of equation \eqref{fourthorder equ} satisfies
\begin{equation*}
 P_{ac}\Psi(t,x)=e^{itH_0}P_{ac}\Psi_0-i\int_{0}^{t}e^{i(t-s)H_0}VP_{ac}\Psi(s)ds-i\int_{0}^{t}e^{i(t-s)H_0}P_{ac}h(s)ds.
\end{equation*}
Then by H\"older inequality and step 1, we have
\begin{equation*}
  \begin{split}
  &\|P_{ac}\Psi(t,x)\|_{L_{t}^{q}L_{x}^{r}(\mathbb{R}\times\mathbb{R}^d)}\\
  \lesssim&~\|e^{itH_0}P_{ac}\Psi_0\|_{L_{t}^{q}L_{x}^{r}(\mathbb{R}\times\mathbb{R}^d)}
                +\Big\|\int_{0}^{t}e^{i(t-s)H_0}P_{ac}h(s,\cdot)ds\Big\|_{L_{t}^{q}L_{x}^{r}(\mathbb{R}\times\mathbb{R}^d)}\\
  &  +\Big\|\int_{0}^{t}e^{i(t-s)H_0}VP_{ac}\Psi(s)ds\Big\|_{L_{t}^{q}L_{x}^{r}(\mathbb{R}\times\mathbb{R}^d)}\\
  \lesssim &~\|\Psi_0\|_{L^2}+\|h(t)\|_{L_{t}^{\tilde{q}^{\prime}}L_{x}^{\tilde{r}^{\prime}}(\mathbb{R}\times\mathbb{R}^d)}
              +\|VP_{ac}\Psi(t)\|_{L_{t}^{2}L_{x}^{\frac{2d}{d+4}}(\mathbb{R}\times\mathbb{R}^d)}\\
  \lesssim &~\|\Psi_0\|_{L^2}+\|h(t)\|_{L_{t}^{\tilde{q}^{\prime}}L_{x}^{\tilde{r}^{\prime}}(\mathbb{R}\times\mathbb{R}^d)}
              +\|V\langle x\rangle^{\sigma}\|_{L^{d/2}(\mathbb{R}^d)}\|\langle x\rangle^{-\sigma}P_{ac}\Psi(t)\|_{L_{t}^{2}L_{x}^{2}(\mathbb{R}\times\mathbb{R}^d)}.\\
  \end{split}
\end{equation*}

Now we  show that
\begin{equation}
  \|\langle x\rangle^{-\sigma}P_{ac}\Psi(t)\|_{L_{t}^{2}L_{x}^{2}(\mathbb{R}\times\mathbb{R}^d)}\lesssim\|\Psi_0\|_{L^2(\mathbb{R}^d)}
  +\|h\|_{L_{t}^{\tilde{q}^{\prime}}L_{x}^{\tilde{r}^{\prime}}(\mathbb{R}\times\mathbb{R}^d)}.
\end{equation}
First, by Duhamel formula for $\Psi$, we have
\begin{equation*}
 \begin{split}
  \|\langle x\rangle^{-\sigma}P_{ac}\Psi(t)\|_{L_{t}^{2}L_{x}^{2}(\mathbb{R}\times\mathbb{R}^d)}
\lesssim &~\|\langle x\rangle^{-\sigma}e^{itH}P_{ac}\Psi_{0}\|_{L_{t}^{2}L_{x}^{2}(\mathbb{R}\times\mathbb{R}^d)}\\
         & +\Big\|\int_{0}^{t}\langle x\rangle^{-\sigma}e^{i(t-s)H}P_{ac}h(s)ds\Big\|_{L_{t}^{2}L_{x}^{2}(\mathbb{R}\times\mathbb{R}^d)}.
 \end{split}
\end{equation*}
For the first term on the right hand side we use local decay estimates. Then, we will finish the proof which only needs to show
\begin{equation}\label{source local decay}
  \Big\|\int_{0}^{t}\langle x\rangle^{-\sigma}e^{i(t-s)H}P_{ac}h(s)ds\Big\|_{L_{t}^{2}L_{x}^{2}(\mathbb{R}\times\mathbb{R}^d)}
  \lesssim\|\Psi_0\|_{L^2(\mathbb{R})}
  +\|h\|_{L_{t}^{\tilde{q}^{\prime}}L_{x}^{\tilde{r}^{\prime}}(\mathbb{R}\times\mathbb{R}^d)}.
\end{equation}

\vskip0.5cm
Step 3:\,\, We show the local decay estimate of the source term \eqref{source local decay}.

Consider the Cauchy problem
\begin{equation}
\left\{ \begin{gathered}
   i\partial_{t}\phi= H_0\phi+h(t)= H\phi-V\phi+h(t), \hfill \\
   \phi(0,\cdot)=\Psi_{0}. \hfill \\
\end{gathered}  \right.
\end{equation}
Then Duhamel formula for the solution $\phi(t,x)$ reads
\begin{equation}\label{01}
  P_{ac}\phi(t)=e^{itH}P_{ac}\Psi_{0}+i\int_{0}^{t}e^{i(t-s)H}P_{ac}V\phi(s)ds-i\int_{0}^{t}e^{i(t-s)H}P_{ac}h(s)ds.
\end{equation}
For the left hand side of \eqref{01}, since $\phi(t)$ is also a solution of $i\partial_{t}\phi=H_0\phi+h(t)$,
by \eqref{freestri}, \eqref{freeretarded} and Duhamel formula again, we have
\begin{equation*}
  \begin{split}
  &\|\langle x\rangle^{-\sigma}P_{ac}\phi(t)\|_{L^{2}_{t}L_{x}^{2}(\mathbb{R}\times\mathbb{R}^d)}\\
  \lesssim&~\|\langle x\rangle^{-\sigma}P_{ac}e^{itH_{0}}\Psi_0\|_{L^{2}_{t}L_{x}^{2}(\mathbb{R}\times\mathbb{R}^d)}
            +\Big\|\langle x\rangle^{-\sigma}P_{ac}\int_{0}^{t}e^{i(t-s)H_0}h(s)ds\Big\|_{L^{2}_{t}L_{x}^{2}(\mathbb{R}\times\mathbb{R}^d)}\\
  \lesssim&~\|\langle x\rangle^{-\sigma}P_{ac}\langle x\rangle^{\sigma}\langle x\rangle^{-\sigma}e^{itH_{0}}\Psi_0\|_{L^{2}_{t}L_{x}^{2}(\mathbb{R}\times\mathbb{R}^d)}\\
   & +\Big\|\langle x\rangle^{-\sigma}P_{ac}\langle x\rangle^{\sigma}\langle x\rangle^{-\sigma}\int_{0}^{t}e^{i(t-s)H_0}h(s)ds\Big\|_{L^{2}_{t}L_{x}^{2}(\mathbb{R}\times\mathbb{R}^d)}\\
  \lesssim&~\|\Psi_0\|_{L^2(\mathbb{R}^d)}+\Big\|\int_{0}^{t}e^{i(t-s)H_0}h(s)ds\Big\|_{L^{2}_{t}L_{x}^{\frac{2d}{d-4}}(\mathbb{R}\times\mathbb{R}^d)}\\
  \lesssim&~\|\Psi_0\|_{L^2(\mathbb{R}^d)}+\|h\|_{L_{t}^{\tilde{q}^{\prime}}L_{x}^{\tilde{r}^{\prime}}(\mathbb{R}\times\mathbb{R}^d)}.
  \end{split}
\end{equation*}
Here, we apply Theorem \ref{projectionbounded}, the boundedness of $P_{ac}$. The local decay estimate for the first term on the right hand side of \eqref{01} follows from \eqref{local decay}.

For the second term of the right hand side of \eqref{01}, notice that
\begin{equation*}
  \begin{split}
  &\Big\|\int_{0}^{t}\langle x\rangle^{-\sigma}e^{i(t-s)H}P_{ac}V\phi(s)ds\Big\|_{L^{2}_{t}L_{x}^{2}(\mathbb{R}\times\mathbb{R}^d)}\\
  \lesssim&~\Big\|\int_{0}^{t}\|\langle x\rangle^{-\sigma}e^{i(t-s)H}P_{ac}V\phi(s)\|_{L_{x}^{2}(\mathbb{R}^d)}ds\Big\|_{L^{2}_{t}(\mathbb{R})}\\
  \lesssim&~\Big\|\int_{0}^{t}\langle t-s\rangle^{-d/4}\|\langle x\rangle^{\sigma}V\phi(s)\|_{L_{x}^{2}(\mathbb{R}^d)}ds\Big\|_{L^{2}_{t}(\mathbb{R})}\\
  \lesssim&~\|\Psi_0\|_{L^2}+\|h\|_{L_{t}^{\tilde{q}^{\prime}}L_{x}^{\tilde{r}^{\prime}}(\mathbb{R}\times\mathbb{R}^d)}.
  \end{split}
\end{equation*}
Thus the whole proof can be concluded.
\end{proof}

\section{Jensen-Kato type decay estimates------the conjugate operator method}\label{commutator method}
\setcounter{equation}{0}

In this section we  apply the abstract theory of decay estimates to the fourth-order Schr\"odinger operator. The abstract theory of decay estimates was developed by Geoegescu, Larenas and Soffer \cite{LS, GG}. This is a completely independent method of getting pointwise estimates in time depending on positive commutator techniques. For dispersive equations, linear or nonlinear, quantitative estimates of the decay rate of the solution is always needed. Here, we establish the pointwise decay estimate of Jensen-Kato type for $H$ by this method. For the history of commutator method,  we refer the readers to  Amrein,  Boutet de Monvel and Georgescu \cite{ABG}.

The conjugate operator $A:=-\frac{i}{2}(x\cdot\nabla+\nabla\cdot x)$ and $P_{ac}$ is the projection onto the space of absolutely continuous spectrum of $H$. In order to apply the abstract theory, we need to verify the following conditions:
\begin{description}
  \item[(a)] $H$ is of class $C^{1}(A)$;
  \item[(b)] $P_{ac}[H,iA]P_{ac}=P_{ac}(q(H)+K)P_{ac}$ for $K\equiv F^{*}E$ in the sense that
  \begin{equation*}
  (\phi, P_{ac}[H,iA]P_{ac}\psi)=(P_{ac}\phi, q(H)P_{ac}\psi)+(FP_{ac}\phi, EP_{ac}\psi)
  \end{equation*}
for $\phi, \psi\in \mathcal{D}(H)$. $q(H)$ is a function of $H$, and $E, F$ are Kato $H$-smooth on the range of $P_{ac}$;
  \item[(c)] $K$ is symmetric on $\mathcal{D}(H)$ ;
  \item[(d)] $K$ is bounded on $L^{2}$  ;
  \item[(e)] $(\phi,P_{ac}[A, K]P_{ac} \psi)=(F^{\prime}P_{ac}\phi, E^{\prime}P_{ac}\psi)$, $F^{\prime}, E^{\prime}$ are Kato $H$-smooth on the range of $P_{ac}.$
\end{description}
Here $\mathcal{D}(H)$ be the form domain of $H$, and by the KLMN theorem we have
$$\mathcal{D}(H)=\mathcal{D}(H_{0})=\mathcal{H}^{2}.$$
For (b), since
$$[H, iA]=4H_{0}-x\cdot\nabla V=4H-(4V(x)+x\cdot\nabla V)=q(H)-K$$
So, $q(H)=4H$ and $K=4V(x)+x\cdot\nabla V$. Further, $4V(x)+x\cdot\nabla V$ is a real valued function, so (c) holds. Rewrite $K={\rm sgn}(K)|K|^{\frac{1}{2}}|K|^{\frac{1}{2}}$, then $(|K|^{\frac{1}{2}})^{*}=|K|^{\frac{1}{2}}$. Then take $E={\rm sgn}(K)|K|^{\frac{1}{2}}$
and $F=|K|^{\frac{1}{2}}$, so $E$ and $F$ are bounded operators on $L^{2}$, and then (d) holds.
For (e), since $[A, K]=-i(x\cdot\nabla V)(4+x\cdot\nabla V)$, so $E^{\prime}$ and $F^{\prime}$ are bounded operators on $L^{2}$ by the assumptions on the potential $V$. From the assumptions on $V$, $x\cdot\nabla V$, Theorem \ref{local decay thm} implies that $E, F, E^{\prime}$ and $F^{\prime}$ are Kato $H$-smooth. Now, we aim to show the first one.
\begin{proposition}\label{EEE}
For measurable function $f$ and operator $\Lambda=-i\nabla_x$, then
\begin{equation}\label{TE}
  e^{-itA}f(\Lambda)e^{itA}=f(e^{-t}\Lambda).
\end{equation}
\end{proposition}
\begin{proof}
Since $A=\frac{1}{2}(x\cdot \Lambda+\Lambda\cdot x)$, and for any $\varphi(x)\in L^{2}(\mathbb{R}^{d})$, we have
\begin{equation*}
  (e^{-itA}\varphi)(x)=e^{-dt/2}\varphi(e^{-t}x),
\end{equation*}
and then
\begin{equation*}
  \begin{split}
    \big[e^{-itA}f(\Lambda)e^{itA}\varphi\big]^{\wedge}(\xi)
       & = e^{-itA} [f(\xi)e^{dt/2}\widehat{\varphi}(e^{t}\xi)] \\
       & = e^{-dt/2}f(e^{-t}\xi)e^{dt/2}\widehat{\varphi}(e^{-t}\cdot e^{t}\xi)\\
       & = f(e^{-t}\xi)\widehat{\varphi}(\xi),
  \end{split}
\end{equation*}
where $\widehat{\varphi}(\xi)$ is the Fourier transform of $\varphi(x)$.
\end{proof}
\begin{proposition}
$H_{0}$ is of class $C^{1}(A)$.
\end{proposition}
\begin{proof}
By the definition of $C^{1}(A)$, we need to show the operator valued function
$$G_{0}(t)=e^{itA}(H_{0}-i)^{-1}e^{-itA}\in C^{1}\big(\mathcal{B}(L^{2}(\mathbb{R}^{d}))\big)$$
for $t\in \mathbb{R}$ . By Proposition \ref{EEE}, we know
$$G_{0}(t)=(e^{-4t}\Delta^{2}-i)^{-1}, G_{0}^{\prime}(t)=(4e^{-4t}\Delta^{2})(e^{-4t}\Delta^{2}-i)^{-2}.$$
We need to check that
\begin{equation*}
\begin{split}
  \|G_{0}^{\prime}(t)\|_{L^{2}\rightarrow L^{2}}
    & = \Big\|\frac{4e^{-4t}|\xi|^{4}}{(e^{-4t}|\xi|^{4}-i)^{2}}\Big\|_{L^{2}\rightarrow L^{2}} \\
    & \leq 4\Big\|\frac{1}{e^{-4t}|\xi|^{4}-i}\Big\|_{L^{2}\rightarrow L^{2}}
           +4\Big\|\frac{1}{(e^{-4t}|\xi|^{4}-i)^{2}}\Big\|_{L^{2}\rightarrow L^{2}} \leq C.
\end{split}
\end{equation*}
Similarly, we have $\|G_{0}^{\prime\prime}(t)\|_{L^{2}\rightarrow L^{2}}\leq C^{\prime}$. Here $C,\, C^{\prime}$ are positive constants.
\end{proof}
%\begin{remark}
%Using the same strategy, it's easy to prove $H_{0}$ is of class $C^{k}(A)$ for $k\in \mathbb{N}^{+}$.
%\end{remark}

\begin{proposition}
Suppose the potential function $V(x)$ satisfies $x\cdot\nabla V\in L^{\infty}(\mathbb{R}^{d})$, then
\begin{center}
  $H=H_{0}+V\in C^{1}(A)$.
\end{center}
\end{proposition}
\begin{proof}
By the definition of $C^{1}(A)$, we need to check the function
$$G(t)=e^{itA}(H-i)^{-1}e^{-itA}\in C^{1}\big(\mathcal{B}(L^{2}(\mathbb{R}^{d}))\big)$$
for $t\in \mathbb{R}$. By the second resolvent formula, we have
\begin{equation*}
\begin{split}
  G(t) & =e^{itA}\big[(H_{0}-i)^{-1}-(H-i)^{-1}V(H_{0}-i)^{-1}\big]e^{-itA} \\
       & =e^{itA}(H_{0}-i)^{-1}e^{-itA}-e^{itA}(H-i)^{-1}e^{-itA}e^{itA}Ve^{-itA}e^{itA}(H_{0}-i)^{-1}e^{-itA}\\
       & =G_{0}(t)-G(t)\tilde{V}(t)G_{0}(t),
\end{split}
\end{equation*}
Here $\tilde{V}(t)$ denotes $e^{itA}Ve^{-itA}$. So we have $G(t)[1+G_{0}(t)\tilde{V}(t)]=G_{0}(t)$. While by Proposition \ref{Fredholm}, the operator
\begin{equation*}
  1+G_{0}(t)\tilde{V}(t)=e^{itA}[1+(H_{0}-i)^{-1}V]e^{-itA}
\end{equation*}
is invertible, thus we get the relationship between $G(t)$ and $G_{0}(t)$,
\begin{equation}\label{G and G0}
  G(t)=G_{0}(t)[1+G_{0}(t)\tilde{V}(t)]^{-1}.
\end{equation}
Further,
\begin{equation}\label{G'}
  G^{\prime}(t)=G^{\prime}_{0}(t)[1+G_{0}(t)\tilde{V}(t)]^{-1}
-G_{0}(t)[1+G_{0}(t)\tilde{V}(t)]^{-2}[G_{0}^{\prime}(t)\tilde{V}(t)+G_{0}(t)\tilde{V}^{\prime}(t)].
\end{equation}

Since $H_{0}\in C^{1}(A)$ equals $G_{0}(t), G^{\prime}_{0}(t)$ are continuous, so if $\tilde{V}(t)$ and $\tilde{V}^{\prime}(t)$
are continuous in $\mathcal{B}(L^{2}(\mathbb{R}^{d}))$, then the proof done. For $\tilde{V}(t)$, $V(x)\in L^{\infty}$ implies that $\tilde{V}(t)$ is continuous. For the second one, since
\begin{equation*}
  \tilde{V}^{\prime}(t)=e^{itA}[iA, V]e^{-itA}=e^{itA}(x\cdot\nabla V)e^{-itA},
\end{equation*}
so by the assumption $\tilde{V}^{\prime}(t)$ is continuous.
\end{proof}
Denote $\mathcal{E}$  be the collection
\begin{equation*}
 \mathcal{E}=\left\{~u\in \mathcal{D}(H)~\big|~\psi_{u}(t):=\langle u, e^{itH}u\rangle\in L_{t}^{2}(\mathbb{R})~\right\}.
\end{equation*}
It was shown that $\mathcal{E}$ is a dense linear subspace of the absolute continuity subspace of $H$ and $[u]_{H}=\|\psi_{u}\|_{L^{2}_{t}}^{\frac{1}{2}}$ is a complete norm on it, see \cite{ABG, GLS}. Note that
\begin{equation*}
  [u]_{H}^{2}=\int_{\mathbb{R}}|\langle u, e^{itH}u\rangle|^{2}dt=2\pi\int_{\mathbb{R}} E_{u}^{\prime}(\lambda)^{2} d\lambda.
\end{equation*}
Through the theory of \cite{LS} and under our assumptions of $V$, one can construct a new conjugate operator $\tilde{A}=A+B$ ( Larenas-Soffer conjugate operator ), where $B$ is the limit in $\mathcal{B}(L^{2}, L^{2})$ as follows,
\begin{equation*}
  B=s-\lim_{t\rightarrow\infty}\int_{0}^{t}e^{-isH}P_{ac}KP_{ac}e^{isH}ds.
\end{equation*}
It's easy to check the bounded operator $B$ exists and be well defined by local decay estimate ( Theorem \ref{local decay thm} ) with the potential $V(x)$ satisfies the same condition as in Theorem \ref{local decay thm}. Further, the Larenas-Soffer conjugate operator $\tilde{A}$ satisfy
$[\tilde{A}, H]=4H$ and $H\in C^{1}(\tilde{A})$.

%\begin{remark}\label{6.4}
%One can use the commutator method to get high energy decay estimate of $R^{(k)}(H; z)$. Since $H\in C^{1}(\widetilde{A})$ and furthermore, we have
%\begin{equation*}
%  zR^{\prime}(H; z)=-R(H_0; z)+\frac{1}{4}[\widetilde{A},\,\, R(H; z)].
%\end{equation*}
%While $H\in C^{1}(\widetilde{A})$ then $[\widetilde{A}, R(H; z)]$ is uniformly bounded for $z\in\mathbb{C}\setminus(\Sigma\cup[0,\infty))$. For higher order derivative of $R(z)$, similarly we have
%\begin{equation*}
%  zR^{(k)}(H; z)=-kR^{(k-1)}(H; z)+\frac{1}{4}[\widetilde{A},\,\, R^{(k-1)}(H; z)],
%\end{equation*}
%so one can use the inductive equality to get higher order decay which can easier the proof of Proposition \ref{Perturbed Resolvent} with the case of $l=0$.
%\end{remark}

\begin{theorem}\label{-1/2decay}
Under the same assumptions as given in Theorem \ref{local decay thm}, and $V$ satisfies $x\cdot\nabla V\in L^{\infty}(\mathbb{R}^{d})$. Then for $u\in\mathcal{D}(H)\cap\mathcal{D}(A)$ satisfying $\langle x\rangle^{\sigma}u\in L^{2}(\mathbb{R}^{d})$ with $\sigma>1/2$, we have
\begin{equation}\label{-1/2}
|\psi_{P_{ac}u}(t)|=O( t^{-1/2}), \,\, t\rightarrow \infty.
\end{equation}
\end{theorem}
\begin{proof}
By \cite[Corollary 8.2]{GLS}, we have
\begin{equation*}
  |\psi_{P_{ac}u}(t)|\leq c\langle t\rangle^{-\frac{1}{2}}\|\psi_{P_{ac}u}(t)\|_{L^{2}(\mathbb{R})}^{\frac{1}{2}}
                              \|t\psi^{\prime}_{P_{ac}u}(t)\|_{L^{2}(\mathbb{R})}^{\frac{1}{2}}.
\end{equation*}
Now, the aim is to check that $\psi_{P_{ac}u}(t)$ and $t\psi^{\prime}_{P_{ac}u}(t)$ are in $L^{2}(\mathbb{R})$.
\begin{equation*}
 \begin{split}
  \|\psi_{P_{ac}u}(t)\|_{L^{2}(\mathbb{R})}
   & = \|\langle P_{ac}u, e^{itH}P_{ac}u \rangle\|_{L^{2}(\mathbb{R})} \\
   & = \Big\|\big\langle\langle x\rangle^{\sigma} P_{ac}u, \langle x\rangle^{-\sigma}e^{itH}P_{ac}u \big\rangle\Big\|_{L^{2}(\mathbb{R})} \\
   & \leq \Big\|\|\langle x\rangle^{\sigma}P_{ac}u\|_{L^{2}_{x}}\|\langle    x\rangle^{-\sigma}e^{itH}P_{ac}u\|_{L^{2}_{x}}\Big\|_{L^{2}(\mathbb{R})}\\
   & \leq c\|\langle x\rangle^{\sigma}P_{ac}u\|_{L^{2}_{x}}\|u\|_{L^{2}_{x}}.
 \end{split}
\end{equation*}
For $t\psi^{\prime}_{P_{ac}u}(t)$, since $H\in C^{1}(A)$, through the theory of \cite{LS}, one can construct a new conjugate operator $\tilde{A} $(Larenas-Soffer conjugate operator) such that $H\in C^{1}(\tilde{A})$ and $[H, i\tilde{A}]=4H$. Further, the domain of $\tilde{A}$ satisfies $\mathcal{D}(\tilde{A})=\mathcal{D}(A)$. Since
\begin{equation*}
  \begin{split}
     4it\psi^{\prime}_{P_{ac}u}(t)
       & = 4i\langle P_{ac}u, itHe^{itH}P_{ac}u\rangle=\langle P_{ac}u, 4tHe^{itH}P_{ac}u\rangle  \\
       & =\langle P_{ac}u, [e^{itH},\tilde{A}]P_{ac}u\rangle\\
       & =\langle  P_{ac}u,  e^{itH}\tilde{A}P_{ac}u \rangle-\langle e^{-itH}\tilde{A}P_{ac}u, P_{ac}u \rangle\\
       & =\big\langle \langle        x\rangle^{\sigma}P_{ac}u, \langle x\rangle^{-\sigma}e^{itH}\tilde{A}P_{ac}u \big\rangle-\big\langle  \langle x\rangle^{-\sigma}e^{-itH}\tilde{A}P_{ac}u,  \langle x\rangle^{\sigma}P_{ac}u \big\rangle
  \end{split}
\end{equation*}
And then
\begin{equation*}
  \begin{split}
    \|t\psi^{\prime}_{P_{ac}u}(t)\|_{L^{2}(\mathbb{R})}
      & \leq \frac{1}{2}\Big\|\big\langle  \langle x\rangle^{-\sigma}e^{-itH}\tilde{A}P_{ac}u,  \langle x\rangle^{\sigma}P_{ac}u       \big\rangle\Big\|_{L^{2}(\mathbb{R})}\\
      & \leq\frac{1}{2}\Big\|\|\langle x\rangle^{\sigma}P_{ac}u\|_{L^{2}_{x}}\|\langle       x\rangle^{-\sigma}e^{itH}\tilde{A}P_{ac}u\|_{L^{2}_{x}}\Big\|_{L^{2}(\mathbb{R})}\\
      & \leq \frac{c}{2}\|\tilde{A}P_{ac}u\|_{L^{2}_{x}}\|\langle x\rangle^{\sigma}u\|_{L^{2}_{x}}.
  \end{split}
\end{equation*}
Here we use the local decay estimates of $H$. Furthermore,
\begin{equation*}
  |\psi_{P_{ac}u}(t)|\leq c\langle t\rangle^{-\frac{1}{2}}\|\langle   x\rangle^{\sigma}u\|_{L^{2}_{x}}\|Au\|^{\frac{1}{2}}_{L^{2}_{x}}\|u\|^{\frac{1}{2}}_{L^{2}_{x}}.
\end{equation*}
\end{proof}
\begin{remark}\label{highenergy LS}
For high energy of $H$, we can get faster decay for $\psi_{P_{ac}u}(t)$, and this coincide with the perturbation method. In fact, we have
\begin{equation*}
  |t\langle P_{ac}u, \, \chi_{\geq1}(H)e^{-itH}P_{ac}u\rangle|=|\langle\chi_{\geq1}(H)[H, \tilde{A}]^{-1}P_{ac}u, \,  [e^{-itH}, \tilde{A}]P_{ac}u\rangle|.
\end{equation*}
It is easy to check $$\chi_{\geq1}(H)[H, \tilde{A}]^{-1}=\chi_{\geq1}(H)(4H)^{-1}\in C^{1}(\tilde{A})$$ and bounded, so that operator $\chi_{\geq1}(H)[H, \tilde{A}]^{-1}$ keeps $\mathcal{D}(\tilde{A})$ the domain of $\tilde{A}$.

Let $\tilde{u}=\chi_{\geq1}(H)[H, \tilde{A}]^{-1}P_{ac}u$, and then we have
\begin{equation*}
 \begin{split}
  |\langle\tilde{u}, [e^{-itH}, \tilde{A}]P_{ac}u\rangle|
&\lesssim |\big\langle\langle x\rangle^{\sigma}\tilde{u}, \, \langle x\rangle^{-\sigma}e^{-itH}\tilde{A}P_{ac}u\big\rangle|
         +|\big\langle\langle x\rangle^{\sigma}u, \langle x\rangle^{-\sigma}e^{itH}P_{ac}\tilde{A}\tilde{u}\big\rangle|\\
&\lesssim \|\tilde{A}P_{ac}u\|_{L^{2}(\mathbb{R}^{d})}\|\langle x\rangle^{\sigma}u\|_{L^{2}(\mathbb{R}^{d})}<\infty.
 \end{split}
\end{equation*}
We can apply the same argument to get much higher decay of high energy part by an iterated process.
\end{remark}

From the expansion of the perturbed resolvent, and through the classical Jensen-Kato's work \cite{J1}, we expect the time decay rate should be $-5/4$ in the 3-dimension. Next, we improve our results by iteration of the previous argument. For higher dimensions $d\geq5$, we also can improve the decay rate through the same way as what we did in the 3-dimensional case.
\begin{proposition}
Suppose $V(x)$ satisfies $x\cdot\nabla (x\cdot\nabla V)\in L^{\infty}(\mathbb{R}^{d})$, then $H\in C^{2}(A)$.
\end{proposition}
\begin{proof}
By the definition of $C^{2}(A)$, we need to prove
\begin{equation*}
G(t)=e^{itA}(H-i)^{-1}e^{-itA}\in C^{2}(\mathcal{B}(L^{2}(\mathbb{R}^{3}))).
\end{equation*}
Since $H\in C^{1}(A)$, so we only need to check $G^{\prime\prime}(t)$ is continuous. By \eqref{G'} we have
\begin{equation*}
  \begin{split}
   G^{\prime\prime}(t)
   & = G_{0}^{\prime\prime}(t)[1+G_{0}(t)\tilde{V}(t)]^{-1}-
          2G_{0}^{\prime}(t)[1+G_{0}(t)\tilde{V}t)]^{-2}[G_{0}^{\prime}(t)\tilde{V}(t)+G_{0}(t)\tilde{V}^{\prime}(t)] \\
   & + 2G_{0}(t)[1+G_{0}(t)\tilde{V}(t)]^{-3}[G_{0}^{\prime}(t)\tilde{V}(t)+G_{0}(t)\tilde{V}^{\prime}(t)]^{2}\\
   & -G_{0}(t)[1+G_{0}(t)\tilde{V}(t)]^{-2}[G_{0}^{\prime\prime}(t)\tilde{V}(t)+2G^{\prime}_{0}(t)\tilde{V}^{\prime}(t)+G(t)\tilde{V}^{\prime\prime}(t)].
  \end{split}
\end{equation*}
Thus, if $G_{0}^{\prime\prime}(t)$ and $\tilde{V}^{\prime\prime}(t)$ are continuous then the proof done.
For $G_{0}^{\prime\prime}(t)$, it's automatically continuous since we have already proved that $H_{0}\in C^{k}(A)$. For $\tilde{V}^{\prime\prime}(t)$, since
\begin{equation*}
  \tilde{V}^{\prime\prime}(t)=e^{itA}[iA,x\cdot\nabla V]e^{-itA}=e^{itA}[x\cdot\nabla(x\cdot\nabla V)]e^{-itA},
\end{equation*}
thus $x\cdot\nabla (x\cdot\nabla V)\in L^{\infty}$ implies that $\tilde{V}^{\prime\prime}(t)\in C\big(\mathcal{B}(L^{2}(\mathbb{R}^{3}))\big)$.
\end{proof}
\begin{remark}
The condition $H\in C^{1}(A)$ can be replaced by $x\cdot\nabla V\in L^{\infty}$. Further, since for $k\in \mathbb{N}^{+}$ we have $H_{0}\in C^{k}(A)$, and then we can prove that $H\in C^{k}(A)$ by the same process with the assumption that $(x\cdot\nabla)^{k-1} (x\cdot\nabla V)\in L^{\infty}$. And under this condition and some suitable assumption of the vector $u$ ,  one can get the time decay rate of the so-called pointwise decay estimates to any order by the commutator method \cite{GLS,LS} and the following argument.
\end{remark}
\begin{proposition}\label{-3}
Under the same assumptions as given in Theorem \ref{local decay thm}, and $V$ satisfies $x\cdot\nabla (x\cdot\nabla V)\in L^{\infty}$. Then for $u\in\mathcal{D}(H)$ with the form of $u=|H|^{\frac{1}{2}}\tilde{u}$, $\tilde{u}\in\mathcal{D}(A^{2})\cap\mathcal{D}(H)$, and $\langle x \rangle^{\sigma}\tilde{u}, \langle x \rangle^{\sigma}A\tilde{u} \in L^{2}(\mathbb{R}^{3})$ with $\sigma>\frac{1}{2}$, we have
\begin{equation}\label{-3/2}
  |\psi_{P_{ac}u}|=O(t^{-3/2}),\,\,t\rightarrow \infty.
\end{equation}
\end{proposition}
\begin{proof}
By \cite[Corollary 8.2]{GLS}, we have
\begin{equation*}
  |t\psi_{P_{ac}u}(t)|\leq c\langle t\rangle^{-\frac{1}{2}}\|t\psi_{P_{ac}u}(t)\|_{L^{2}(\mathbb{R})}^{\frac{1}{2}}
                              \|t^{2}\psi^{\prime}_{P_{ac}u}(t)\|_{L^{2}(\mathbb{R})}^{\frac{1}{2}}.
\end{equation*}
Now, our target is to prove that $t\psi_{P_{ac}u}(t)$ and $t^{2}\psi^{\prime}_{P_{ac}u}(t)$ belong to $L_{t}^{2}(\mathbb{R})$. For this we use the LS conjugate operator $\tilde{A}$. For $t\psi_{P_{ac}u}(t)$, since
\begin{equation*}
  \begin{split}
    i4t\psi_{P_{ac}u}(t)
      & =i\langle P_{ac}u, 4te^{itH}P_{ac}u\rangle=i\langle P_{ac}|H|^{\frac{1}{2}}\tilde{u}, 4te^{itH}P_{ac}|H|^{\frac{1}{2}}\tilde{u}  \rangle \\
      & =i\langle P_{ac}\tilde{u}, 4tHe^{itH}{\rm sgn}(H)P_{ac}\tilde{u} \rangle=\langle P_{ac}\tilde{u}, [e^{itH},\tilde{A}]{\rm sgn}(H)P_{ac}\tilde{u}       \rangle\\
      & =\langle P_{ac}\tilde{u}, e^{itH}\tilde{A}{\rm sgn}(H)P_{ac}\tilde{u}\rangle- \langle e^{-itH}\tilde{A}P_{ac}\tilde{u},{\rm sgn}(H)P_{ac}\tilde{u}\rangle
  \end{split}
\end{equation*}
And then
\begin{equation*}
  \begin{split}
    \|t\psi_{P_{ac}u}(t)\|_{L_{t}^{2}}
    & \leq \frac{1}{2}\Big\|\langle e^{-itH}\tilde{A}P_{ac}\tilde{u},{\rm sgn}(H)P_{ac}\tilde{u}\rangle \Big\|_{L_{t}^{2}}\\
    & =\frac{1}{2}\Big\|\big\langle \langle x\rangle^{-\sigma}e^{-itH}\tilde{A}P_{ac}\tilde{u},\langle x\rangle^{\sigma}{\rm sgn}(H)P_{ac}\tilde{u}\big\rangle     \Big\|_{L_{t}^{2}}\\
    & \leq \frac{1}{2}\Big\|\|\langle x\rangle^{-\sigma}e^{-itH}\tilde{A}P_{ac}\tilde{u}\|_{L_{x}^{2}}\|\langle     x\rangle^{\sigma}{\rm sgn}(H)P_{ac}\tilde{u}\|_{L_{x}^{2}}\Big\|_{L_{t}^{2}}\\
    & \leq c\frac{1}{2}\|\langle x\rangle^{\sigma}\tilde{u}\|_{L_{x}^{2}}\|\tilde{A}P_{ac}\tilde{u}\|_{L_{x}^{2}}\leq c\frac{1}{2}\|\langle x\rangle^{\sigma}\tilde{u}\|_{L_{x}^{2}}\|A\tilde{u}\|_{L_{x}^{2}}.
  \end{split}
\end{equation*}

For $t^{2}\psi^{\prime}_{P_{ac}u}(t)$, since
\begin{equation*}
  \begin{split}
    i16t^{2}\psi^{\prime}_{P_{ac}u}(t)
      & = \langle P_{ac}\tilde{u}, 4tH[e^{itH},\tilde{A}]{\rm sgn}(H)P_{ac}\tilde{u} \rangle\\
      & = \langle P_{ac}\tilde{u}, \tilde{\mathcal{A}}^{2}(e^{itH}){\rm sgn}(H)P_{ac}\tilde{u}\rangle-
          \langle P_{ac}\tilde{u}, 4\tilde{\mathcal{A}}(e^{itH}){\rm sgn}(H)P_{ac}\tilde{u} \rangle.
  \end{split}
\end{equation*}
Here $\tilde{\mathcal{A}}(e^{itH})=i[e^{itH},\tilde{A}]$ and
\begin{equation*}
  \begin{split}
  \langle P_{ac}\tilde{u}, \tilde{\mathcal{A}}^{2}(e^{itH})P_{ac}\tilde{u}\rangle
  &=\langle e^{-itH}P_{ac}\tilde{u}, \tilde{A}^{2}P_{ac}\tilde{u} \rangle\\
  &-2\langle \tilde{A}P_{ac}\tilde{u}, e^{itH}\tilde{A}P_{ac}\tilde{u}\rangle
   +\langle \tilde{A}^{2}P_{ac}\tilde{u}, e^{itH}P_{ac}\tilde{u} \rangle.
  \end{split}
\end{equation*}
So the second term just the same as $4t\psi_{P_{ac}u}(t)$. Further, for the first term
\begin{equation*}
  \begin{split}
  \|\langle P_{ac}\tilde{u}, \tilde{\mathcal{A}}^{2}(e^{itH})P_{ac}\tilde{u}\rangle\|_{L_{t}^{2}}
     & \leq 2\|\langle \tilde{A}^{2}P_{ac}\tilde{u}, e^{itH}P_{ac}\tilde{u} \rangle\|_{L_{t}^{2}}
            +2\|\langle \tilde{A}P_{ac}\tilde{u}, e^{itH}\tilde{A}P_{ac}\tilde{u}\rangle\|_{L_{t}^{2}} \\
     & \leq 2\Big\|\big\langle \langle x\rangle^{-\sigma}e^{-itH}\tilde{A}^{2}P_{ac}\tilde{u}, \langle x\rangle^{\sigma}P_{ac}\tilde{u}      \big\rangle\Big\|_{L_{t}^{2}}\\
     & \quad +2\Big\|\big\langle  \langle x\rangle^{-\sigma}e^{-itH}\tilde{A}P_{ac}\tilde{u}, \langle      x\rangle^{\sigma}\tilde{A}P_{ac}\tilde{u}\big\rangle\Big\|_{L_{t}^{2}}\\
     & \leq 2\Big\|\| \langle x\rangle^{-\sigma}e^{-itH}\tilde{A}^{2}P_{ac}\tilde{u}\|_{L_{x}^{2}}\|\langle x\rangle^{\sigma}P_{ac}\tilde{u}\|_{L_{x}^{2}}\Big\|_{L_{t}^{2}}\\
     & \quad +2\Big\|\|\langle x\rangle^{-\sigma}e^{-itH}\tilde{A}P_{ac}\tilde{u}\|_{L_{x}^{2}}\|\langle      x\rangle^{\sigma}\tilde{A}P_{ac}\tilde{u}\|_{L_{x}^{2}}\Big\|_{L_{t}^{2}}\\
     & \leq 2c\|\tilde{A}^{2}P_{ac}\tilde{u}\|_{L_{x}^{2}}\|\langle x\rangle^{\sigma}P_{ac}\tilde{u}\|_{L_{x}^{2}}
     +2c\|\tilde{A}P_{ac}\tilde{u}\|_{L_{x}^{2}}\|\langle x\rangle^{\sigma}\tilde{A}P_{ac}\tilde{u}\|_{L_{x}^{2}}\\
    & \leq 2c\|A^{2}\tilde{u}\|_{L_{x}^{2}}\|\langle x\rangle^{\sigma}\tilde{u}\|_{L_{x}^{2}}
     +2c\|A\tilde{u}\|_{L_{x}^{2}}\|\langle x\rangle^{\sigma}A\tilde{u}\|_{L_{x}^{2}}.
  \end{split}
\end{equation*}
Finally,
\begin{equation*}
  |\psi_{P_{ac}u}(t)|\leq c\langle t \rangle^{-\frac{3}{2}}\Big\{\|\langle x\rangle^{\sigma}\tilde{u}\|^{\frac{1}{2}}_{L_{x}^{2}}
   \big(\|A^{2}\tilde{u}\|^{\frac{1}{2}}_{L_{x}^{2}}+\|A\tilde{u}\|^{\frac{1}{2}}_{L_{x}^{2}}\big)
   +\|A\tilde{u}\|^{\frac{1}{2}}_{L_{x}^{2}}\|\langle x\rangle^{\sigma}A\tilde{u}\|^{\frac{1}{2}}_{L_{x}^{2}}\Big\}.
\end{equation*}
\end{proof}
\begin{theorem}\label{-5}
Under the same conditions of Proposition \ref{-3}, let $u=|H|^{\frac{3}{8}}\tilde{u}$, then
\begin{equation}\label{-5/4}
  |\psi_{P_{ac}u}|=O(t^{-5/4}),\,\,t\rightarrow \infty.
\end{equation}
\end{theorem}
\begin{proof}
For $0\leq Re(z)\leq1$ and $z\in\mathbb{C}$, define function
\begin{equation*}
\phi(z)=\big|\langle |H|^{z/2}P_{ac}u, e^{itH}|H|^{z/2}P_{ac}u\rangle\big|.
\end{equation*}
Here we use that $HP_{ac}\geq0$. While through \eqref{-1/2} and \eqref{-3/2} we know
\begin{equation*}
  \phi(0)\leq c_{u}\langle t\rangle^{-1/2},\quad  \phi(1)\leq C_{u}\langle t\rangle^{-3/2},
\end{equation*}
and then by Hadamard's three line lemma, we have
\begin{equation*}
  |\psi_{P_{ac}u}(t)|=\phi(\frac{3}{4})\leq C^{\prime}_{u}(\langle t\rangle^{-\frac{1}{2}})^{1-\frac{3}{4}}(\langle   t\rangle^{-\frac{3}{2}})^{\frac{3}{4}}=C^{\prime}_{u}\langle t \rangle^{-\frac{5}{4}}.
\end{equation*}
\end{proof}
%Notice that, under the zero regular assumption that $1+vR_{0}v$ is invertible and we have clearly expansion around zero, so that for any $\epsilon>0$, $|H|^{-\frac{3}{8}+\varepsilon}\langle x\rangle^{-\frac{3}{2}+\epsilon}$ is $L^{2}(\mathbb{R}^{3})\rightarrow L^{2}(\mathbb{R}^{3})$ bounded
%since $|H_{0}|^{-\frac{3}{8}+\varepsilon}\langle x\rangle^{-\frac{3}{2}+\epsilon}$ is $L^{2}(\mathbb{R}^{3})\rightarrow L^{2}(\mathbb{R}^{3})$ bounded. Hence, Theorem \ref{-5} can be rewritten as follows.
%\begin{theorem}
%Let $H\in C^{2}(A)$, for any $u\in\mathcal{D}(A^{2})\cap\mathcal{D}(H)$ such that $\langle x \rangle^{\sigma}u, \langle x \rangle^{\sigma}Au \in L^{2}(\mathbb{R}^{3}), \sigma>\frac{1}{2}+4$, then
%\begin{equation}
%  |\psi_{P_{ac}u}|\leq C_{u}\langle t \rangle^{-\frac{5}{4}}.
%\end{equation}
%\end{theorem}
\begin{remark}
Notice that we use Jensen-Kato's strategy to obtain the Jensen-Kato type decay estimate \eqref{JK3}, \eqref{JKd} and the local decay estimate \eqref{local decay} for $H=(-\Delta)^2+V$ under the absence of positive embedded eigenvalues assumption. In fact, in case $\lambda>0$ is an embedded eigenvalue,  the same estimates hold replacing $H$ by $\bar{H}$. $\bar{H}:=\bar{P}H\bar{P}$ and $\bar{P}=1-P_{eign}$ where $P_{eign}$ denotes the projection onto the eigenspace. See Remark \ref{remove absence}. We remind the reader of the Fermi Golden Rule (see e.g. \cite{RS2, Golenia}). The Fermi Golden Rule states that, for simplicity isolated and simple, an embedded eigenvalue is unstable under a perturbation $V$ provided
\begin{equation*}
  \rm{Im}\Big(\lim_{\epsilon\downarrow 0}\langle V\psi, \bar{P}(\bar{H}-\lambda-i\epsilon)^{-1}\bar{P}V\psi\rangle\Big)\neq0.
\end{equation*}
The existence of the limit can be inferred from the limiting absorption principle. The limiting absorption principle can be deduced by positive commutator estimates, see e.g. \cite{ABG, Mourre, Georg-Gera-Moller}, provided there exists an operator $A$ such that $H$ and $A$ satisfy a Mourre estimate near $\lambda>0$ and $(\bar{H}-i)^{-1}$ admits two bounded commutators with $A$. Precisely, one needs the revised condition that $\bar{P}H\bar{P}$ is of class $C^{2}(A)$, see e.g. \cite{ABG, Georg-Gera-Moller} or in the section above. Since we localize energy around $\lambda$, Mourre estimate implies the local decay estimate around the positive embedded eigenvalue $\lambda$. Here, we remark that Ben-Artzi and Devinatz \cite{Ben-Devi} have established the limiting absorption principle for general Schr\"odinger type operator $H_{0}+V$ with short range potential.
\end{remark}

The abstract theory also can deal with functions of $H$. We will apply this theory to the operator $\sqrt{H+m^2}$ ($m>0$ large enough such that $H+m^2\geq0$) to get the Jensen-Kato type decay estimates. The difficulty is to prove $\sqrt{H+m^2}\in C^{1}(\tilde{A})$.
\begin{lemma}
$\sqrt{H+m^2}\in C^{1}(\tilde{A}).$
\end{lemma}
\begin{proof}
By the definition of $C^{1}(\tilde{A})$, we need to show the operator valued function
$$g(t)=e^{it\tilde{A}}(\sqrt{H+m^2}-i)^{-1}e^{-it\tilde{A}}\in C^{1}\big(\mathcal{B}(L^{2}(\mathbb{R}^{d}))\big),\,\, t\in \mathbb{R}.$$
First, since $\tilde{A}$ is self-adjoint then
\begin{equation*}
  \|g(t)\|_{L^{2}\rightarrow L^{2}}
  =\|(\sqrt{H+m^2}-i)^{-1}\|_{L^{2}\rightarrow L^{2}}
  \leq 1.
\end{equation*}
Further, since we have proven that $H\in C^{1}(\tilde{A})$ and,
\begin{equation*}
 \begin{split}
  g^{\prime}(t)&=-e^{it\tilde{A}}[(\sqrt{H+m^2}-i)^{-1},\,\, i\tilde{A}]e^{-it\tilde{A}}\\
               &=e^{it\tilde{A}}(\sqrt{H+m^2}-i)^{-1}[\sqrt{H+m^2},\,\, i\tilde{A}](\sqrt{H+m^2}-i)^{-1}e^{-it\tilde{A}}.
 \end{split}
\end{equation*}
We claim that $$e^{it\tilde{A}}[\sqrt{H+m^2},\,\, i\tilde{A}]e^{-it\tilde{A}}\in C(\mathcal{B}(L^{2}(\mathbb{R}^{d}))).$$
In fact, by the square root formula \cite[P. 316]{RS1}
\begin{equation*}
  \sqrt{H+m^2}=\frac{1}{\pi}\int_{0}^{\infty}\frac{\lambda^{-1/2}}{\lambda+H+m^2}(H+1)d\lambda,
\end{equation*}
and then
\begin{equation*}
 \begin{split}
  [\sqrt{H+m^2},i\tilde{A}]&=\frac{1}{\pi}\int_{0}^{\infty}\lambda^{-1/2}[\frac{1}{\lambda+H+m^2}(H+m^2),\,\, i\tilde{A}]d\lambda\\
                         &=\frac{1}{\pi}\int_{0}^{\infty}-\lambda^{-1/2}[\frac{1}{\lambda+H+m^2},\,\, i\tilde{A}]d\lambda\\
                         &=\frac{1}{\pi}\int_{0}^{\infty}\lambda^{-1/2}\frac{1}{\lambda+H+m^2}[H,\,\, i\tilde{A}]\frac{1}{\lambda+H+m^2}d\lambda
 \end{split}
\end{equation*}
Since $H\in C^{1}(\tilde{A})$, thus $g^{\prime}(t)$ is bounded.
\end{proof}
Following the same argument as in Remark \ref{highenergy LS} for $\sqrt{H+m^2}$, we have
\begin{theorem}
Under the same assumptions as given in Theorem \ref{local decay thm}, and $V$ satisfies $x\cdot\nabla V\in L^{\infty}(\mathbb{R}^{d})$. Then for $u\in\mathcal{D}(\sqrt{H+m^2})\cap\mathcal{D}(A)$ and $\langle x\rangle^{\sigma}u\in L^{2}(\mathbb{R}^{d})$ with $\sigma>1/2$, we have
\begin{equation}
|\tilde{\psi}_{P_{ac}u}(t)|=O(t^{-1/2}),\,\,t\rightarrow \infty,
\end{equation}
with $\tilde{\psi}_{u}(t)=\langle u, e^{-it\sqrt{H+m^2}}u\rangle$.
\end{theorem}

\begin{proof}
Denote $\tilde{H}=\sqrt{H+m^2}$, and then we divided $|\tilde{\psi}_{P_{ac}u}(t)|$ into high energy part and low energy part:
\begin{equation*}
  |\tilde{\psi}_{P_{ac}u}(t)|=|\langle P_{ac}u,\, \chi_{\geq1}(H)e^{-it\tilde{H}}P_{ac}u\rangle+\langle P_{ac}u,\, \chi_{<1}(H)e^{-it\tilde{H}}P_{ac}u\rangle|.
\end{equation*}

For the high energy part $\langle P_{ac}u,\, \chi_{\geq1}(H)e^{-it\tilde{H}}P_{ac}u\rangle$, we have
\begin{equation*}
  \begin{split}
  &\big|~t~\langle P_{ac}u,\, \chi_{\geq1}(H)e^{-it\tilde{H}}P_{ac}u\rangle\big|
  =\big|\langle P_{ac}u,\, \frac{\chi_{\geq1}(H)}{[\tilde{H}, \tilde{A}]}~t[\tilde{H}, \tilde{A}]e^{-it\tilde{H}}P_{ac}u\rangle\big|\\
  = &\big|\langle P_{ac}u,\, \frac{\chi_{\geq1}(H)}{[\tilde{H}, \tilde{A}]}[e^{-it\tilde{H}}, \tilde{A}]P_{ac}u\rangle\big|
  = \big|\langle P_{ac}u,\, \frac{\chi_{\geq1}(H)\tilde{H}}{4H}[e^{-it\tilde{H}}, \tilde{A}]P_{ac}u\rangle\big|\\
  \leq &\big|\langle \frac{\chi_{\geq1}(H)\tilde{H}}{4H}P_{ac}u,\, e^{-it\tilde{H}}\tilde{A}P_{ac}u\rangle\big|+\big|\langle \tilde{A}\frac{\chi_{\geq1}(H)\tilde{H}}{4H}P_{ac}u,\, e^{-it\tilde{H}}P_{ac}u\rangle\big|\\
  \lesssim & \|u\|_{L^{2}(\mathbb{R}^{d})}\Big(\|\tilde{A}u\|_{L^{2}(\mathbb{R}^{d})}+\big\|\tilde{A}\frac{\chi_{\geq1}(H)\tilde{H}}{4H}P_{ac}u\big\|_{L^{2}(\mathbb{R}^{d})}\Big)
  <\infty.
  \end{split}
\end{equation*}
Here we used the fact that $\frac{\chi_{\geq1}(H)\tilde{H}}{4H}$ is bounded in $L^{2}(\mathbb{R}^{d})$. So that for the high energy part we have
\begin{equation}
  \big|\langle P_{ac}u,\, \chi_{\geq1}(H)e^{-it\tilde{H}}P_{ac}u\rangle\big|=O(|t|^{-1}),\,\,t\rightarrow\infty.
\end{equation}

For the low energy part, we claim that
\begin{equation}
 \big|\langle P_{ac}u,\, \chi_{<1}(H)e^{-it\tilde{H}}P_{ac}u\rangle\big|=O(|t|^{-1/2}),\,\,t\rightarrow\infty.
\end{equation}
Note that $\big|\langle P_{ac}u,\, \chi_{<1}(H)e^{-it\tilde{H}}P_{ac}u\rangle\big|=\big|\langle P_{ac}u,\, \chi_{<1}(H)e^{-it(\tilde{H}-m)}P_{ac}u\rangle\big|$, and then we use the same approach as in Theorem \ref{-1/2decay} to show:
\begin{equation*}
  \big|\langle P_{ac}u,\, \chi_{<1}(H)e^{-it(\tilde{H}-m)}P_{ac}u\rangle\big|=O(|t|^{-1/2}),\,\,t\rightarrow\infty.
\end{equation*}

Now, we need only to show that
\begin{equation}\label{tilde1}
\big|\langle P_{ac}u,\, \chi_{<1}(H)e^{-it(\tilde{H}-m)}P_{ac}u\rangle\big|\in L^{2}_{t}(\mathbb{R}),
\end{equation}
\begin{equation}\label{tilde2}
\big|~t~\langle P_{ac}u,\, \chi_{<1}(H)(\tilde{H}-m)e^{-it(\tilde{H}-m)}P_{ac}u\rangle\big|\in L^{2}_{t}(\mathbb{R}).
\end{equation}

In fact, we have
\begin{equation*}
  \begin{split}
       &\sup_{0\leq|z|<\infty}\big\|\langle x\rangle^{-\sigma}\chi_{<1}(H)(\tilde{H}-m-z)^{-1}P_{ac}\langle x\rangle^{-\sigma}\big\|_{L^{2}\rightarrow L^{2}}\\
  \leq &\sup_{0\leq|z|<1}\big\|\langle x\rangle^{-\sigma}\chi_{<1}(H)(\tilde{H}-m-z)^{-1}P_{ac}\langle x\rangle^{-\sigma}\big\|_{L^{2}\rightarrow L^{2}}\\
       &+\sup_{1\leq|z|<\infty}\big\|\langle x\rangle^{-\sigma}\chi_{<1}(H)(\tilde{H}-m-z)^{-1}P_{ac}\langle x\rangle^{-\sigma}\big\|_{L^{2}\rightarrow L^{2}}\\
    =  &\sup_{0\leq|z|<1}\big\|\langle x\rangle^{-\sigma}\chi_{<1}(H)(\tilde{H}+m+z)(H+m^2-(m+z)^2)^{-1}P_{ac}\langle x\rangle^{-\sigma}\big\|_{L^{2}\rightarrow L^{2}}\\
       &+\sup_{1\leq|z|<\infty}\big\|\langle x\rangle^{-\sigma}\chi_{<1}(H)(\tilde{H}-m-z)^{-1}P_{ac}\langle x\rangle^{-\sigma}\big\|_{L^{2}\rightarrow L^{2}}<\infty.\\
  \end{split}
\end{equation*}
So that by the Corollary of \cite[P.146]{RS2}, we have the local decay for the low energy part:
\begin{equation}\label{localdecaytilde}
  \int_{\mathbb{R}}\big\|\langle x\rangle^{-\sigma}\chi_{<1}(H)e^{-it(\tilde{H}-m)}P_{ac}\big\|^{2}_{L^{2}\rightarrow L^{2}}dt<\infty.
\end{equation}
Thus \eqref{tilde1} follows by \eqref{localdecaytilde}.

Observe that $[\tilde{H}, \tilde{A}]=\tilde{H}-m^2\tilde{H}^{-1}$ since we have proved $\tilde{H}\in C^{1}(\tilde{A})$. For \eqref{tilde2}, we have
\begin{equation*}
 \begin{split}
   &\big|~t~\langle P_{ac}u,\, \chi_{<1}(H)(\tilde{H}-m)e^{-it(\tilde{H}-m)}P_{ac}u\rangle\big|\\
  =&\big|\langle P_{ac}u,\, \chi_{<1}(H)(\tilde{H}-m)[\tilde{H}-m, \tilde{A}]^{-1}~t~[\tilde{H}-m, \tilde{A}]e^{-it(\tilde{H}-m)}P_{ac}u\rangle\big|\\
  =&\big|\langle P_{ac}u,\, \chi_{<1}(H)\tilde{H}(\tilde{H}+m)^{-1}[e^{-it(\tilde{H}-m)}, \tilde{A}]P_{ac}u\rangle\big|.\\
 \end{split}
\end{equation*}
Following the same argument as what we did for $-it\psi^{\prime}_{P_{ac}u}(t)$ in Theorem \ref{-1/2decay}, and the estimate \eqref{localdecaytilde} again yields that the \eqref{tilde2} holds.

\end{proof}

\section{Appendix}
\setcounter{equation}{0}
\noindent{\bf \large{A1}: Proof of Theorem \ref{resolventexpansion3}}
\begin{lemma}\label{inverse formula}
(\cite[Corollary 2.2]{JN2, JN}) Let $F\subset\mathbb{C}$ have zero as an accumulation point. Let $T(z), z\in F$ be a family of bounded operators of the form
\begin{equation*}
  T(z)=T_{0}+zT_{1}(z)
\end{equation*}
with $T_{1}(z)$ uniformly bounded as $z\rightarrow0$. Suppose $0$ is an isolated point of the spectrum of $T_{0}$, and let $S$ be the corresponding Riesz projection. If $T_{0}S=0$, then for sufficient small $z\in F$ the operator $\widetilde T(z):S\mathcal{H}\rightarrow S\mathcal{H}$ defined by
\begin{equation}\label{equ:61}
  \widetilde{T}(z)=\frac{1}{z}(S-S(T(z)+S)^{-1}S)=\sum_{j=0}^{\infty}(-1)^{j}z^{j}S[T_{1}(z)(T_{0}+S)^{-1}]^{j+1}S
\end{equation}
is uniformly bounded as $z\rightarrow0$. The operator $T(z)$ has a bounded inverse in $\mathcal{H}$ if and only if $\widetilde{T}(z)$ has a bounded inverse in $S\mathcal{H}$, and in this case
\begin{equation}\label{equ:62}
  T(z)^{-1}=(T(z)+S)^{-1}+\frac{1}{z}(T(z)+S)^{-1}S\widetilde{T}(z)^{-1}S(T(z)+S)^{-1}.
\end{equation}
\end{lemma}

\begin{lemma}
 Let $\langle x\rangle^{\kappa} V\in L^{2}(\mathbb{R}^3)$ with some $\kappa$ ($\kappa>9$) sufficient large and let $p$ be the largest integer satisfying  $\kappa>2p+5$, then $M(\mu)-\frac{(1+i)\alpha}{8\pi}P\mu^{-1}-U$ is a uniformly bounded operator valued function in
\begin{equation*}
  E=\{\mu~|~\mathrm{Re}\, \mu>0, |\mu|\leq1\}
\end{equation*}
and has the following asymptotic expansion for small $\mu\in E$
\begin{equation}\label{equ:59}
  M(\mu)=\frac{(1+i)\alpha}{8\pi}P\mu^{-1}+\sum_{j=0}^{p-1}M_{j}\mu^{j}+\mu^{p}\mathfrak{R}_{0}(\mu,|x-y|)
\end{equation}
where $P=\alpha^{-1}\langle v, \cdot\rangle,\,\, \alpha=\|v\|^{2}$
and $M_{0}-U$, $M_{j},j=1,2,3,\cdots,p-1$ are integral operators given by the kernels
\begin{equation*}
  (M_{0}-U)(|x-y|)=-\frac{1}{8\pi}v(x)|x-y|v(y),
\end{equation*}
\begin{equation*}
  M_{j}(|x-y|)=\frac{(-1)^{j}(1-i^{j})}{8\pi(j+2)!}v(x)|x-y|^{j+1}v(y),
\end{equation*}
\begin{equation*}
  \mathfrak{R}_{0}(\mu,|x-y|)=v(x)\frac{1}{8\pi(p+2)!}\frac{1}{\mu^{2}|x-y|}\int_{0}^{\mu}(e^{it|x-y|}-e^{-t|x-y|})^{(p+3)}(\mu-t)^{p+2}dt v(y),
\end{equation*}
and $\mathfrak{R}_{0}(\mu)$ is uniformly bounded in norm. Further, the operators $M_{0}-U$, $M_{j}$ are compact.
\end{lemma}
\begin{proof}
By \eqref{equ:60} and \eqref{sqlap odd}, we have
\begin{equation*}
 \begin{split}
  M(\mu)&=U+v[\sum_{j=0}^{p+1}\frac{i^{j}-(-1)^{j}}{8\pi j!}\mu^{j-2}|x-y|^{j-1}+\mu^{p}\mathfrak{R}(\mu,|x-y|)]v\\
        &=U+v[\frac{1+i}{8\pi}\mu^{-1}+\sum_{j=2}^{p+1}\frac{i^{j}-(-1)^{j}}{8\pi j!}\mu^{j-2}|x-y|^{j-1}]v+v\mu^{p}\mathfrak{R}(\mu,|x-y|)v\\
        &=U+v[\frac{1+i}{8\pi}\mu^{-1}+\sum_{j=0}^{p-1}\frac{(-1)^{j}(1-i^{j})}{8\pi (j+2)!}\mu^{j}|x-y|^{j+1}]v+v\mu^{p}\mathfrak{R}(\mu,|x-y|)v
 \end{split}
\end{equation*}
thus we get \eqref{equ:59}. Using the Hilbert-Schmidt norm, it's trivial to check the operators $M_{0}-U, M_{j}$ are compact.
\end{proof}

%\begin{theorem}\label{resolventexpansion3}
%Assume that 0 is a regular point for $H$. Let $\langle x\rangle^{\kappa} V(x)\in L^{2}(\mathbb{R}^3)$ for $\kappa$ large enough and $p$ be the largest integer satisfying $\kappa>2p+5$, then there exists $\mu_0>0$ such that for $|\mu|\leq \mu_0$, $w(H-z)^{-1}w$ has the expansion in $\mathcal{B}\big(L^{2}(\mathbb{R}^3), \,L^{2}(\mathbb{R}^3)\big)$
%\begin{equation}\label{equ:63}
%  wR(H; z)w=U-Qm_{0}^{-1}Q+\sum_{j=1}^{p-1}M^{\prime}_{j}\mu^{j}+\mu^{p}\mathfrak{R}(\mu).
%\end{equation}
%Here $\mathfrak{R}(\mu)$ is uniformly bounded and the coefficients $M^{\prime}_{j}$ can be computed explicitly.
%\end{theorem}
\begin{proof}[\bf Proof of Theorem \ref{resolventexpansion3}]
Writing
\begin{equation}\label{equ:64}
  M(\mu)=\frac{(1+i)\alpha}{8\pi\mu}(P+\mu\widetilde{M}(\mu))
\end{equation}
where
\begin{equation}\label{equ:66}
 \begin{split}
  \widetilde{M}(\mu)&=\frac{8\pi}{(1+i)\alpha}\left(\sum_{j=0}^{p-1}M_{j}\mu^{j}+\mu^{p}\mathfrak{R}_{0}(\mu,|x-y|)\right)\\
                    &=\frac{8\pi}{(1+i)\alpha}(M_{0}+\mu M_{1}+\mu^{2}M_{2}(\mu)).
 \end{split}
\end{equation}
Applying Lemma \ref{inverse formula} to $P+\mu\widetilde{M}(\mu)$, for sufficiently small $\mu$ we have
\begin{equation}\label{equ:65}
  M(\mu)^{-1}=\frac{8\pi\mu}{(1+i)\alpha}\left[(1+\mu\widetilde{M}(\mu))^{-1}+\mu^{-1}(1+\mu\widetilde{M}(\mu))^{-1}Qm(\mu)^{-1}Q(1+\mu\widetilde{M}(\mu))^{-1}\right],
\end{equation}
where $Q=1-P$ and
\begin{equation}\label{equ:67}
  \begin{split}
   m(\mu)&=\sum_{j=0}^{\infty}\mu^{j}(-1)^{j}Q\left[\frac{8\pi}{(1+i)\alpha}(M_{0}+\mu M_{1}+\mu^{2}M_{2}(\mu))\right]^{j+1}Q \\
         &=\frac{8\pi}{(1+i)\alpha}QM_{0}Q-\frac{8\pi}{(1+i)\alpha}\mu Q(\frac{8\pi}{(1+i)\alpha}M_{0}^{2}-M_{1})Q+\mu^{2}m_{2}(\mu)\\
         &=\frac{8\pi}{(1+i)\alpha}[m_{0}+\mu m_{1}(\mu)].
  \end{split}
\end{equation}
Notice that $M_0=U-\frac{1}{8\pi}v(x)|x-y|v(y)$ which is the symmetric form of integral kernel of $I+(-\Delta)^{-2}V$. So that $QM_{0}Q$ is invertible in $L^{2}(\mathbb{R}^{3})$ under the assumption that zero is a regular point of $H$. Hence we get the asymptotic expansion by expanding each terms of \eqref{equ:65} in powers of $\mu$.
\end{proof}

Finally, we remark that, the invertibility of $QM_{0}Q$ is related to the spectral properties of $H$ at zero. If $QM_{0}Q$ is not invertible, denote $S:QL^{2}(\mathbb{R}^{3})\rightarrow QL^{2}(\mathbb{R}^{3})$ be the orthogonal projection on $\ker QM_{0}Q$, since $QM_{0}Q$ is self-adjoint, we have $\dim S<\infty$. Then $QM_{0}Q+S$ is invertible, and by applying lemma \ref{inverse formula} to $m_{0}+\mu m_{1}(\mu)$ we get the inverse of $m(\mu)^{-1}$. In order to get the expansion of $m(\mu)^{-1}$, by the inverse formula \eqref{equ:62} we need to get the expansion of $\widetilde{M}_{1}(\mu)$ which can be calculated from  the inverse of $m(\mu)^{-1}$. This is an iterative process but it will stops in finite steps. Since we assume that $V(x)=O(|x|^{-\beta})$ as $|x|\rightarrow\infty$, there exists a $\mu_0$ such that for $\iota\in(0,\mu_0)$ we have $\iota^4\in \rho(H)$. Since $H$ is selfadjoint, so $\limsup_{\iota\downarrow0}\|\iota^4(H-\iota^4)^{-1}\|\leq1$ and then from \eqref{equ:58} we have
\begin{equation}\label{finiteiteration}
  \limsup_{\iota\downarrow0}\|\iota^4M(\iota)^{-1}\|\leq\infty.
\end{equation}
Since each iteration adds to the singularity of $M(\mu)^{-1}$, after a few steps the leading term must be invertible and the process stops, due to \eqref{finiteiteration}. For Schr\"odinger operator $-\Delta+V$, Jensen and Nenciu have proved that the dimension of the null space of $QM_{0}Q$ is one in 1-dimension. The difficulty to analyze the structure of $QM_0Q$ for $H$ is caused by the dimension $d\geq3$. Even if we consider such problem for $H$ in 1-dimension, it is also difficult since $H$ is a higher order operator.
\vskip0.5cm

\noindent{\bf A2: Proof of Lemma \ref{Fredholm}}
\begin{proof}

{\bf Fredholm theorem} We need to prove that there exists an operator inverse $(1+R(H_0; z)V)^{-1}$, which is continuous in $L^{2}(\mathbb{R}^{d})$.

Step 1: We claim that the equation $(H-z)\psi=0$ for $\psi\in L^{2}(\mathbb{R}^{d})$ admits only trivial solution $\psi=0$. First, $(H-z)\psi=0$ implies $H\psi=z\psi\in L^{2}(\mathbb{R}^{d})$, hence $\psi\in\mathcal{H}^{4}(\mathbb{R}^{d})$. Therefore,
$$((H-z)\psi, \psi)=(H\psi,\psi)-z(\psi,\psi).$$
If $z\in\mathbb{C}\setminus\mathbb{R}$, then for $\psi\neq0$,
$$\mathrm{Im}((H-z)\psi, \psi)=-\mathrm{Im}z(\psi, \psi)\neq 0.$$
Since the scalar product $(H\psi,\psi)$ is real. Hence $(H-z)\psi\neq0$ for $\psi\neq0$.
It remains to consider $\mathrm{Re}z<V_{0}$ ( recall that $V_0=\min\{V(x), x\in\mathbb{R}^{d}\}$ ), then $\mathrm{Re}((H-z)\psi, \psi)\geq(V_{0}-z)(\psi,\psi)\neq0$ for $\psi\neq0$. Hence $(H-z)\psi\neq0$ for $\psi\neq0$.

Step 2: We prove that the equation $[1+R(H_0; z)V]\psi=0$ for $\psi\in L^{2}(\mathbb{R}^{d})$ admits only zero solution. Indeed, the identity $[1+R(H_0; z)V]\psi=0$ implies $(H-z)\psi=0$ by the Born decomposition formula. Hence $\psi=0$ as in step 1. Now the operator $[1+R(H_0; z)V]$ is invertible in $L^{2}(\mathbb{R}^{d})$ by Fredholm Theorem ( see e.g. \cite{KY, RS1} ).
\end{proof}

\bigskip

\noindent{\bf Acknowledgements:} X. Yao is partially supported by NSFC grants No.11371158 and 11771165. Part of this work was done while A. Soffer was a visiting professor at CCNU (Central China Normal University). A. Soffer is partially supported by NSFC grant No.11671163 and NSF grant DMS01600749. And this work was partially supported by a grant from the Simons Foundation (No.395767 to Avraham Soffer).

\end{document}